\documentclass{amsart}

\usepackage{amsfonts,amssymb,amsmath,amsthm}
\usepackage{tikz}
\usepackage{tikz-cd}

\usepackage{subcaption}
\usepackage{graphicx}
\usepackage{caption}
\usepackage{subcaption}
\usepackage{hyperref}

\newtheorem{thm}{Theorem}[section]
\newtheorem{cor}[thm]{Corollary}
\newtheorem{lem}[thm]{Lemma}
\newtheorem{prop}[thm]{Proposition}

\theoremstyle{definition}
\newtheorem{defn}[thm]{Definition}
\newtheorem{exm}[thm]{Example}

\theoremstyle{remark}
\newtheorem{remark}[thm]{Remark}

\newcommand{\lk}{\mbox{lk}}
\newcommand{\C}{\mathcal{C}}
\newcommand{\cI}{\mathcal{I}}
\newcommand{\cJ}{\mathcal{J}}

\newcommand{\cR}{\mathcal{R}}
\newcommand{\R}{\mathbb{R}}
\newcommand{\PC}{\mathcal{P}}
\newcommand{\QC}{\mathcal{Q}}

\newcommand{\wPC}{\widehat{\PC}}
\newcommand{\wQC}{\widehat{\QC}}

\newcommand{\walpha}{\widehat{\alpha}}
\newcommand{\wtau}{\widehat{\tau}}
\newcommand{\wsigma}{\widehat{\sigma}}

\newcommand{\wDelta}{\widehat{\Delta}}
\newcommand{\wgamma}{\widehat{\gamma}}

\newcommand{\A}{\mathcal{A}}

\graphicspath{{Images/}}

\title{Associated Primes of Spline Complexes}
\author{Michael DiPasquale}

\begin{document}

\begin{abstract}
The spline complex $\cR/\cJ[\Sigma]$ whose top homology is the algebra $C^\alpha(\Sigma)$ of mixed splines over the fan $\Sigma\subset\R^{n+1}$ was introduced by Schenck-Stillman in \cite{LCoho} as a variant of a complex $\cR/\cI[\Sigma]$ of Billera~\cite{Homology}.  In this paper we analyze the associated primes of homology modules of this complex.  In particular, we show that all such primes are linear.  We give two applications to computations of dimensions.  The first is a computation of the third coefficient of the Hilbert polynomial of $C^\alpha(\Sigma)$, including cases where vanishing is imposed along arbitrary codimension one faces of the boundary of $\Sigma$, generalizing the computations in~\cite{FatPoints,TSchenck08}.  The second is a description of the fourth coefficient of the Hilbert polynomial of $HP(C^\alpha(\Sigma))$ for simplicial fans $\Sigma$.  We use this to derive the result of Alfeld, Schumaker, and Whiteley on the generic dimension of $C^1$ tetrahedral splines for $d\gg 0$~\cite{ASWTet} and indicate via an example how this description may be used to give the fourth coefficient in particular nongeneric configurations.
\end{abstract}

\maketitle

\section{Introduction}
Let $\PC$ be a subdivision of a region in $\R^n$ by convex polytopes.  $C^r(\PC)$ denotes the set of piecewise polynomial functions (splines) on $\PC$ that are continuously differentiable of order $r$. Study of the spaces $C^r(\PC)$ is a fundamental topic in approximation theory and numerical analysis (see \cite{Boor}) while within the past decade geometric connections have been made between $C^0(\PC)$ and equivariant cohomology rings of toric varieties \cite{Paynes}.  Splines are currently used in a wide variety of other applications such as computer aided geometric design (CAGD)~\cite{CAGD} and isogeometric analysis~\cite{Isogeometric}.

A central problem in spline theory is to determine the dimension of (and a basis for) the vector space $C^r_d(\PC)$ of splines whose restriction to each facet of $\PC$ has degree at most $d$.  The spaces $C^r_d(\Delta)$ for simplicial complexes $\Delta$ in $\R^2$ and $\R^3$ have been well-studied using Bernstein-Bezier methods by Alfeld, Schumaker and coauthors~\cite{AS4r,AS3r,ASWTet,Tri}.  A signature result appears in \cite{AS3r}, which gives a dimension formula for $C^r_d(\PC)$ when $d \ge 3r+1$ and $\PC$ is a generic simplicial complex.  

An algebraic approach to the dimension question was pioneered by Billera in \cite{Homology} using homological and commutative algebra.  He introduces a chain complex $\cR/\cI$, whose top homology is the spline algebra.  Using a computation due to Whiteley~\cite{WhiteleyM}, he deduces the dimension of $C^1$ splines over generic triangulations $\Delta\subset\R^2$, solving a conjecture of Strang~\cite{Strang}.  Schenck-Stillman use a similar chain complex $\cR/\cJ$ in~\cite{LCoho} to compute the dimension of $C^r_d(\Delta)$, $\Delta\subset\R^2$, for $d\gg 0$.  In \cite{TSchenck08}, building on work of Rose~\cite{r1,r2} on dual graphs, this method is extended to give the dimension $C^r_d(\PC)$ of splines over a polytopal subdivision $\PC\subset\R^2$ for $d\gg 0$.

The results of this paper are as follows.  Working in the context of fans $\Sigma\subset\R^{n+1}$, we introduce the notation $\cR/\cJ[\Sigma,\Sigma']$ for the spline complex, where $\Sigma'\subset\Sigma$ is a subfan.  This is well-suited to describing the spline complexes that arise from imposing vanishing along codimension one faces of the boundary, in such a way that topological contributions are clear.  Using the notion of a lattice fan, first introduced in~\cite{LS}, we describe localizations of the spline complex $\cR/\cJ[\Sigma,\Sigma']$.  We then prove Theorem~\ref{thm:assPrimes1}, which identifies the associated primes of the homology modules of the spline complex as linear primes arising from the hyperplane arrangement of affine spans of codimension one faces, and Theorem~\ref{thm:assPrimes2}, which identifies more precisely the associated primes of minimal possible codimension (this is a slight extension of~\cite[Theorem~2.6]{Chow}).

We give two applications of these theorems to computations of dimension of the space $\dim C^\alpha(\Sigma)$.  In Section~\ref{sec:Third}, we derive the third coefficient of the Hilbert polynomial of the graded algebra $C^\alpha(\Sigma)$ of mixed splines on the polyhedral fan $\Sigma\subset\R^{n+1}$, where vanishing may be imposed along arbitrary codimension one faces of the boundary of $\Sigma$ (Corollary~\ref{cor:ThirdCoeff}).  This result draws on two papers of Schenck, together with Geramita and McDonald, where the third coefficient is computed in the simplicial mixed smoothness case and the polytopal uniform smoothness case, respectively ~\cite{FatPoints,TSchenck08,Chow}; however no boundary conditions are imposed in either of these papers.  The computation in Section~\ref{sec:Third} also clarifies certain topological contributions to the third coefficient.

In Section~\ref{sec:Fourth}, we describe the fourth coefficient of the Hilbert polynomial of the graded algebra $C^\alpha(\wDelta)$, where $\Delta\subset\R^3$ is a simplicial complex (Proposition~\ref{prop:SimplicialH3MainTerm}).  We use this to recover a result (for $d\gg 0$) of Alfeld, Schumaker and Whiteley on the dimension of $C^1_d(\wDelta)$ for generic $\Delta\subset\R^3$ ~\cite{ASWTet}.  In Example~\ref{ex:3DimMorganScot2} we illustrate how Proposition~\ref{prop:SimplicialH3MainTerm} may be used to compute the fourth coefficient in nongeneric cases.

\section{Polytopal Complexes and Fans}

In this section we introduce polytopal complexes and polyhedral fans, which are the underlying objects over which we define splines.

\begin{defn}\label{def:Polytope}
Fix a vector space $\R^n$ of dimension $n$ and a finite set $V$ of vectors.  The convex polytope determined by $V$ is the set
\[
\sigma=\mbox{conv}(V)=\{\sum\limits_{v\in V} \lambda_v v|\lambda_v\ge 0\in \R \mbox{ and } \sum_{v\in V} \lambda_v=1\}.
\]
The \textit{dimension} of $\sigma$ is the largest dimension of an affine space containing $\sigma$.
\end{defn}

\begin{defn}\label{def:PC}
A \textit{polytopal complex} $\PC\subset\R^n$ is a collection of polytopes satisfying
\begin{enumerate}
\item If $\gamma\in\PC$ then all faces of $\gamma$ are in $\PC$.
\item $\gamma_1\cap\gamma_2\in\PC$ is a face of both $\gamma_1,\gamma_2$, for all $\gamma_1,\gamma_2\in\PC$.
\end{enumerate}
A maximal face of $\PC$ under inclusion is a \textit{facet}.  The dimension of $\PC$ is the maximum dimension of a facet.
\end{defn}

The star of a face $\psi\in\PC$ is defined as
\[
\mbox{st}_\PC (\gamma):=\{\psi\in\PC|\exists\sigma\in\PC, \psi\in\sigma,\gamma\in\sigma \}.
\]
The dual graph $G(\PC)$ of an $n$-dimensional polytopal complex $\PC$ is defined by taking vertices to represent facets $\sigma\subset\PC_n$.  Two vertices in $G(\PC)$ corresponding to facets $\sigma_1,\sigma_2$ are connected by an edge iff $\sigma_1\cap\sigma_2\in\PC_{n-1}$.

We will mostly consider `homogeneous' analogues of the polytopal complexes $\PC$: polyhedral fans.  First we need the `homogeneous' analog of a polytope, which is a cone.

\begin{defn}\label{def:Cone}
Fix a vector space $\R^{n+1}$ and a finite set $V$ of nonzero points (vectors) in $\R^{n+1}$.  The convex polyhedral cone in $\R^{n+1}$ defined by $V$ is the \textit{positive hull} of $V$, namely the set
\[
\sigma=\mbox{cone}(V)=\{\sum\limits_{v\in V} \lambda_v v|\lambda_v\ge 0 \in\R \}.
\]

The \textit{dimension} of $\sigma$ is the largest dimension of an affine space containing $\sigma$.

A \textit{face} of $\sigma$ is either $\sigma$ or the intersection of $\sigma$ with a hyperplane $H$, which passes through the origin and satisfies that $\sigma$ lies on one side of $H$.  Such hyperplanes are called \textit{supporting hyperplanes} of $\sigma$.

A \textit{ray} $\rho$ is the set of nonnegative multiples of a single nonzero vector.
\end{defn}

\begin{defn}\label{def:Fan}
An \textit{polyhedral fan} $\Sigma\subset\R^{n+1}$ is a finite collection of cones (also called \textit{faces} of $\Sigma$) such that
\begin{enumerate}
\item If $\gamma\in\Sigma$, every face of $\gamma$ is in $\Sigma$
\item $\gamma_1\cap\gamma_2$ is a face of both $\gamma_1,\gamma_2$ for every $\gamma_1,\gamma_2\in\Sigma$
\end{enumerate}

A maximal face of $\Sigma$ under inclusion is a \textit{facet}.  The dimension of $\Sigma$ is the maximum dimension of a facet of $\Sigma$.
\end{defn}

Given a polytopal complex $\PC\subset\R^n$, we build a fan $\wPC$, called the \textit{cone over} $\PC$ or \textit{homegenization} of $\PC$, as follows.  Let $\R^n$ have coordinates $x_1,\ldots,x_n$ and $\R^{n+1}$ have coordinates $x_0,\ldots,x_n$.  Then set $i:\R^n\rightarrow\R^{n+1}$ to be the inclusion $i(x_1,\ldots,x_n)=(1,x_1,\ldots,x_n)$.  The cone $\wPC\subset\R^{n+1}$ over $\PC$ is the fan with cones $\mbox{cone}(i(\gamma))$ for $\gamma\in\PC$.  We can go the other direction as well.

\begin{defn}\label{def:raygen}
Given an abstract polyhedral fan $\Sigma$ and $\rho\in\Sigma_1$, the \textit{ray generator} $u_\rho$ of $\rho$ is the unit vector whose positive multiples generate the ray $\rho$.
\end{defn}

Using these ray generators we define two polytopal complexes which we will associate to a polyhedral fan.

\begin{defn}\label{def:FanComplexes}
Let $\sigma\subset\R^d$ be a cone.  Define
\begin{itemize}
\item $\lk(\sigma)=\mbox{conv}(u_\rho|\rho\in\sigma_1)$
\item $\PC(\sigma)=\mbox{conv}(\mathbf{0}\cup \lk(\sigma))$
\item $\lk(\Sigma)=\{\lk(\gamma)|\gamma\in\Sigma\}$
\item $\PC(\Sigma)=\{\PC(\gamma)|\gamma\in\Sigma\}$
\end{itemize}
\end{defn}

We chose unit vectors $u_\rho$ so the definition of above would be canonical, but this does not matter so much - we may refer to $\lk(\Sigma)$ and $\PC(\Sigma)$ as formed using positive scalar multiples of the vectors $u_\rho$.  See Example~\ref{ex:SC}.

We can identify $\lk(\Sigma)$ (topologically) as the intersection of $\Sigma$ with the unit $n$-sphere $\mathbb{S}^n\subset\R^{n+1}$, and $\PC(\Sigma)$ as the intersection of $\Sigma$ with the unit $(n+1)$-ball.  If $\Sigma=\wPC$, then $\lk(\Sigma)$ is homeomorphic to $\PC$.

\begin{defn}\label{def:PolyProperties}
Fix a polytopal complex/polyhedral fan $X\subset\R^n$ of dimension $n$.  Then $X$ is
\begin{enumerate}
\item \textit{Pure} if every facet $\sigma\in X$ has dimension $n$.
\item \textit{Non-branching} if every codimension one face $\tau\in X_{n-1}$ is contained in at most two facets.
\item \textit{Hereditary} if the dual graph $G(\mbox{st}_X(\psi))$ of the star of every face $\psi$ is connected.
\end{enumerate}
\end{defn}

In the definition of a \textit{pseudomanifold} one assumes that $X$ satisfies $(1)$ and $(2)$ and is \textit{strongly connected}, that is the dual graph $G(X)$ is connected.  This is implied by the hereditary condition, since the star of the empty face is $X$.  The hereditary condition is equivalent to requiring that both $X$ and the star of every one of its faces is a pseudomanifold.  We will always assume that $X$ is a hereditary pseudomanifold.

Given a hereditary pseudomanifold $X$, which is a polytopal complex or polyhedral fan, the boundary complex $\partial X$ is the subcomplex of $X$ consisting of all faces which are contained in a codimension one face $\tau$ so that $\tau$ is only contained in a single facet.

We will denote by $X_d,X^0_d,f_d(X),f^0_d(X)$ the set of $d$-faces of $X$, the set of interior $d$-faces of $X$, the number of $d$-faces of $X$, and the number of interior $d$-faces of $X$, respectively.

\begin{exm}\label{ex:SC}
The polytopal complex $\QC$ in Figure~\ref{fig:SC} has vertices $A=(-1,-1),$ $B=(1,-1), C=(1,1), D=(-1,1), A'=(-2,-2), B'=(2,-2), C'=(2,2), D'=(-2,2)$.  It has $5$ facets and $12$ edges.  Figure~\ref{fig:SC} shows the cone $\wQC$ over $\QC$, and Figure~\ref{fig:PFSC} shows the polytopal complex $\PC(\wQC)$.  The complex $\lk(\wQC)$ is the set of all faces of $\PC(\wQC)$ that don't contain the origin.  In Figure~\ref{fig:PFSC} this is the upper hull of the complex; note that $\lk(\wQC)$ is homeomorphic to the original complex $\QC$.

\begin{figure}[htp]
\centering
\begin{subfigure}[b]{.3\textwidth}
\includegraphics[width=\textwidth]{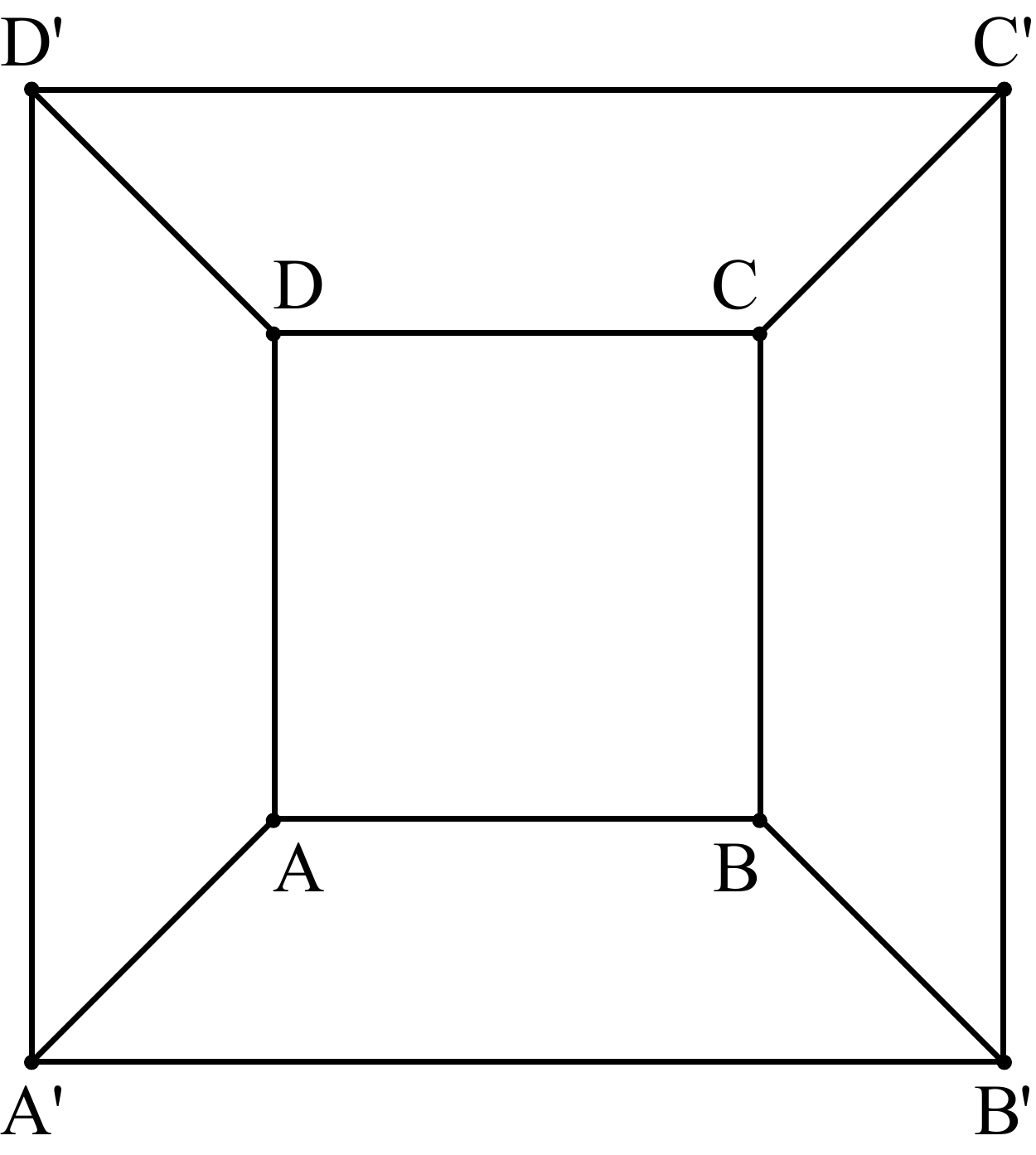}
\caption{A polytopal complex $\QC$}
\label{fig:SC}
\end{subfigure}

\begin{subfigure}[b]{.4\textwidth}
\includegraphics[width=\textwidth]{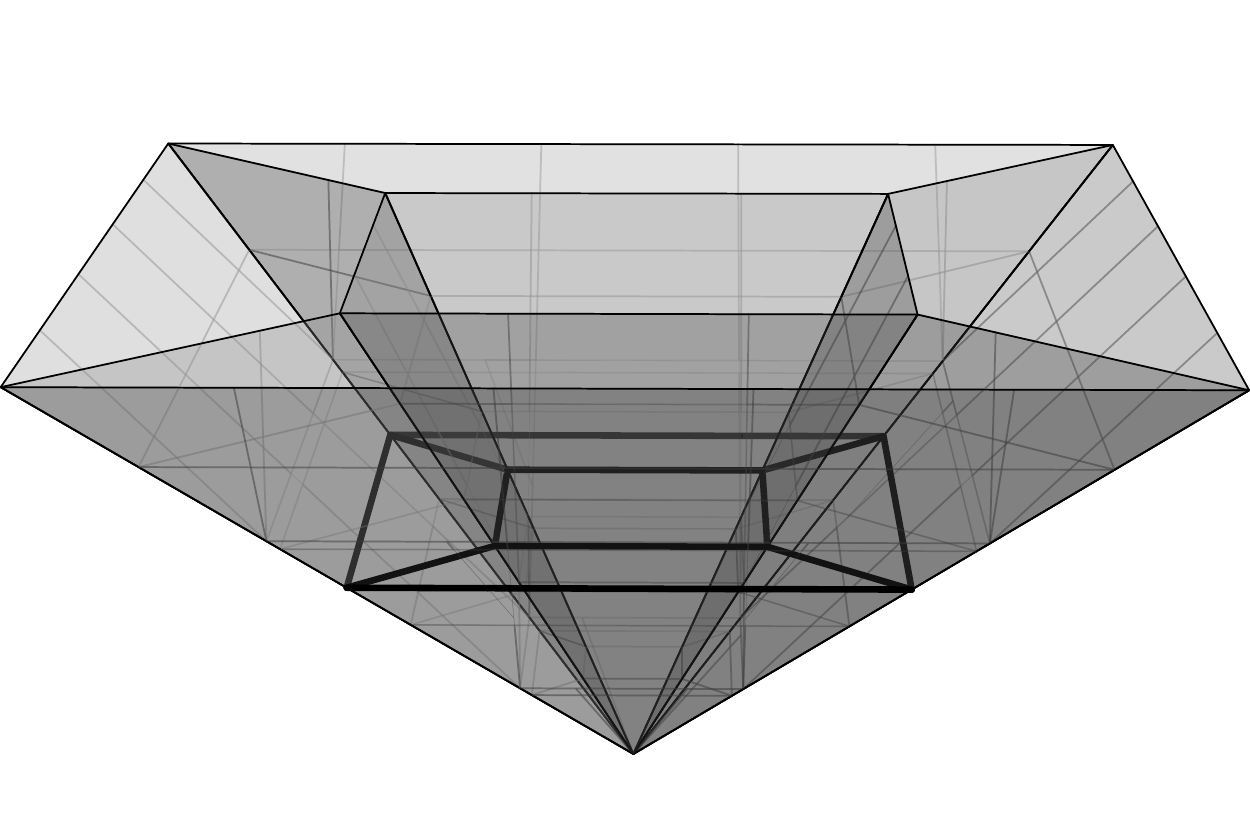}
\caption{The fan $\wQC$ over $\QC$}
\label{fig:FSC}
\end{subfigure}
\begin{subfigure}[b]{.4\textwidth}
\includegraphics[width=\textwidth]{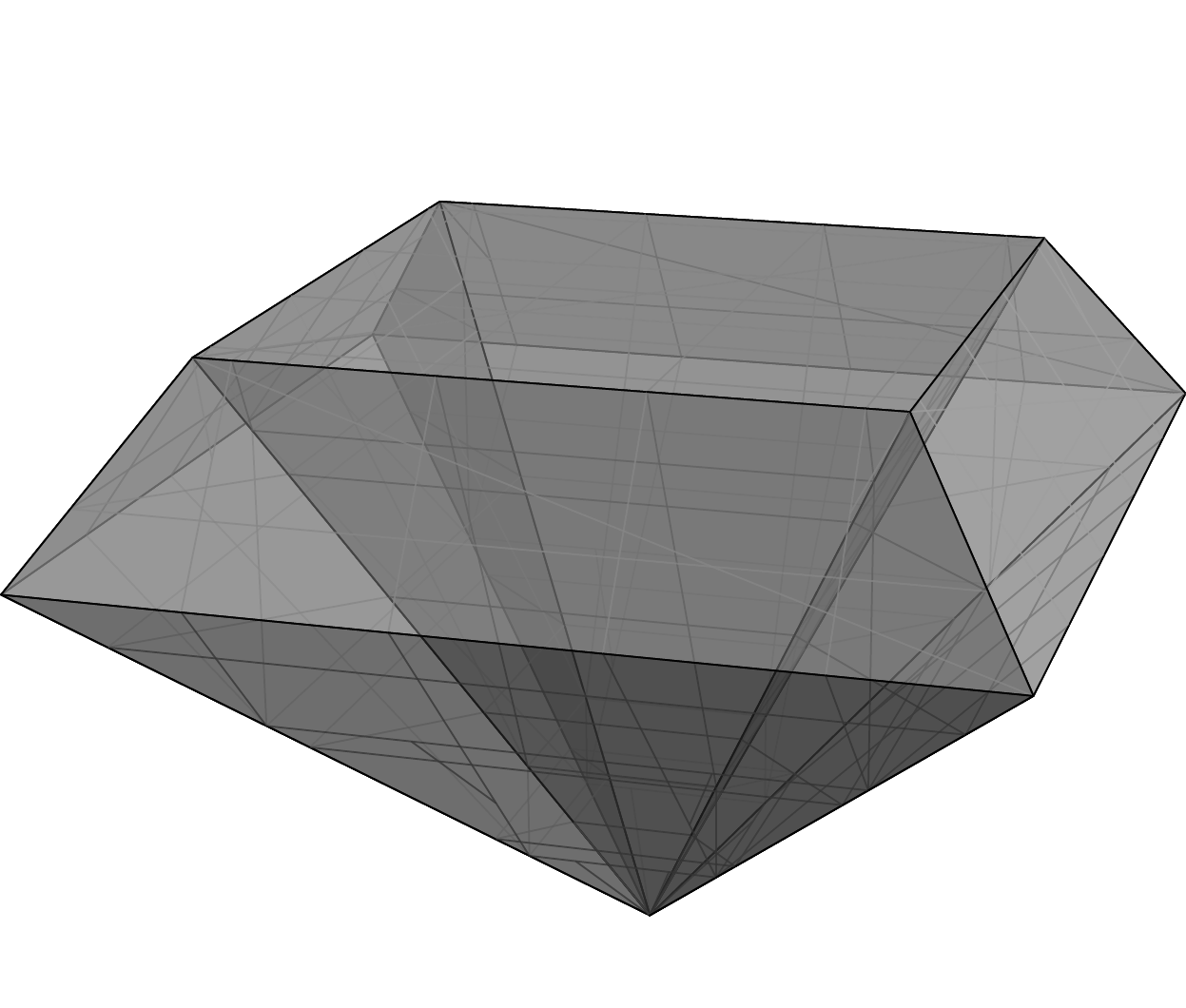}
\caption{The polytopal complex $\PC(\wQC)$}
\label{fig:PFSC}
\end{subfigure}
\caption{}
\end{figure}

\end{exm}

\section{Splines and the Spline Complex}

Given a polytopal complex or polyhedral fan $X \subset\R^n$, we assign integers $\alpha(\tau)\ge -1$, called \textit{smoothness parameters}, to every codimension one face $\tau\in X_{n-1}$ so that $\alpha(\tau)\ge 0$ for every $\tau\in X^0_n$.  We denote by $X^{-1}$ the subcomplex of $\partial X$ whose cones are contained in a codimension one face $\tau$ of $\Sigma$ so that $\alpha(\tau)=-1$.  We also set $X^{\ge 0}_d=X_d\setminus X^{-1}_d$.  The interaction of the pair $(X,X^{-1})$ will be crucial.

Let $\PC\subset\R^n$ be a polytopal complex and $\alpha$ a list of smoothness parameters.  Every codimension one face $\tau$ has a unique affine span $\mbox{aff}(\tau)\subset\R^n$.  Set $R=\R[x_1,\ldots,x_n]$ and let $l_\tau$ be a choice of generator for the ideal $I(\tau)$ of polynomials which vanish on $\tau$ (equivalently vanish on its affine span).

\begin{defn}\label{def:PolySpline}
The algebra of splines $C^\alpha(\PC)$ is the subalgebra of tuples
\[
\{F=(F_\sigma)_{\sigma\in\Sigma_{n+1}}\}\subset \bigoplus_{\sigma\in\PC_n} R
\]
satisfying
\begin{enumerate}
\item $l^{\alpha(\tau)+1}_\tau | (F_{\sigma_1}-F_{\sigma_2})$ for every pair of facets $\sigma_1,\sigma_2$ with $\sigma_1\cap\sigma_2=\tau\in\Sigma_n$.
\item $l^{\alpha(\tau)+1}_\tau| F_{\sigma}$ for every $\tau\in\sigma\cap\partial\PC$, provided this is nonempty.
\end{enumerate}
\end{defn}

If $F\in R$ is a polynomial, we denote by $\mbox{deg}(F)$ the maximal degree of a monomial of $F$.  The vector space $C^\alpha_d(\PC)$ is the set of splines $\{(F_\sigma)\in C^\alpha(\PC)| \mbox{deg}(F_\sigma)\le d\}$.  We can easily extend these definitions to fans.

\begin{defn}\label{def:FanSpline}
Let $\Sigma\subset\R^{n+1}$ be a pure, $(n+1)$-dimensional, hereditary fan, and $\alpha$ a list of smoothness parameters for $\tau\in\Sigma_n$.  Define compatible smoothness parameters on $\PC(\Sigma)$ by assigning $\alpha(\tau)=-1$ on every codimension one face of $\lk(\Sigma)\subset\partial\PC(\Sigma)$.  Set $S=\R[x_0,\ldots,x_n]$.  Then we define the $S$-algebra of mixed splines on $\Sigma$ by
\[
C^\alpha(\Sigma)=C^\alpha(\PC(\Sigma)).
\]
\end{defn}

Recall that the polynomial ring $S$ is naturally graded by degree, where $S_j$ is the vector space of polynomials of degree $j$, and that an $S$-module $M$ is (nonnegatively) graded if $M=\bigoplus_{i\ge 0} M_i$ where each $M_i$ is an $\R$-vector space and the multiplication map satisfies
\[
S_j\times M_i \rightarrow M_{i+j}.
\]
The algebra $C^\alpha(\Sigma)$ is graded by $C^\alpha(\Sigma)_d=\{(F_\sigma)|\mbox{deg} F_\sigma=d\}$.  This is a consequence of the fact that the generators $l_\tau$ of $I(\tau)$ are linear forms for every $\tau\in\Sigma_n$.

If $\PC\subset\R^n$ is any polytopal complex, we can assign smoothness parameters to $\wPC$ in the natural way: $\walpha(\wtau)=\alpha(\tau)$ for every codimension one face $\tau$ of $\PC$.  With this assignment of smoothness parameters, we have the following result of Billera-Rose.

\begin{prop}\label{prop:Homogenization}~\cite[Theorem~2.6]{DimSeries}
$C^\alpha_d(\PC)\cong C^\alpha(\wPC)_d$.
\end{prop}

The following lemma, also due to Billera-Rose, provides a useful tool for computing $C^\alpha(\PC)$.

\begin{lem}\label{lem:SplineMatrix}
$C^\alpha(\PC)$ is (isomorphic to) the kernel of the map
\[
\phi:S^{f_n(\PC)} \oplus \left(\bigoplus\limits_{\tau\in\PC_{n-1} } S(-\alpha(\tau)-1)\right) \rightarrow S^{f_n(\PC)},
\]
where $\phi$ is the matrix
\[
\begin{pmatrix}
& \vline & l^{\alpha(\tau_1)+1}_{\tau_1} & & \\
\delta_n & \vline  & & \ddots & \\
& \vline & & & l^{\alpha(\tau_k)+1}_{\tau_k}
\end{pmatrix},
 \]
$k=|\PC^{\ge 0}_{n-1}|$, $C=\textup{coker } \phi$ and the matrix $\delta_n$ is the top dimensional cellular boundary map of $\PC$ relative to $\PC^{-1}$.
\end{lem}
\begin{proof}
This is an expression of the divisibility conditions in Definition~\ref{def:PolySpline} in the form of a matrix.
\end{proof}

At this point we switch to exclusively using fans.  One could equivalently use central polytopal complexes instead; we use fans to emphasize that no conditions are imposed on faces which do not contain the origin.

Following Billera in ~\cite{Homology}, we extend the top dimensional boundary map in~\ref{lem:SplineMatrix} to a complex taking into account information of all relevant lower dimensional faces.  It will be useful to do this for an arbitrary pair of fans $(\Sigma,\Sigma')$, where $\Sigma'\subset\Sigma$ is a subfan.

\begin{defn}\label{def:TopComplex}
Let $\Sigma\subset\R^{n+1}$ be a fan with smoothness parameters $\alpha$, $\Sigma'\subset\Sigma$ a subfan, and set $S=\R[x_0,\ldots,x_n]$.  Define the complex $\cR[\Sigma,\Sigma']$ with the following modules in homological degree $i$ for $i=0,\ldots,n+1$.
\[
\begin{array}{rl}
\cR[\Sigma,\Sigma']_i= & \bigoplus\limits_{\gamma\in(\Sigma_i\setminus\Sigma'_i)} S\\
=& \bigoplus\limits_{\gamma\in(\PC(\Sigma)_i\setminus(\PC(\Sigma')_i\cup \lk(\Sigma)_i)} S,
\end{array}
\]
where the differential $\delta_i: \cR[\Sigma,\Sigma']_i\rightarrow \cR[\Sigma,\Sigma']_{i-1}$ is the cellular differential of the relative chain complex of the pair $(\PC(\Sigma),\lk(\Sigma)\cup \PC(\Sigma'))$ with coefficients in $S$.
\end{defn}

Given a fan $\Sigma$ with smoothness parameters $\alpha$, we associate ideals to its faces as follows.  For a codimension one face $\tau\in\Sigma^{\ge 0}_n$, set  $J(\tau)=\langle l^{\alpha(\tau)+1}_\tau \rangle$.  For any non-facet $\gamma\in\Sigma$,
\[
J(\gamma):=\sum\limits_{\gamma\in\tau\in\Sigma^{\ge 0}_n} J(\tau);
\]
if $\sigma\in\Sigma_{n+1}$, set
\[
J(\sigma):=0.
\]

\begin{defn}\label{def:QEcomplex}
Let $\Sigma'\subset\Sigma$ be a subfan of an $(n+1)$-dimensional fan $\Sigma\subset\R^{n+1}$ with smoothness parameters $\alpha$, and set $S=\R[x_0,\ldots,x_n]$.  Define complexes $\cJ[\Sigma,\Sigma'],\cR/\cJ[\Sigma,\Sigma']$ with the following modules in homological degree $i$ for $i=0,\ldots,n+1$.
\[
\begin{array}{rl}
\cJ[\Sigma,\Sigma']_i= & \bigoplus\limits_{\gamma\in(\Sigma_i\setminus\Sigma'_i)} J(\gamma)\\
\cR/\cJ[\Sigma,\Sigma']_i= & \bigoplus\limits_{\gamma\in(\Sigma_i\setminus\Sigma'_i)} S/J(\gamma).
\end{array}
\]
The differentials of $\cJ[\Sigma,\Sigma'],\cR/\cJ[\Sigma,\Sigma']$ are obtained by restricting and quotienting the differential of $\cR[\Sigma,\Sigma']$.
\end{defn}

\begin{lem}\label{lem:SplinesTop}
Let $\Sigma\subset\R^{n+1}$ be a pure, hereditary, $(n+1)$-dimensional fan with smoothness parameters $\alpha$, and let $\cR/\cJ[\Sigma,\Sigma^{-1}]$ be as in Definition~\ref{def:QEcomplex}.  Then
\[
H_{n+1}(\cR/\cJ[\Sigma,\Sigma^{-1}])=C^\alpha(\Sigma).
\]
\end{lem}
\begin{proof}
This is equivalent to the statement
\[
C^\alpha(\Sigma)=\mbox{ker}(S^{f_{n+1}(\Sigma)}\xrightarrow{\bar{\delta}_{n+1}} \bigoplus_{\tau\in\Sigma^{\ge 0}_n} S/J(\tau) ),
\]
where $\delta_{n+1}:S^{f_{n+1}(\Sigma)}\rightarrow \bigoplus_{\tau\in\Sigma^{\ge 0}_n} S$ is the top dimensional cellular boundary map of $\PC(\Sigma)$ relative to $\PC(\Sigma^{-1})\cup\lk(\Sigma)$.  This follows from Lemma~\ref{lem:SplineMatrix}; or it can be seen directly since it is another way to state the divisibility conditions from Definition~\ref{def:PolySpline}.  Explicitly, a tuple $(F_\sigma)_{\sigma\in\Sigma_{n+1}}$ is sent by $\bar{\delta}$ to the tuple $(F_{\sigma_1}- F_{\sigma_2})_\tau$ mod $J(\tau)$, where $\tau\in\Sigma_n$ is the codimension one face along which $\sigma_1,\sigma_2$ intersect.  This is $0$ iff $F_{\sigma_1}-F_{\sigma_2}\in J(\tau)$, i.e. iff $l^{\alpha(\tau)+1}_\tau | F_{\sigma_1}-F_{\sigma_2}$.  If $\tau\in\partial\PC$, then there is only one facet, say $\sigma$, containing $\tau$ and $F_\sigma$ $\equiv$  $0$ mod $J(\tau)$ iff $l^{\alpha(\tau)+1}_\tau|F_\sigma$.
\end{proof}

\begin{remark}
There is a tautological short exact sequence of complexes
\[
0\rightarrow \cJ[\Sigma,\Sigma'] \rightarrow \cR[\Sigma,\Sigma'] \rightarrow \cR/\cJ[\Sigma,\Sigma']\rightarrow 0
\]
We will frequently use this exact sequence of complexes in proofs.
\end{remark}

\begin{remark}
The most well studied case is when $\alpha(\tau)=r$ for every interior codimension one face $\tau\in\Sigma$ and $\alpha(\tau)=-1$ for every codimension one face in $\partial\Sigma$.  In this case $\Sigma^{-1}=\partial\Sigma$ and $C^\alpha(\Sigma)$ is denoted by $C^r(\Sigma)$.
\end{remark}

We spend the rest of the section investigating the homology of $\cR[\Sigma,\Sigma']$.  The complex $\cR[\Sigma,\Sigma']$ is defined so that
\[
H_i(\cR[\Sigma,\Sigma'])=H_i(\PC(\Sigma),\PC(\Sigma')\cup\lk(\Sigma);S)
\]
where the homology group on the right is the cellular homology of $\PC(\Sigma)$ relative to $\PC(\Sigma')\cup\lk(\Sigma)$ with coefficients in $S$.  This agrees with the so-called \textit{Borel-Moore} homology of the fan $\Sigma$ relative to the subfan $\Sigma'$.  The homology of this complex is described in more detail in the following proposition.

\begin{prop}\label{prop:LowHom}
Let $\Sigma$ be an $(n+1)$-dimensional abstract fan, with $n\ge 1$, $\Sigma'\neq\mathbf{0}\subset\Sigma$ a subfan, possibly empty, and $\cR[\Sigma,\Sigma']$ as defined above.
\[
H_i(\cR[\Sigma,\Sigma'])\cong\left\lbrace
\begin{array}{rl}
0 & \mbox{if } i=0,1\\
\widetilde{H}_{i-1}(\PC(\Sigma')\cup\lk(\Sigma);S)\cong H_{i-1}(\lk(\Sigma),\lk(\Sigma');S) & \mbox{if } i\ge 2.
\end{array}
\right.
\]
\end{prop}

\begin{remark}
If $\Sigma=\wPC$, then $\lk(\Sigma)$ is homeomorphic to $\PC$ and $\lk(\Sigma^{-1})$ is homeomorphic to $\PC^{-1}$.
\end{remark}

\begin{proof}
We use the identification $H_i(\cR[\Sigma,\Sigma'])\cong H_i(\PC(\Sigma),\PC(\Sigma')\cup\lk(\Sigma);S)$. Consider the long exact sequence of the pair in singular homology corresponding to the inclusion $\PC(\Sigma')\cup\lk(\Sigma)\hookrightarrow \PC(\Sigma)$, with coefficients in $S$:
\[
\cdots\rightarrow H_i(\PC(\Sigma))\rightarrow H_i(\PC(\Sigma),\PC(\Sigma')\cup\lk(\Sigma)) \rightarrow H_{i-1}(\PC(\Sigma')\cup\lk(\Sigma))\rightarrow\cdots
\]
$\PC(\Sigma)$ is contractible, so
\[
H_i(\PC(\Sigma))=\left\lbrace
\begin{array}{rl}
S & \mbox{if } i=0\\
0 & \mbox{otherwise},
\end{array}
\right.
\]
The map $H_0(\PC(\Sigma')\cup\lk(\Sigma))\rightarrow H_0(\PC(\Sigma))=S$ is surjective, hence $H_0(\PC(\Sigma),\PC(\Sigma')\cup\lk(\Sigma))=0$ and we have a short exact sequence
\[
0 \rightarrow H_1(\PC(\Sigma),\PC(\Sigma')\cup\lk(\Sigma)) \rightarrow H_0(\PC(\Sigma')\cup\lk(\Sigma))\rightarrow S \rightarrow 0,
\]
Hence $H_1(\PC(\Sigma),\PC(\Sigma')\cup\lk(\Sigma))=0$ if $\PC(\Sigma')\cup\lk(\Sigma)$ is connected.  Since $\Sigma$ is hereditary, $\lk(\Sigma)$ is connected.  So if $\Sigma'=\emptyset$, $\PC(\Sigma')\cup\lk(\Sigma)=\lk(\Sigma)$ is connected.  If $\Sigma'\neq\emptyset$, then $\PC(\Sigma')$ is connected since $\mathbf{0}$ is contained in every face.  Furthermore $\PC(\Sigma')\cap\lk(\Sigma)\neq\emptyset$ since every face of $\Sigma$ other than $\mathbf{0}$ intersects nontrivially with $\lk(\Sigma)$, and we assumed $\Sigma'\neq\mathbf{0}$.  So $\PC(\Sigma')\cup\lk(\Sigma)$ is connected and the conclusion follows.

The isomorphisms
\[
H_i(\PC(\Sigma),\PC(\Sigma')\cup\lk(\Sigma))\cong H_{i-1}(\PC(\Sigma')\cup\lk(\Sigma))
\]
for $i\ge 2$ are immediate from the long exact sequence of the pair.  Finally, the isomorphism
\[
\widetilde{H}_j(\PC(\Sigma')\cup\lk(\Sigma))\cong H_j(\lk(\Sigma),\lk(\Sigma'))
\]
is a consequence of excision and the long exact sequence of the pair $(\PC(\Sigma')\cup\lk(\Sigma),\PC(\Sigma'))$.  The key observation is that $\PC(\Sigma')\cup\lk(\Sigma)$ is the \textit{mapping cone} of the inclusion $\lk(\Sigma')\hookrightarrow\lk(\Sigma)$.  That is, topologically, $\PC(\Sigma')\cup\lk(\Sigma)$ may be identified with the space
\[
\lk(\Sigma)\cup(\lk(\Sigma')\times I)/\sim,
\]
where $I=[0,1]$ is the unit interval, all points of the form $(x,0)$ are identified as a single point, and $(x,1)$ is identified with the image of $x$ in $\lk(\Sigma)$.  A more detailed discussion may be found in \cite[p. 125]{Hatcher}.
\end{proof}

\begin{exm}\label{ex:TopSC}
Let $\Sigma=\wQC$ as in Figure~\ref{fig:FSC}, with uniform smoothness parameters $\alpha(\tau)=r$ on interior codimension one faces and $\alpha(\tau)=-1$ on boundary codimension one faces.  Then $\Sigma^{-1}=\partial\Sigma$.  The complex $\cR[\Sigma,\Sigma^{-1}]$ is nonzero in homological degrees $1,2,$ and $3$.  It has the form
\[
S^5 \rightarrow S^8 \rightarrow S^4 \rightarrow 0,
\]
where $S=\R[x,y,z]$ is the polynomial ring in three variables.  By definition, $H_*(\cR[\Sigma,\partial\Sigma])$ computes the homology of the complex $\PC(\Sigma)$ relative to $\PC(\partial\Sigma)\cup\lk(\Sigma)=\partial\PC(\Sigma)$ with coefficients in $S=\R[x,y,z]$.  From Figure~\ref{fig:PFSC} it is clear that this is equivalent to computing $H_*(D^3,\mathbb{S}^2;S)$, the homology of a $3$-disk relative to its boundary with coefficients in $S$.  By excision, $H_*(D^3,\mathbb{S}^2;S) \cong \widetilde{H}_*(D^3/\mathbb{S}^2;S) =\widetilde{H}_*(\mathbb{S}^3;S)$.  Hence $H_i(\cR[\Sigma,\partial\Sigma])=0$ except when $i=3$, when $H_3(\cR[\Sigma,\partial\Sigma])=S$.

Equivalently, using Proposition~\ref{prop:LowHom}, we see that $H_0(\cR[\Sigma,\partial\Sigma])=0$ for $i=0,1$.  We have $\lk(\Sigma)$ is homeomorphic to $\QC$ and $\lk(\partial\Sigma)$ is homeomorphic to $\partial\QC$.  It is clear that the homology of $\QC$ relative to its boundary gives the homology of a $2$-sphere.  Shifting the homological dimensions up, we again arrive at the homology of the complex $\cR[\Sigma,\partial\Sigma]$.
\end{exm}

\begin{exm}\label{ex:TopSCB}
Again let $\Sigma=\wQC$ be as in Figure~\ref{fig:FSC}, but suppose we impose vanishing on the entire boundary, so that $\Sigma^{-1}=\emptyset$.  Then the complex $\cR[\Sigma,\Sigma^{-1}]=\cR[\Sigma]$ has the form
\[
S^5\rightarrow S^{12} \rightarrow S^8 \rightarrow S,
\]
where the final $S$ corresponds to the cone vertex.  $H_*(\cR[\Sigma,\Sigma^{-1}])\cong H_*(\Sigma,\PC(\Sigma^{-1})\cup\lk(\Sigma);S)=H_*(\Sigma,\lk(\Sigma);S)$.  Since $\lk(\Sigma)$ is contractible, this is the same as the reduced homology of a point.  We conclude that $H_i(\cR[\Sigma])$ vanishes for all $i$.

Equivalently, using Proposition~\ref{prop:LowHom}, we see that $H_0(\cR[\Sigma,\partial\Sigma])=0$ for $i=0,1$.  We have $\lk(\Sigma)$ is homeomorphic to $\QC$.  The homology of $\QC$ gives the homology of a $2$-disk.  Hence $H_i(\cR[\Sigma])=H_{i-1}(\QC;S)=0$ for $i=2,3$.  Note that $H_1(\cR[\Sigma])=S$ while $H_0(\QC;S)=S$.
\end{exm}

\section{Lattice Fans}

In~\cite{LS} certain complexes $\PC_W\subset\PC$, called \textit{lattice complexes}, are discussed in the context of describing localization of $C^r(\PC)$.  In this section we describe how this construction carries over to the context of a pair $(\Sigma,\Sigma')$.  In the end this will yield information about localizations of the entire complex $\cR/\cJ[\Sigma,\Sigma']$.

\begin{defn}\label{def:HypArr}
Let $\Sigma\subset\R^{n+1}$ be an $(n+1)$-dimensional fan and $\Sigma'\subset\Sigma$ a subfan.
\begin{enumerate}
\item $\A(\Sigma,\Sigma')$ denotes the hyperplane arrangement $\bigcup\limits_{\tau\in\Sigma_n\setminus\Sigma'_n} \mbox{aff}(\tau)$.
\item $L_{\Sigma,\Sigma'}$ denotes the intersection lattice $L(\A(\Sigma,\Sigma'))$ of $\A(\Sigma,\Sigma')$, ordered with respect to reverse inclusion.
\item The \textit{support} of a face $\gamma\in \Sigma$, denoted $\mbox{supp}(\gamma)$, is the collection of flats $W\in L_{\Sigma,\Sigma'}$ so that $W\subseteq\mbox{aff}(\gamma)$.
\end{enumerate}
\end{defn}

\begin{defn}\label{def:LatticePair}
Let $\Sigma$ be an $(n+1)$-dimensional fan, $\Sigma'\subset\Sigma$ a subfan,  $W\in L_{\Sigma,\Sigma'}$, and $\sigma\in\Sigma_{n+1}$.

Define $\Sigma_W^c$ to be the subfan of $\Sigma$ consisting of all faces whose affine span does not contain $W$ (equivalently whose support does not contain $W$).

Define $\Sigma_{W,\sigma}\subset\Sigma$ to be the subfan with faces $\gamma\subset\sigma'\in\Sigma_{n+1}$ so that there is a chain $\sigma=\sigma_0,\sigma_1,\ldots,\sigma_k=\sigma'$ with $\sigma_{i-1}\cap \sigma_i=\tau_i\in\Sigma_n$ and $W\subset\mbox{aff}(\tau_i)$ for $i=1,\ldots,k$.  We call $\Sigma_{W,\sigma}$ a \textit{lattice fan}.

Define an equivalence relation $\sim_W$ on $\Sigma_{n+1}$ by $\sigma\sim_W \sigma'$ if $\sigma'\in\Sigma_{W,\sigma}$.  
\end{defn}

\begin{defn}\label{def:LatticeNotation}
We will use the following notation
\begin{itemize}
\item $[\sigma]_W:$ equivalence class of $\sigma$ under $\sim_W$
\item $\Gamma_W:$ a set of distinct representatives $\sigma\in\Sigma_{n+1}$ of the equivalence classes $[\sigma]_W$
\item $\Sigma_W=\bigsqcup_{\sigma\in\Gamma_W} \Sigma_{W,\sigma}$
\item $\Sigma'_{W,\sigma}=(\Sigma^c_W\cup\Sigma')\cap\Sigma_{W,\sigma}$
\item $\Sigma^{-1}_{W,\sigma}=(\Sigma^c_W\cup\Sigma^{-1})\cap\Sigma_{W,\sigma}$
\item $(\Sigma^{\ge 0}_{W,\sigma})_i= i$-faces of $\Sigma_{W,\sigma}$ not contained in $\Sigma^{-1}_{W,\sigma}$
\item $(\Sigma_W,\Sigma'_W)=\sqcup_{\sigma\in\Gamma_W} (\Sigma_{W,\sigma},\Sigma'_{W,\sigma})$
\item $\cJ[\Sigma_W,\Sigma'_W]=\bigoplus\limits_{\sigma\in\Gamma_W} \cJ[\Sigma_{W,\sigma},\Sigma'_{W,\sigma}]$
\item $\cR[\Sigma_W,\Sigma'_W]=\bigoplus\limits_{\sigma\in\Gamma_W} \cR[\Sigma_{W,\sigma},\Sigma'_{W,\sigma}]$
\item $\cR/\cJ[\Sigma_W,\Sigma'_W]=\bigoplus\limits_{\sigma\in\Gamma_W} \cR/\cJ[\Sigma_{W,\sigma},\Sigma'_{W,\sigma}]$
\end{itemize}
\end{defn}

\begin{remark}
If $\sigma$ has no codimension one face whose affine span contains $W$, then $[\sigma]_W$ consists only of $\sigma$.
\end{remark}

\begin{remark}
$\Sigma_{W,\sigma}$ is the component of the lattice complex $\Sigma_W$ containing the face $\sigma$.
\end{remark}

\begin{remark}
The equivalence relation $\sim_W$ is similar to one used by Yuzvinsky in~\cite{Yuz} but is different in some subtle ways.  See Remark~\ref{rem:Yuz} following Example~\ref{ex:CubeOct}.
\end{remark}

\begin{remark}
If $W\subset\mbox{aff}(\gamma)$, where $\gamma\in\Sigma$ is a face of $\Sigma$, then $\mbox{st}(\gamma)\subset \Sigma_{W,\sigma}$ for any facet $\sigma$ with $\gamma\in\sigma$.
\end{remark}

\begin{exm}
Let $\QC$ be the polytopal complex from Example~\ref{ex:SC} and set $\Sigma=\wQC$.  Let $W=\mbox{aff}(v)$, where $v$ is an internal ray of $\Sigma$.  Let $\sigma_1$ be any facet containing $v$ and $\sigma_2$ any facet not containing $v$.  Then $W$, $\Sigma_{W,\sigma_1}$, and $\Sigma_{W,\sigma_2}$ are shown in Figure~\ref{fig:SchlegelCubeVertexLine}.  Notice that $\Sigma_W$ consists of two nontrivial components.  Also let $V$ be the affine span of the internal codimension one face of $\Sigma_{W,\sigma_2}$.  Then $W\subset V$ ($V<W$ in $L_{\Sigma,\Sigma^{-1}}$) and $\Sigma_{W,\sigma_2}=\Sigma_{V,\sigma_2}$.  Occasionally we will want to replace $W$ by the minimal flat $V$ satisfying $\Sigma_{V,\sigma_2}=\Sigma_{W,\sigma_2}$.

\begin{figure}
\begin{subfigure}[b]{.3\textwidth}
\includegraphics[width=\textwidth]{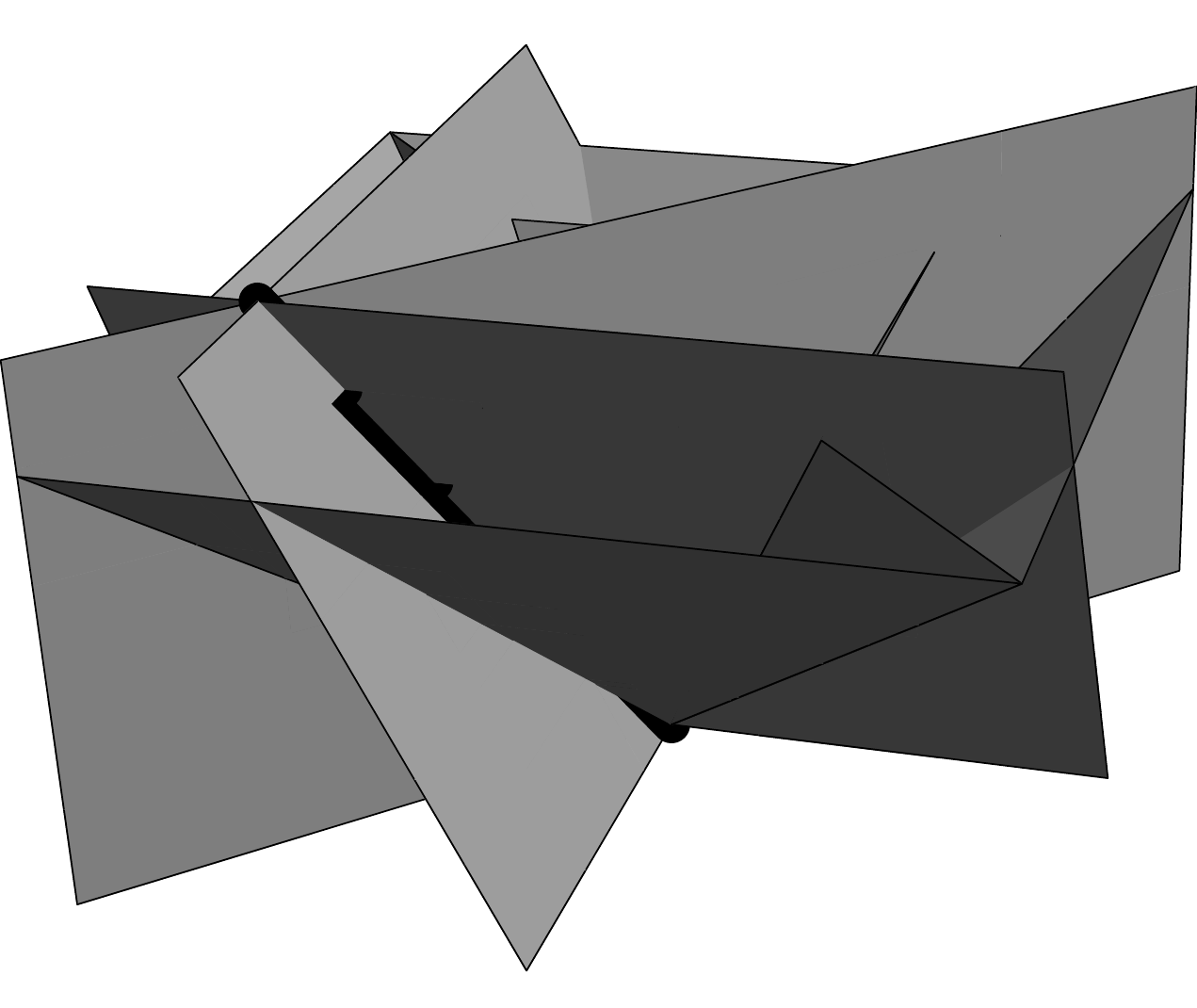}
\caption{$W=\mbox{aff}(\widehat{v})$}
\end{subfigure}
\begin{subfigure}[b]{.3\textwidth}
\includegraphics[width=\textwidth]{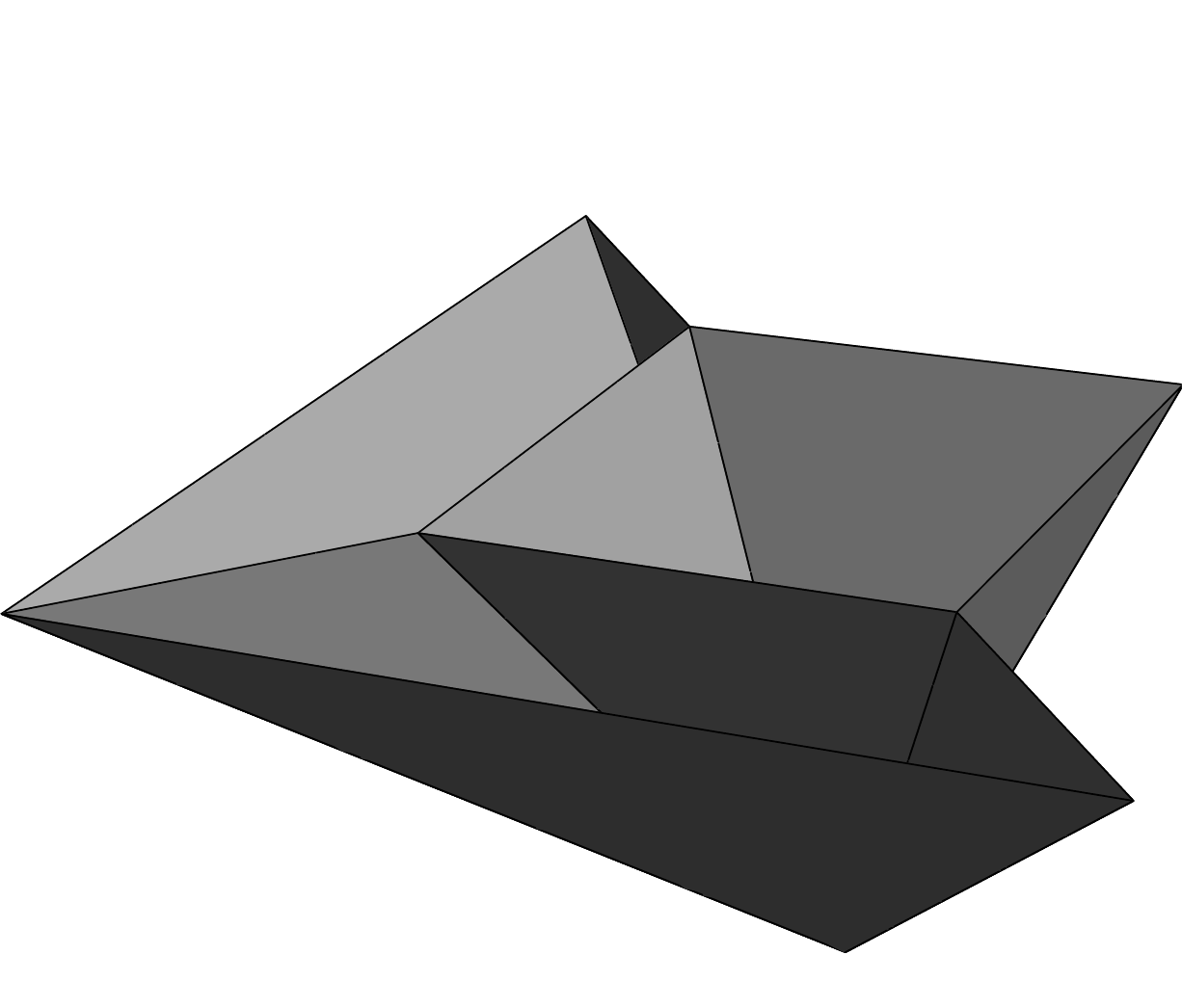}
\caption{$\Sigma_{W,\sigma_1}$}
\end{subfigure}
\begin{subfigure}[b]{.25\textwidth}
\includegraphics[width=\textwidth]{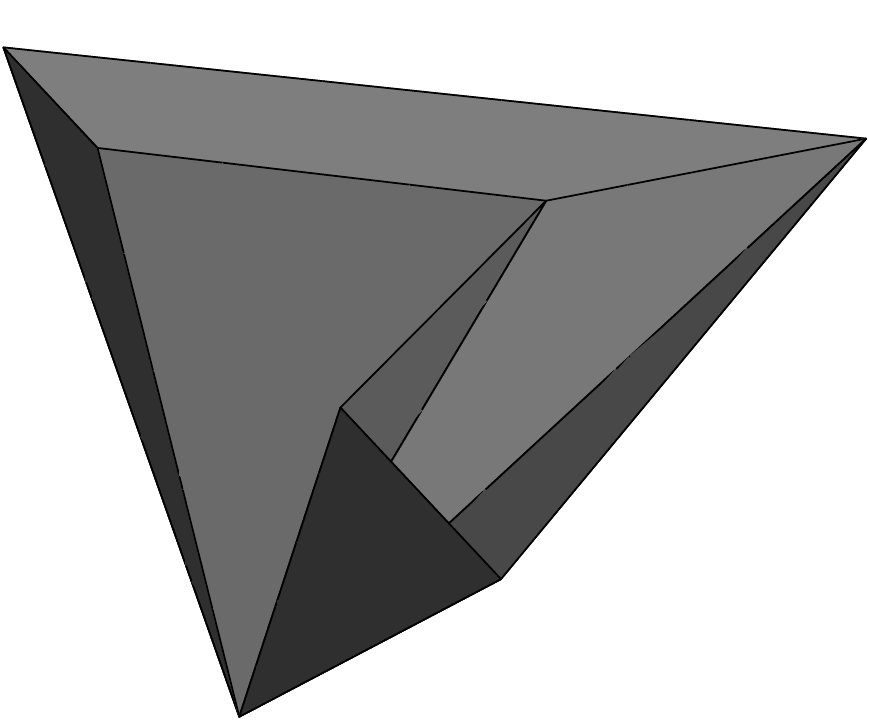}
\caption{$\Sigma_{W,\sigma_2}$}
\end{subfigure}
\caption{}\label{fig:SchlegelCubeVertexLine}
\end{figure}

\end{exm}

In the simplicial case $\Sigma_{W,\sigma}$ is always the star of a face.

\begin{lem}\label{lem:FanStar}
Let $\Sigma\subset\R^{n+1}$ be a simplicial fan, $\Sigma'\subset\Sigma$ a subfan.  Then $\Sigma_{W,\sigma}=\mbox{st}_\Sigma(\gamma)$ for some face $\gamma$ with $W\subset\mbox{aff}(\gamma)$.
\end{lem}
\begin{proof}
This is the content of~\cite[Lemma~2.7]{LS}.
\end{proof}

With these notations in place the following lemma is almost immediate.

\begin{lem}\label{lem:Localize}
Let $\Sigma\subset\R^{n+1}$ be an $(n+1)$-dimensional fan, $\Sigma'\subset\Sigma$ a subfan, and $P\in \mbox{spec}(S)$.  Set $W=\max_{V\in L_{\Sigma,\Sigma'}}\{I(V)|I(V)\subset P\}$.  Then
\[
\begin{array}{rl}
\cR/\cJ[\Sigma,\Sigma']_P= & \cR/\cJ[\Sigma,\Sigma^c_W\cup \Sigma']_P\\
=& \cR/\cJ[\Sigma_W,\Sigma'_W]_P,
\end{array}
\]
\end{lem}
\begin{proof}
Each module in the chain complex $\cR/\cJ[\Sigma,\Sigma']$ is a direct sum of modules of the form $S/J(\tau)$ for $\tau\in\Sigma\setminus\Sigma'$.  Under localization, all of these go to zero unless $J(\tau)\subset P$, in other words, $J(\tau)\subset I(W)$, hence $W\subset\mbox{aff}(\tau)$ and $\tau\not\in\Sigma'$.  This proves the first equality.  The second equality simply rewrites the complex $\cR/\cJ[\Sigma,\Sigma^c_W\cup\Sigma']$ as a direct sum across connected components of $\Sigma_W$, using the observation that $\Sigma\setminus(\Sigma^c_W\cup\Sigma')=\sqcup_{\sigma\in\Gamma_W} \Sigma_{W,\sigma}\setminus \Sigma'_{W,\sigma}$.
\end{proof}

We do an extended computation to show how the complexes $\Sigma_{W,\sigma}$ and their topology can be used to compute certain localizations of the complex $\cR/\cJ[\Sigma,\Sigma^{-1}]$.  This is a rather long process to compute a localization which is fairly quick to do by hand, however it illustrates the general procedure.

\begin{exm}\label{ex:SCCentralLoc}
Let $\QC$ be the complex from~\ref{ex:SC} and $\Sigma=\wQC$.  We show in Figure~\ref{fig:SCCentralLine} the affine spans of $4$ interior codimension one faces which intersect along the $z$-axis, which we denote by $W$.  In Figure~\ref{fig:SCCentralLinePF} we show the polytopal complex $\PC(\Sigma_{W,\sigma})$, where $\sigma$ is any of the four facets with a codimension one face $\gamma$ with $W\in\mbox{supp}(\gamma)$.  Note that the central facet is removed.  In this case $\Sigma_W=\Sigma_{W,\sigma}$.
\begin{figure}[htp]
\centering
\begin{subfigure}[b]{.45\textwidth}
\centering
\includegraphics[width=.8\textwidth]{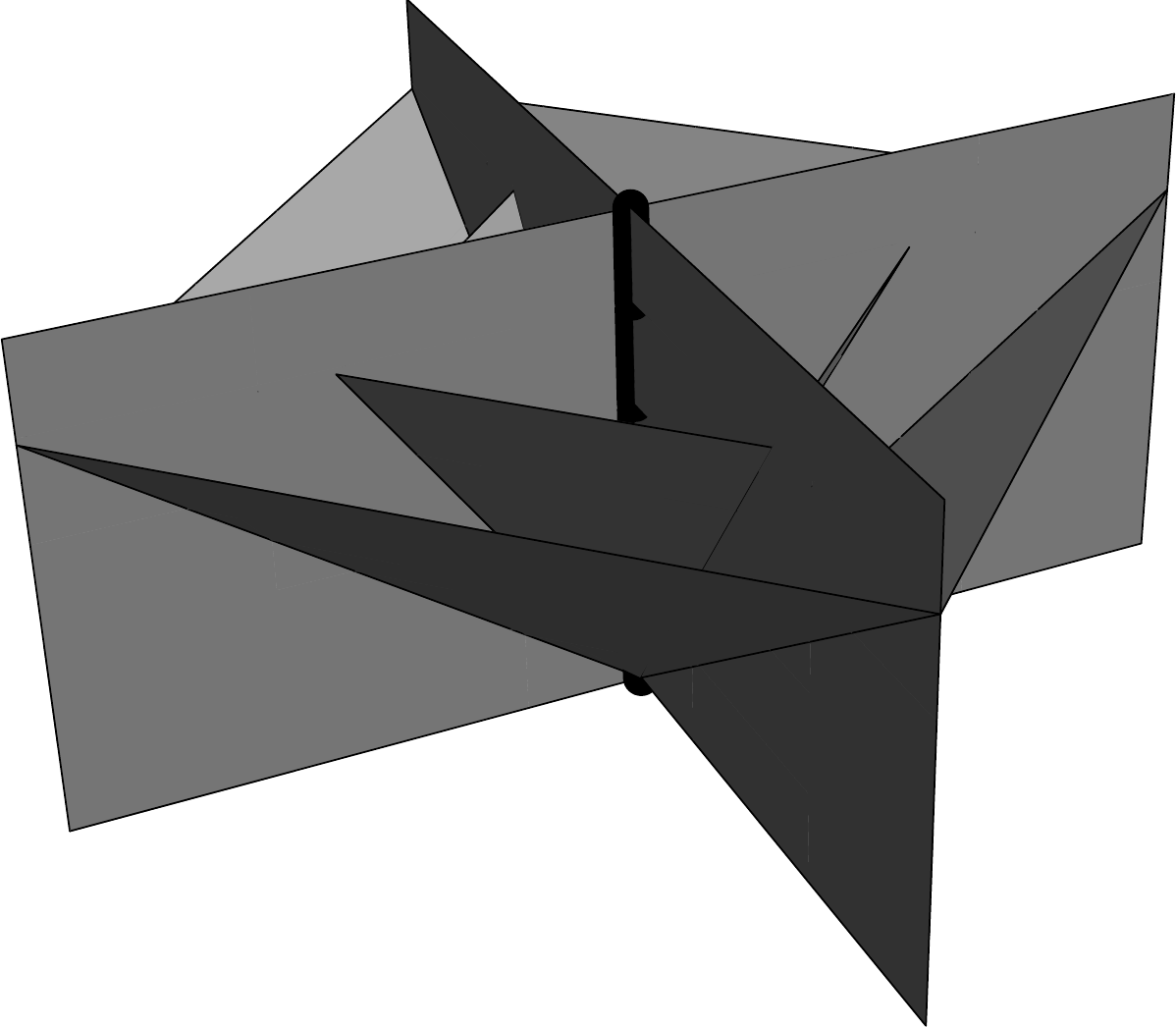}
\caption{$W$}\label{fig:SCCentralLine}
\end{subfigure}
\begin{subfigure}[b]{.45\textwidth}
\centering
\includegraphics[width=.8\textwidth]{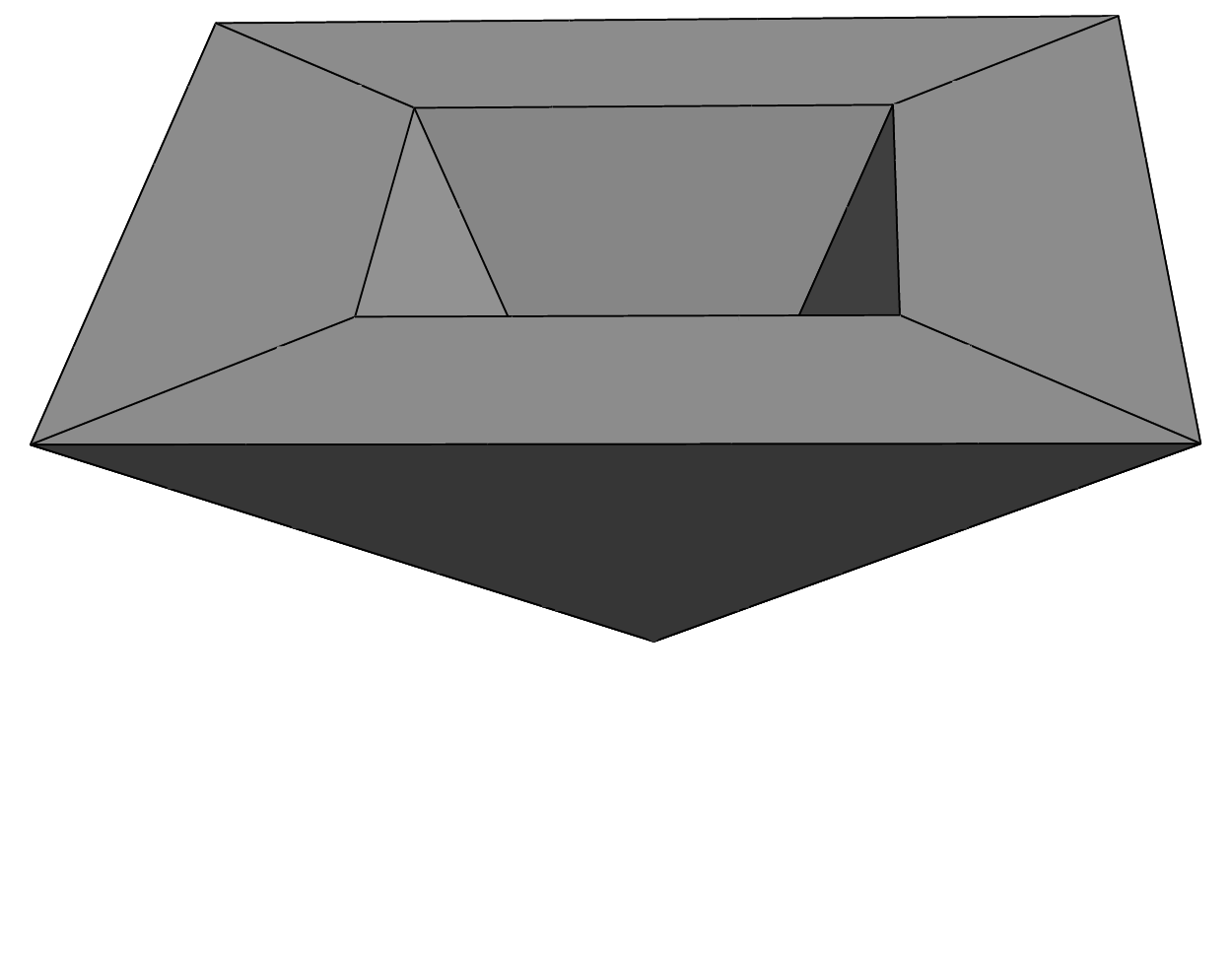}
\caption{$\PC(\Sigma_{W,\sigma})$}\label{fig:SCCentralLinePF}
\end{subfigure}
\caption{}
\end{figure}

Since the only codimension one facets whose affine spans contain $W$ are interior, $\Sigma^{-1}_{W,\sigma}=\partial\Sigma_{W,\sigma}$ regardless of what smoothness parameters we assign.

The complex $\cR[\Sigma_{W,\sigma},\Sigma^{-1}_{W,\sigma}]=\cR[\Sigma_{W,\sigma},\partial\Sigma_{W,\sigma}]$ is concentrated in homological degrees $3$ and $2$ and has the form
\[
S^4\rightarrow S^4 \rightarrow 0 \rightarrow 0.
\]
$H_*(\cR[\Sigma_{W,\sigma},\partial\Sigma_{W,\sigma}])$ computes the homology of $\PC(\Sigma_{W,\sigma})$ relative to $\partial\PC(\Sigma_{W,\sigma})$, so
\[
H_i(\cR[\Sigma_{W,\sigma},\partial\Sigma_{W,\sigma}])=H_{i-1}(\lk(\Sigma_{W,\sigma}),\partial\lk(\Sigma_{W,\sigma}))
\]
for $i\ge 2$ by Proposition~\ref{prop:LowHom}.

From Figure~\ref{fig:SCCentralPoint}, which displays $\Sigma_{W,\sigma})$ and its boundary (up to homeomorphism)  we see that the homology on the left-hand side is the same as the homology of a $2$-sphere with $2$ points identified.
\begin{figure}
\includegraphics[width=.3\textwidth]{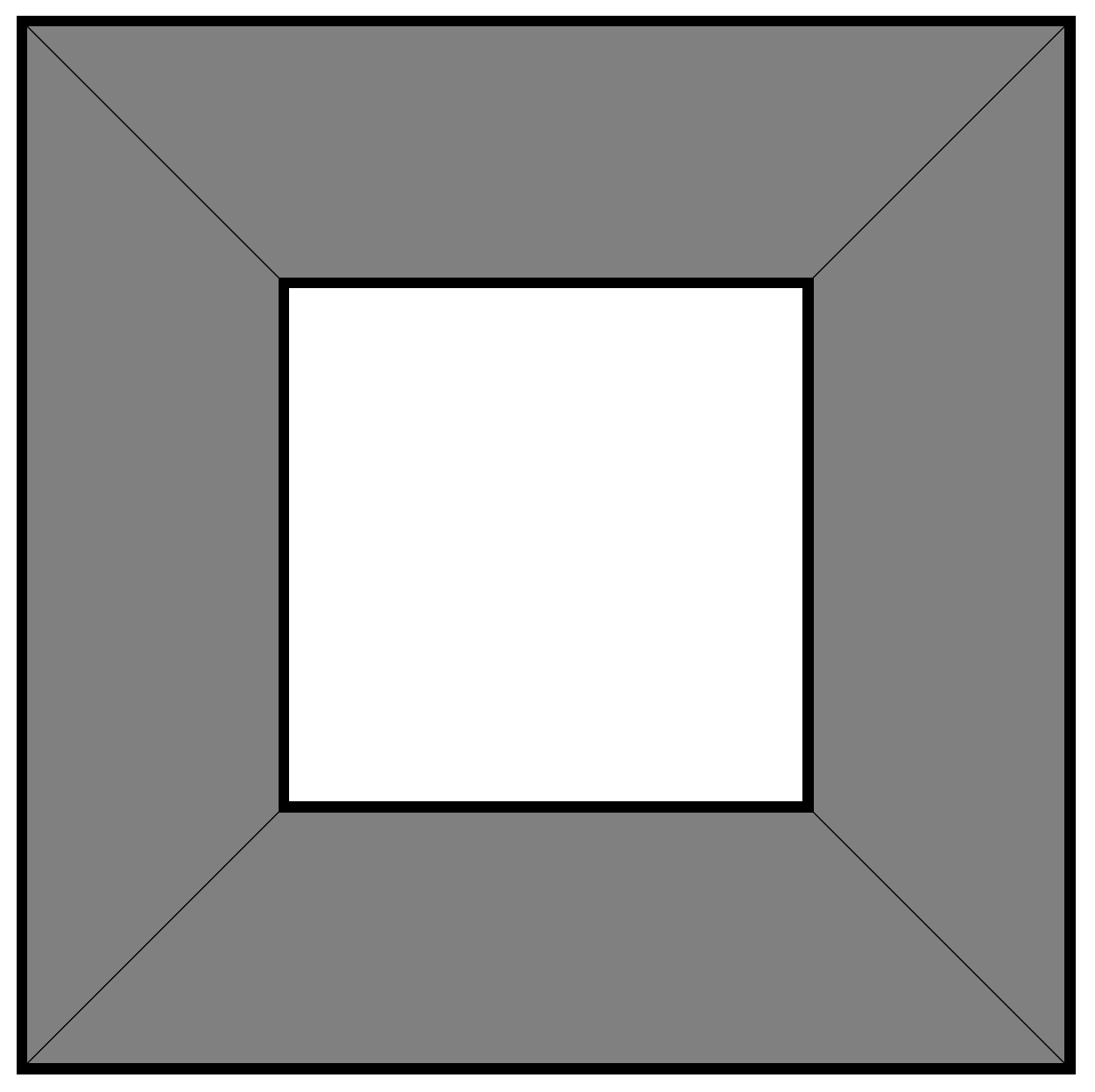}
\caption{$lk(\Sigma_{W,\sigma})$}\label{fig:SCCentralPoint}
\end{figure}
Thus $H_3(\cR[\Sigma_{W,\sigma},\partial\Sigma_{W,\sigma})=H_2(\cR[\Sigma_{W,\sigma},\partial\Sigma_{W,\sigma})=S$ while the lower two homologies vanish.

Now, via the tail end of the long exact sequence
\[
0\rightarrow \cJ[\Sigma_{W,\sigma},\partial\Sigma_{W,\sigma}]\rightarrow \cR[\Sigma_{W,\sigma},\partial\Sigma_{W,\sigma}]\rightarrow \cR/\cJ[\Sigma_{W,\sigma},\partial\Sigma_{W,\sigma}] \rightarrow 0,
\]
we obtain that
\[
H_2(\cR/\cJ[\Sigma_W,\sigma])=S/(\sum_{i=1}^4 J(\tau_i)),
\]
where $\tau_1,\ldots,\tau_4$ are the four interior codimension one facets of $\Sigma_{W,\sigma}$.  By Lemma ~\ref{lem:Localize}, we have shown that
\[
H_2(\cR/\cJ[\Sigma,\Sigma^{-1}])_{I(W)}=(S/(\sum_{i=1}^4 J(\tau_i)))_{I(W)}.
\]
In fact we could replace $I(W)$ by any prime $P$ containing $I(W)$, as long is there is no other flat $V\in L_{\Sigma,\Sigma^{-1}}$, with $I(W)\subsetneq I(V)$, so that $I(V)\subset P$.
\end{exm}

From Lemma~\ref{lem:Localize} and Example~\ref{ex:SCCentralLoc} we see that it is useful to understand the homology of the complexes $\cR[\Sigma_{W,\sigma},\Sigma'_{W,\sigma}]$.  To this end we introduce a variant of a graph used by Schenck~\cite[Definition~2.5]{Chow}, which also builds on dual graphs of Rose~\cite{r1,r2}.  This graph simplifies the computation of the homology of $\Sigma_{W,\sigma}$ in homological degree $\dim(W)+1$.  In order to construct this graph we need the following easy lemma.

\begin{lem}\label{lem:TwoFace}
Suppose $\psi\subset\R^{n+1}$ is a convex polyhedral cone of dimension $d+2$ and $W\subset\mbox{aff}(\psi)$, where $W$ is a linear subspace of dimension $d$.  Then $\psi$ has at most $2$ faces $\gamma_1,\gamma_2\in\Sigma_{d+1}$ so that $W\subset\mbox{aff}(\gamma_1)$ and $W\subset\mbox{aff}(\gamma_2)$.
\end{lem}
\begin{proof}
This follows from the fact that the intersection of the affine hulls of three distinct codimension one faces of a convex cone $\psi$ cannot intersect in a codimension $2$ linear space.  This would require the supporting hyperplane of one of the faces to be `between' the other two, hence this hyperplane would meet the interior of $\psi$, which is a contradiction.
\end{proof}

\begin{defn}\label{def:WGraph}
Suppose $\Sigma\subset\R^{n+1}$ is a pure,hereditary, $(n+1)$-dimensional fan, $\Sigma'\subset\partial\Sigma$ is a subfan, and $W\subset\R^{n+1}$ is a $d$-dimensional subspace so that $W\subset\bigcap\limits_{\tau\in\Sigma_n\setminus\Sigma'_n} \mbox{aff}(\tau)$.

$G_W(\Sigma,\Sigma')$ is the graph with one vertex for every face in $\Sigma_{d+1}\setminus\Sigma'_{d+1}$.  Also $G_W(\Sigma,\Sigma')$ has one distinguished vertex $v_b$ iff there is at least one face $\psi\in\Sigma_{d+2}\setminus\Sigma'_{d+2}$ having only one face $\gamma$ so that $\gamma\in\Sigma_{d+1}\setminus\Sigma'_{d+1}$.  Two vertices $v,w$ corresponding to $\gamma_v,\gamma_w\in\Sigma_{d+1}\setminus\Sigma'_{d+1}$ are connected in $G_W(\Sigma)$ iff there is a $\psi\in\Sigma_{d+2}\setminus\Sigma'_{d+2}$ so that $\gamma_v,\gamma_w$ are the faces of $\psi$ whose affine spans contain $W$.  Connect the vertex $v$ to the vertex $v_b$ if the corresponding face $\gamma_v\in\Sigma_{d+1}$ is contained in a face $\psi\in\Sigma_{d+2}$ so that $\gamma_v$ is the only $(d+1)$-face of $\psi$ so that $\gamma_v\in\Sigma_{d+1}\setminus\Sigma'_{d+1}$.
\end{defn}

\begin{prop}\label{prop:GHom}
Let $\Sigma\subset\R^{n+1}$ be a pure, hereditary, $(n+1)$-dimensional fan and $\Sigma'\subset\partial\Sigma$ a subfan. Suppose that $W\subset\R^{n+1}$ is a $d$-dimensional subspace so that $W\subset\bigcap\limits_{\tau\in\Sigma_n\setminus\Sigma'_n} \mbox{aff}(\tau)$.  Then
\begin{enumerate}
\item If $\Sigma_d\setminus\Sigma'_d=\emptyset$, then $H_d(\cR[\Sigma,\Sigma'])=0$ and
\[
H_{d+1}(\cR[\Sigma,\Sigma'])=\left\lbrace
\begin{array}{rl}
0 & v_b\in G_W(\Sigma,\Sigma')\\
S & \mbox{otherwise}
\end{array}
\right.
\]
\item If $\Sigma_d\setminus\Sigma'_d\neq\emptyset$, then $H_{d+1}(\cR[\Sigma,\Sigma'])=H_d(\cR[\Sigma,\Sigma'])=0$.
\end{enumerate}
\end{prop}

\begin{proof}
The main point of this proof is that the top cellular boundary map
\[
\phi_W: \bigoplus_{e\in G_W(\Sigma,\Sigma')} S \rightarrow \bigoplus_{v\neq v_b\in G_W(\Sigma,\Sigma')} S
\]
of $G_W(\Sigma,\Sigma')$ (relative to $v_b$, if $v_b$ is present) is really the same (by definition!) as the cellular map
\[
\delta_{d+2}:\cR[\Sigma,\Sigma']_{d+2}\rightarrow \cR[\Sigma,\Sigma']_{d+1}.
\]
$(1)$ Since $\cR[\Sigma,\Sigma']_d=0$, $H_d(\cR[\Sigma,\Sigma'])=0$ and $H_{d+1}(\cR[\Sigma,\Sigma'])=\mbox{coker}(\phi_W)=H_0(G_W(\Sigma,\Sigma'),v_b;S)$, where $v_b$ is understood to be the emptyset if $G_W(\Sigma,\Sigma')$ has no vertex $v_b$.  Since $G_W(\Sigma,\Sigma')$ is connected, this proves $(1)$.

$(2)$ Let $\gamma\in\Sigma_d\setminus\Sigma'_d\neq\emptyset$.  We claim that $\Sigma=\mbox{st}_\Sigma(\gamma)$.  Suppose there is a facet $\sigma'\not\in\mbox{st}_\Sigma(\gamma)$.  Since $\Sigma$ is hereditary, we may assume that $\sigma'$ is adjacent to a $\sigma\in\mbox{st}_\Sigma(\gamma)$.  Set $\tau=\sigma\cap\sigma'$ and $H=\mbox{aff}(\tau)$.  $\tau\not\in\Sigma'$ since $\tau$ is interior, hence $W\subset H$ and $\gamma\subset W\cap\sigma\subset H\cap\sigma$.  This is a contradiction since $H\cap\sigma=\tau$ and we assumed $\gamma\not\in\sigma'$.  Hence $\Sigma=\mbox{st}_\Sigma(\gamma)$, and $\gamma$ is the unique face in $\Sigma_d\setminus\Sigma'_d$.  It follows that $H_d(\cR[\Sigma,\Sigma'])=0$ since the map
\[
\delta_{d}:\cR[\Sigma,\Sigma']_{d+1}=\oplus_{\gamma\in\Sigma_{d+1}\setminus\Sigma'_{d+1}} S \rightarrow S=\cR[\Sigma,\Sigma']_d
\]
is surjective.  Furthermore, in this case $v_b$ is \textit{not} present in $G_W(\Sigma,\Sigma')$, so $\mbox{coker}(\phi_W)$ $=S=\mbox{coker}(\delta_{d+2})$ and $H_{d+2}(\cR[\Sigma,\Sigma'])=0$ as well.
\end{proof}

\begin{exm}\label{ex:SCCentralLoc2}
Let $\Sigma$, $W$, and $\Sigma_{W,\sigma}$ all be as in Example~\ref{ex:SCCentralLoc}.  The graph $G_W(\Sigma_{W,\sigma},\partial\Sigma_{W,\sigma})$ has four vertices corresponding to the four interior codimension one faces and four edges which connect these vertices into a cycle.  There is no vertex $v_b$ since all four facets having a codimension one face whose affine span contains $W$ have precisely $2$ such faces.  Hence $H_2(\cR[\Sigma_{W,\sigma},\partial\Sigma_{W,\sigma}])=S$ by Proposition~\ref{prop:GHom}, as we computed in Example~\ref{ex:SCCentralLoc}.
\end{exm}

\begin{exm}\label{ex:SCXLine}
Let $\Sigma$ be as in Example~\ref{ex:SCCentralLoc}.  Let $V\in L_{\Sigma,\Sigma^{-1}}$ be the $x$-axis, which we obtain as the intersection of the four affine spans shown in Figure~\ref{fig:SCYLine}.  The corresponding lattice complex $\Sigma_{V,\sigma}$, where $\sigma$ is any of the three facets having a codimension one face $\gamma$ so that $V\in\mbox{supp}(\gamma)$, is shown in Figure~\ref{fig:PSCYLine}.  The graph $G_V(\Sigma_{V,\sigma},\Sigma^{-1}_{V,\sigma})$ can have three or four vertices depending on whether we impose vanishing along none, one, or both of the codimension one faces of $\partial\Sigma\cap\Sigma_{V,\sigma}$.  $G_V(\Sigma_{V,\sigma},\Sigma^{-1}_{V,\sigma})$ has $v_b$ as a vertex unless $\Sigma^{-1}$ contains neither of these codimension one faces, hence this is the only case which leads to a nontrivial contribution to $H_2(\cR/[\Sigma,\Sigma^{-1}])$.  For such a choice of $\Sigma^{-1}$, Proposition~\ref{prop:GHom} yields that $H_2(\cR[\Sigma,\Sigma^{-1}])=S$.  The same arguments as in Example~\ref{ex:SCCentralLoc} yield that
\[
H_2(\cR/\cJ[\Sigma,\Sigma^{-1}])_{I(V)}=(S/\sum_{i=1}^4 J(\tau_i))_{I(V)},
\]
where $\tau_1,\ldots,\tau_4$ are the four codimension one faces of $\Sigma_{V,\sigma}$ whose affine spans contain $V$.

\begin{figure}[htp]
\begin{subfigure}[b]{.4\textwidth}
\includegraphics[width=\textwidth]{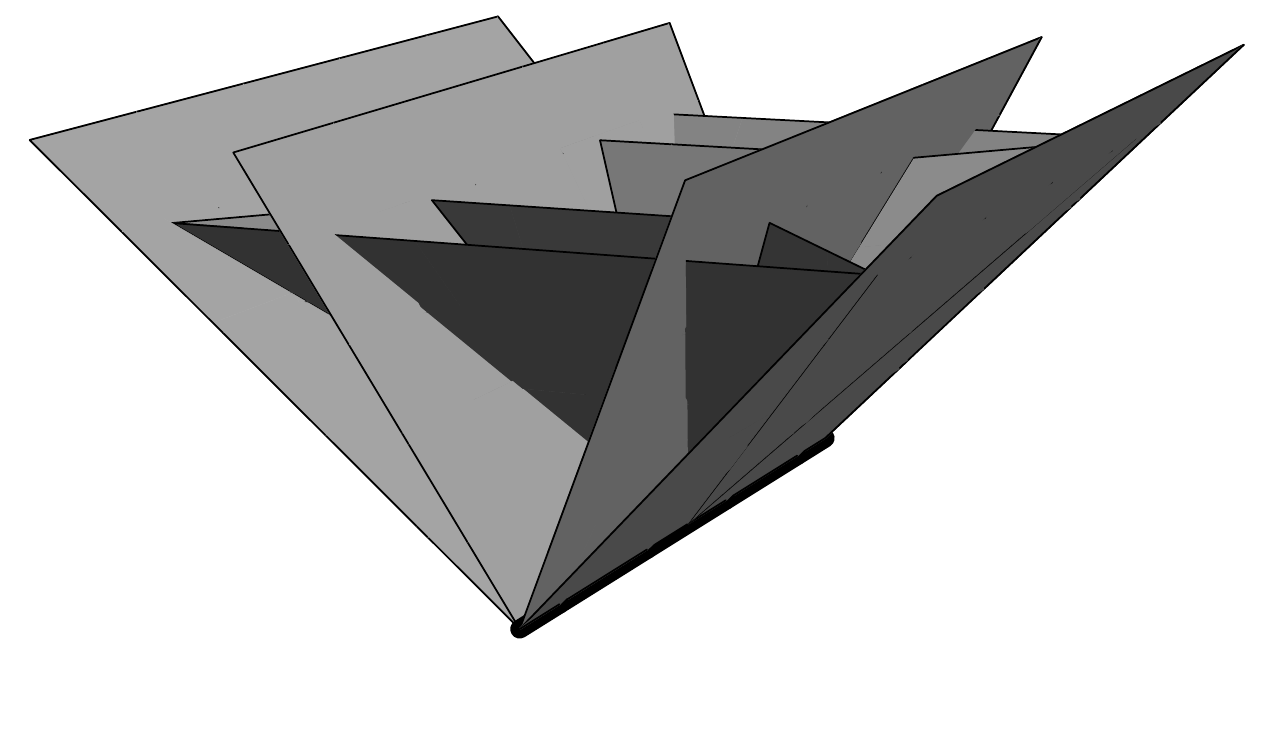}
\caption{$V$}\label{fig:SCYLine}
\end{subfigure}
\begin{subfigure}[b]{.4\textwidth}
\includegraphics[width=\textwidth]{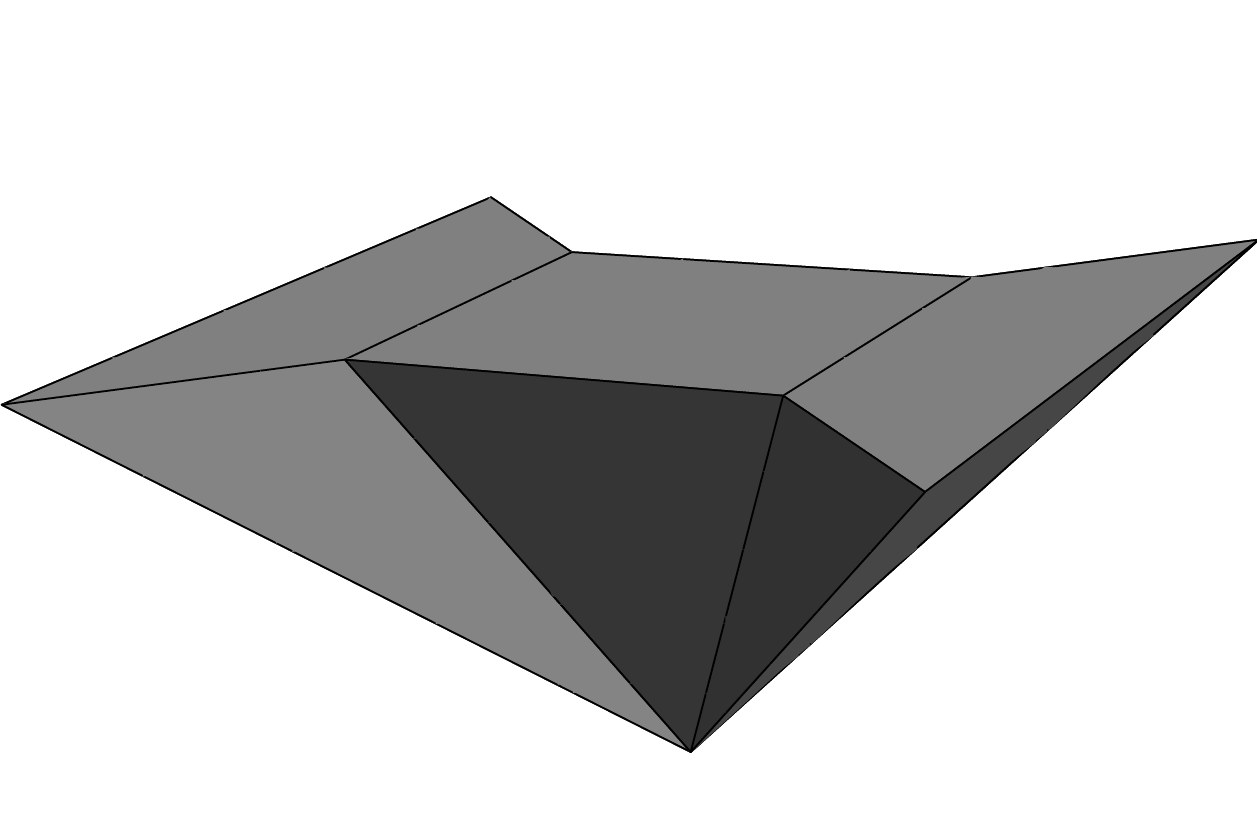}
\caption{$\PC(\Sigma_{V,\sigma})$}\label{fig:PSCYLine}
\end{subfigure}
\caption{}
\end{figure}

\end{exm}

\section{Associated Primes of the Spline Complex}\label{sec:Ass}

The primary objective of this section is to describe associated primes of the homology modules $H_i(\cR/\cJ[\Sigma,\Sigma'])$, which requires some commutative algebra.  Recall that if $M$ is an $R$-module, a prime $P\subset R$ is \textit{associated} to $M$ if there is an injection
\[
R/P \xrightarrow{\cdot m} M
\]
given by multiplication by some $m\in M$.  Equivalently, $P=\mbox{ann}(m)$.  The set of associated primes of $M$ is denoted by $\mbox{Ass}_R(M)$.

We collect a few facts about associated primes which we will need.  The first two can be found in~\cite[\S 6]{Matsumura}.  The third is a special case of a general theorem~\cite[Theorem 23.2]{Matsumura} which describes behavior of associated primes under flat extensions.  Since we only need a particular case, we give a proof here.

\begin{prop}\label{prop:AssFacts}
Let $R$ be a commutative ring, and $M,M_1,\ldots,M_k$ $R$-modules.
\begin{enumerate}
\item $P\in\mbox{Ass}_R(M)\iff PR_P\in\mbox{Ass}_{R_P}(M_P)$
\item $\mbox{Ass}(\bigoplus_{i=1}^k M_i)=\bigcup_{i=1}^k \mbox{Ass}(M_i)$
\item Suppose $S=R[x_1,\ldots,x_k]$.  Then
\[
\mbox{Ass}_S(M\otimes_R S)=\{PS|P\in \mbox{Ass}_R(M)\}.
\]
\end{enumerate}
\end{prop}
\begin{proof}[Proof of (3)]
The general case follows directly from the case $S=R[x]$ by induction on $k$.  So we prove
\[
\mbox{Ass}_S(M\otimes_R S)=\{PS|P\in \mbox{Ass}_R(M)\}
\]
when $S=R[x]$.  First, $PS$ is prime because $S/PS=R/P[x]$ is a domain.  Now, suppose $P\in \mbox{Ass}_R(M)$.  Then there is an injection
\[
R/P\hookrightarrow M.
\]
Since $S$ is a \textit{flat} extension of $R$, tensoring with $S$ is exact.  Tensoring the above injection with $S$ hence yields an injection
\[
S/PS\hookrightarrow M\otimes_R S,
\]
So $PS\in \mbox{Ass}_S(M\otimes_R S)$.  Now suppose $Q\in \mbox{Ass}_S(M\otimes_R S)$.  Via the identification $M\otimes_R S=M[x]$, $Q=\mbox{ann}_S(f)$ for some $f=\sum_i m_ix^i$.  First we show that $P=Q\cap S=\mbox{ann}_R(f)\in \mbox{Ass}_R (M)$.  We need to show that $P=\mbox{ann}_R(m)$ for some $m\in M$.  $rf=0$ for some $r\in R$ iff $rm_i=0$ for every $i$.  Hence $P=\cap_{i=0}^k \mbox{ann}_R(m_i)$.  But $P$ is prime, so we must have $P=\mbox{ann}_R(m_i)$ for some $i$.  Hence $P\in \mbox{Ass}_R (M)$.  Now we show that $Q=PS$.  Suppose $g=\sum_{j=0}^k a_jx^j\in Q$, with $a_j\in R$.  We need to show that $a_j\in P$ for $j=0,\ldots,k$.  Allowing some $a_j,m_i$ to be zero, we may assume that $f=\sum_{i=0}^k m_ix^i$.  Expanding $fg$ gives
\[
\sum_c \left( \sum_{i+j=c} r_jm_i \right) x^c.
\]
Setting $fg=0$ yields the equations
\[
\sum_{i+j=c} r_jm_i=0
\]
for every $c$.  To establish that $a_0\in P$, note first that $a_0m_0=0\implies a_0\in \mbox{ann}_R(m_0)$.  Now multiply the equation
\[
a_0m_1+a_1m_0=0
\]
by $a_0$ to obtain $a_0^2m_1=0$, so $a_0^2\in \mbox{ann}_R(m_1)$.  Continuing in this way, we see that $a_0^{k+1}\in \cap_{j=0}^k \mbox{ann}_R(m_j)=P$.  But $P$ is prime, so $a_0\in P$ and all terms involving $a_0$ in the above equations drop out.  Now repeat this process with each successive $a_j$, yielding $g\in PS$.
\end{proof}

\begin{lem}\label{lem:ProjAssW}
Let $\Sigma\subset\R^{n+1}$ be a pure, hereditary, $(n+1)$-dimensional fan, $\alpha$ a choice of smoothness parameters, $\Sigma'\subset\Sigma$ a subfan, and $W\subset\R^{n+1}$ a linear subspace so that $W\subset\cap_{\tau\in \Sigma_n\setminus\Sigma'_n}\mbox{aff}(\tau)$.  Then 
\[
P\in\mbox{Ass}_S(H_i(\cR/\cJ[\Sigma,\Sigma'])) \implies P\subseteq I(W)
\]
\end{lem}
\begin{proof}
Let $d=\dim W$.  Let $V$ be a complementary vector space, so $V\cap W=\mathbf{0}$ and $\dim V=n+1-d$, and let $\pi:\R^{n+1}\rightarrow V$ be the projection onto $V$ with kernel $W$.  Then we can view the coordinate ring $\R[V]$ of $V$ in two ways.  Via the inclusion $i:V\rightarrow \R^{n+1}$ we represent $\R[V]$ as the quotient $S\xrightarrow{i^*} S/I(V)$.  Via the projection $\pi:\R^{n+1}\rightarrow V$ with kernel $W$, we represent $\R[V]\xrightarrow{\pi^*} S$ as an inclusion, where $\R[V]$ is generated as a subalgebra of $S$ by a choice of $n+1-d$ linear forms which vanish on $W$.  Here we will regard $\R[V]$ as a subalgebra of $S$ via $\pi^*$.  We first prove that there is a complex $\C$ of $\R[V]$-modules so that $\cR/\cJ[\Sigma,\Sigma']\cong \C\otimes_{\R[V]} S$ as complexes.

Denote $J(\gamma)\cap\R[V]$ by $J(\pi(\gamma))$: $\pi(\gamma)$ is a cone in $V$ and smoothness parameters can be assigned naturally to its faces so that this makes notational sense.  The ideals $J(\gamma)$ for $\gamma\in \Sigma\setminus\Sigma'$ are generated by powers of linear forms contained in the subalgebra $\R[V]$, since every face $\gamma\in \Sigma\setminus\Sigma'$ has $W\subset\mbox{aff}(\gamma)$.  It follows that $J(\pi(\gamma))\otimes_{\R[V]} S=J(\gamma)$.
We hence have
\[
\dfrac{S}{J(\gamma)}\cong \dfrac{\R[V]}{J(\pi(\gamma))}\otimes_{\R[V]} S.
\]
This tensor decomposition respects the differential of the complex $\cR/\cJ[\Sigma,\Sigma']$, so the desired complex $\C$ of $\R[W]$-modules is obtained by setting
\[
\C_i=\bigoplus_{\dim \gamma=i} \dfrac{\R[V]}{J(\pi(\gamma))}
\]
with cellular differential.

Now that we have found such a $\C$, we have $H_i(\cR/\cJ[\Sigma,\Sigma'])=H_i(\C\otimes_{\R[V]} S)$. Since $S$ is a \textit{flat} $\R[V]$-algebra, $\otimes_{\R[V]} S$ is exact and hence
\[
H_i(\C\otimes_{\R[V]} S)\cong H_i(\C)\otimes_{\R[V]} S.
\]
Every associated prime of $H_i(\C)$ is contained in the homogeneous maximal ideal $I(\pi(W))=I(W)\cap\R[V]$ of $\R[V]$.  Now the result follows from~\ref{prop:AssFacts} part (3), since associated primes of $H_i(\C)\otimes_{\R[V]} S$ are obtained by extending associated primes of $H_i(\C)$.
\end{proof}

\begin{remark}
There is a natural way to interpret $\C$ geometrically.  From the proof of Lemma~\ref{lem:ProjAssW}, we see that
\[
\C_i=\bigoplus_{\dim \gamma=i} \dfrac{\R[V]}{J(\pi(\gamma))},
\]
where $\pi$ is the projection (with kernel $W$) of $\R^{n+1}$ onto a complementary subspace $V$.  Hence $\C$ `should' be $\cR/\cJ[\pi(\Sigma),\pi(\Sigma)\setminus\pi(\Sigma\setminus\Sigma')]$.  The reason for the choice of $\pi(\Sigma)\setminus\pi(\Sigma\setminus\Sigma')$ is that it is possible for $\pi(\tau)=\pi(\psi)$, where $\psi\in\Sigma\setminus\Sigma'$ and $\tau\in\Sigma'$.  In order for this to make sense in the framework we have presented, $\pi(\Sigma)$ and $\pi(\Sigma)\setminus\pi(\Sigma\setminus\Sigma')$ both need to have the structure of fans.  A priori all we know about $\pi(\Sigma)$ is that it is a union of cones (the projections of the faces of $\Sigma$), and it may be that there is no meaningful way to give this union the structure of a fan.  We describe a special case where it is possible to give $\pi(\Sigma)$ and $\pi(\Sigma)\setminus\pi(\Sigma\setminus\Sigma')$ a meaningful structure, which we will use in \S~\ref{sec:Fourth}.

Suppose $\PC\subset\R^n$ is the star of a vertex $v\in\PC_0$ and $\Sigma=\wPC\subset\R^{n+1}$.  Let $\PC_v$ be the fan with faces $\mbox{cone}(\gamma-v)$ for every $\gamma\in\PC$ with $v\in\gamma$.  This puts faces of $\PC_v$ in a clear dimension preserving bijection with faces of $\PC$ containing $v$.  For each $\tau\in\PC_{n-1}$ with $v\in\tau$, assign the smoothness parameter $\alpha(\tau)$ to the codimension one face $\mbox{cone}(\tau-v)$ of $\PC_v$.  See Figure~\ref{fig:StarConeProj} for this setup.

\begin{figure}
\begin{subfigure}[b]{.25\textwidth}
\centering
\includegraphics[width=\textwidth]{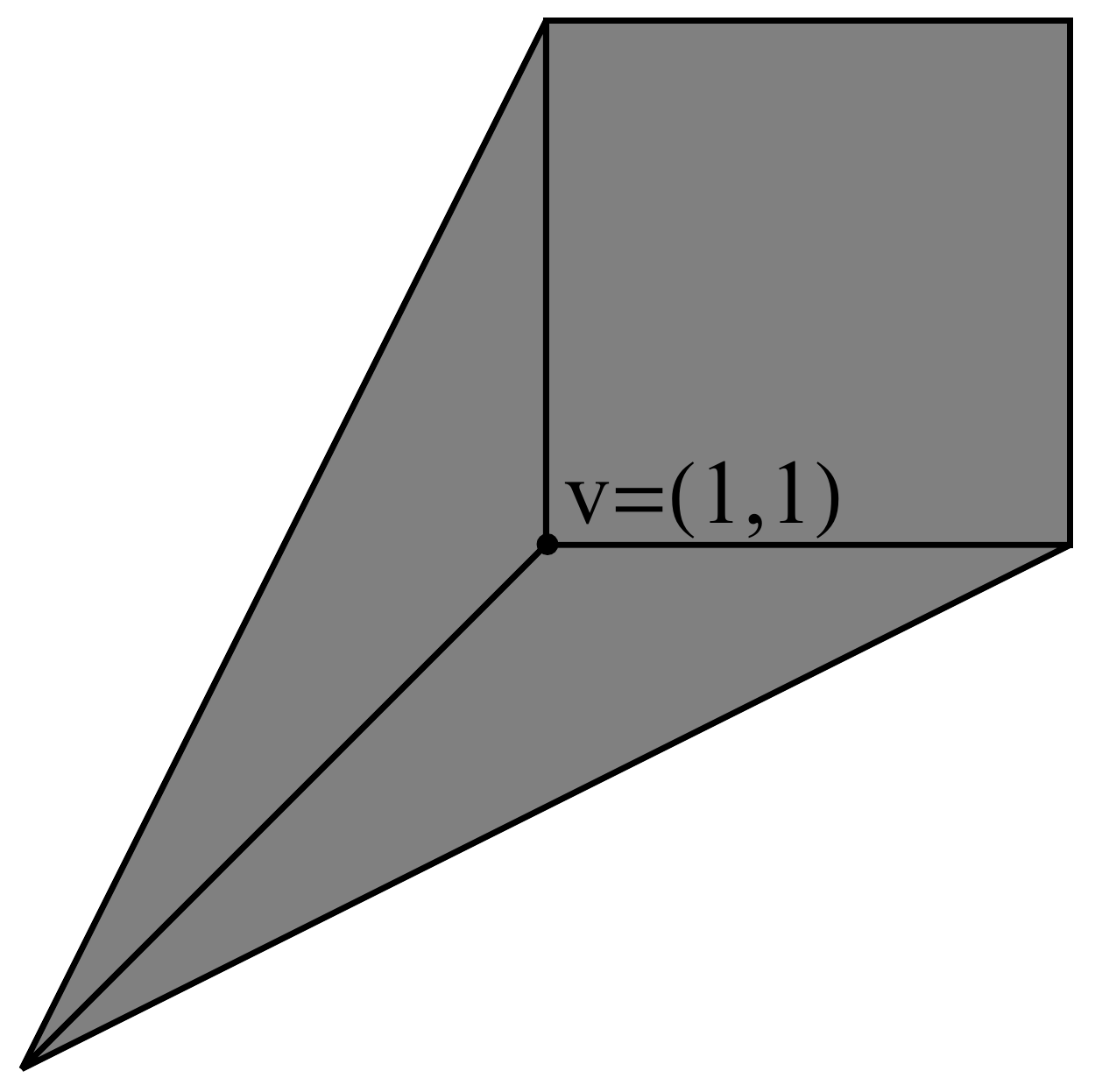}
\caption{$\PC$}
\end{subfigure}
\begin{subfigure}[b]{.4\textwidth}
\includegraphics[width=\textwidth]{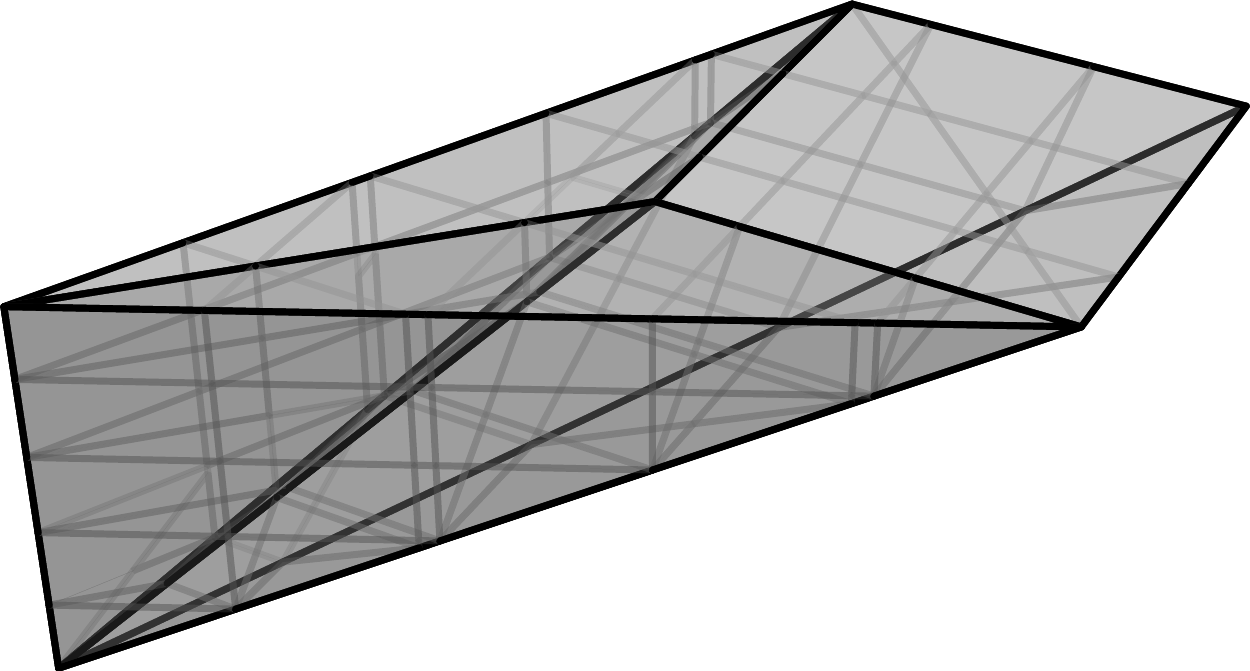}
\caption{$\Sigma$}
\end{subfigure}
\begin{subfigure}[b]{.25\textwidth}
\includegraphics[width=\textwidth]{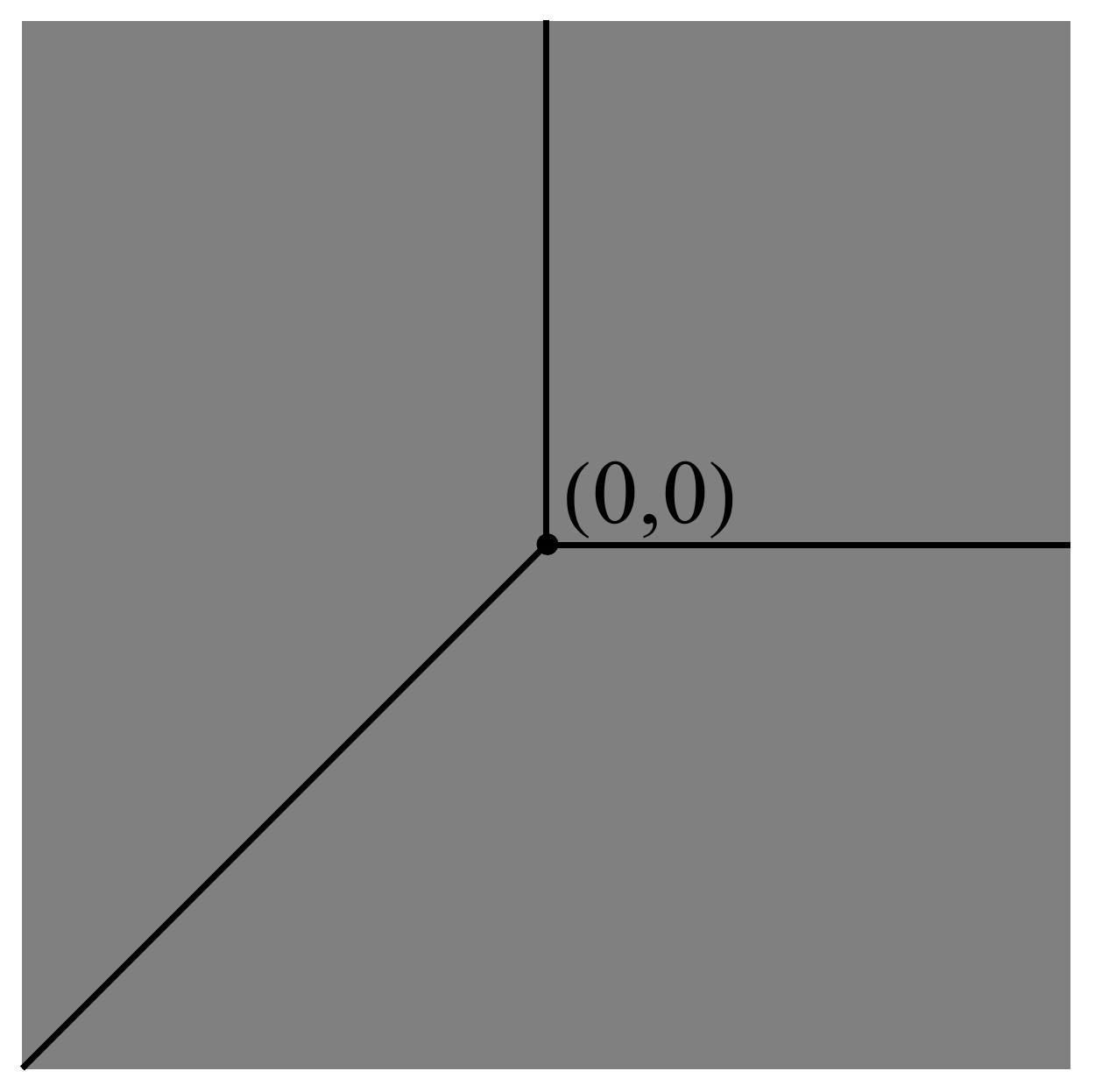}
\caption{$\PC_v$}
\end{subfigure}
\caption{A $2$-dimensional star, its cone and projection}\label{fig:StarConeProj}
\end{figure}

\end{remark}

\begin{lem}\label{lem:StarProjection}
Let $\PC,\Sigma,\PC_v$ be as defined above.  Set $W=\mbox{aff}(\widehat{v})\subset\R^{n+1}$, let $V\cong\R^n$ be the complementary subspace defined by the vanishing of $x_0$, and denote by $\pi:\R^{n+1}\rightarrow V$ the projection with kernel $W$. Also let $R=\R[V]=\R[x_1,\ldots,x_n]$ and $S=\R[x_0,\ldots,x_n]$.  Then
\begin{enumerate}
\item $\pi(\Sigma)=|\PC_v|$
\item $\pi(\Sigma)\setminus\pi(\Sigma\setminus\Sigma^{-1})=\PC_v^{-1}$
\item $\cR/\cJ[\Sigma,\Sigma^{-1}]\cong \cR/\cJ[\PC_v,\PC^{-1}_v](-1)\otimes_R S$,
\end{enumerate}
where the $-1$ in parentheses records a homological shift in dimension.  In particular, if $\Sigma^{-1}=\partial\Sigma$, then $\cR/\cJ[\Sigma,\partial\Sigma]\cong \cR/\cJ[\PC_v,\partial\PC_v](-1)\otimes_R S.$
\end{lem}
\begin{proof}
To show that $\pi(\Sigma)=|\PC_v|$, let us first show that if $\gamma\subset\PC$ is a face of $\PC$ containing $v$, then $\pi(\wgamma)=\mbox{cone}(\gamma-v)$.  Let $\gamma=\mbox{conv}(v,v+q_1,\ldots,v+q_k)$ with coordinates $v=(v_1,\ldots,v_n)$ and $q_i=(q^1_i,\ldots,q^n_i)$ for $i=1,\ldots,k$.  Set $q'_i=(0,q^1_i,\ldots,q^n_i)\in V$ and $v'=(1,v_1,\ldots,v_n)$.  We have $W=\mbox{aff}(\widehat{v})=\mbox{aff}(v')$, $\wsigma=\mbox{cone}(v',v'+q'_1,\ldots,v'+q'_k)$.  Then $\pi(v'+q'_i)=q'_i$ and $\pi(\wgamma)=\mbox{cone}(\mathbf{0},q'_1,\ldots,q'_k)=\mbox{cone}(\gamma-v)$.  Now, suppose that $\gamma\in\PC$ does not contain $v$.  By definition of the star of a vertex, $\gamma\subset\psi$, where $\psi\in\PC$ contains $v$.  Since $\wgamma\subset\widehat{\psi}$, $\pi(\wgamma)\subset\pi(\widehat{\psi})\subset|\PC_v|$.

To show that $\pi(\Sigma)\setminus\pi(\Sigma\setminus\Sigma^{-1})=\PC^{-1}_v$, we claim that
\[
\pi(\Sigma)\setminus\pi(\Sigma\setminus\Sigma^{-1})=\bigcup\limits_{\gamma\in\Sigma^{-1},\widehat{v}\subset\gamma}\pi(\gamma).
\]
To prove this, suppose $x\in \pi(\Sigma)\setminus\pi(\Sigma\setminus\Sigma^{-1})$.  First we show that $x\in\pi(\gamma)$ for some $\gamma\in\Sigma^{-1}$ such that $\widehat{v}\subseteq\gamma$.  Suppose not, and let $x'\in\gamma$ so that $\pi(x')=x$.  Then $x'+w\notin\gamma$ for some positive multiple $w$ of $v'$.  But $x'+w\in\sigma$ for any facet $\sigma$ containing $x'$, since $\Sigma=\mbox{st}(\widehat{v})$.  Hence it follows that $x'+w\in\Sigma\setminus\Sigma^{-1}$.  Since $\pi(x'+w)=\pi(x')=x$, this is a contradiction.  So $x\in\pi(\gamma)$ for some $\gamma\in\Sigma^{-1}$ such that $\widehat{v}\subseteq\gamma$.  We now claim that $\pi(\gamma)\cap\pi(\Sigma\setminus\Sigma^{-1})=\emptyset$.  To see this suppose that $y\in\gamma$ and $\pi(y)=\pi(y')$ for some $y'\notin\Sigma^{-1}$.  Then $y'=y+w$ for some $w\in\mbox{aff}(v')=W$.  Since only nonnegative muliplies of $v'$ intersect nontrivially with $\Sigma$, we have either that $y'=y+w$ for some $w\in\widehat{v}$ or $y=y'+w$ for some $w\in\widehat{v}$.  But $\widehat{v}\subset\gamma$, so we see that in either case, $y'\in\gamma$.  This contradicts the choice of $y'$, since we assumed $\gamma\in\Sigma^{-1}$.

We have
\[
\begin{array}{rl}
\cR/\cJ[\Sigma,\Sigma^{-1}]_{i+1}= & \bigoplus\limits_{\gamma\in\Sigma^{\ge 0}_{i+1}} \dfrac{S}{J(\gamma)}\\
= & \bigoplus\limits_{\gamma\in\pi(\Sigma^{\ge 0}_{i+1})} \dfrac{R}{J(\pi(\gamma))}\otimes_R S\\
= & \cR/\cJ[\PC_v,\PC^{-1}_v]_i \otimes_R S.
\end{array}
\]
The differential is in degree $0$ and so commutes with tensoring, hence
\[
\cR/\cJ[\Sigma,\Sigma^{-1}]\cong \cR/\cJ[\PC_v,\PC_v^{-1}](-1)\otimes_R S.
\]
Finally, we must show that $\partial\pi(\Sigma)=\pi(\Sigma)\setminus\pi(\Sigma\setminus\partial\Sigma)$.  Since interior faces of $\pi(\Sigma)$ are projections of interior faces of $\Sigma$, this is clear.
\end{proof}

The following theorem generalizes~\cite[Lemma~3.1]{Spect} and~\cite[Lemma~2.4]{Chow}, precisely describing the form which associated primes of $\cR/\cJ[\Sigma,\Sigma^{-1}]$ must take.

\begin{thm}\label{thm:assPrimes1}
Let $\Sigma\subset\R^{n+1}$ be a pure, hereditary, $(n+1)$-dimensional fan with smoothness parameters $\alpha$.  For $1\le i\le n$ we have
\begin{enumerate}
\item $\mbox{Ass}_S(H_i(\cR/\cJ[\Sigma,\Sigma^{-1}]))\subset\{I(W)|W\in L_{\Sigma,\Sigma^{-1}},\dim(W)\le i-1\}$
\item If $H_i(\cR[\Sigma_{W,\sigma},\Sigma^{-1}_{W,\sigma}])=0$ for every $W\in L_{\Sigma,\Sigma^{-1}}$ with $\dim W=i-1$, then
\[
\mbox{Ass}_S(H_i(\cR/\cJ[\Sigma,\Sigma^{-1}]))\subset\{I(W)|W\in L_{\Sigma,\Sigma^{-1}},\dim(W)\le i-2\}
\]
\item If $\Sigma$ is simplicial, then 
\[
\mbox{Ass}_S(H_i(\cR/\cJ[\Sigma,\Sigma^{-1}]))\subset\{I(\gamma)|\gamma\in\Sigma,\mbox{aff}(\gamma)\in L_{\Sigma,\Sigma^{-1}}, \dim(\gamma)\le i-2\}
\]
\end{enumerate}
\end{thm}
\begin{proof}
Assume $H_i(\cR/\cJ[\Sigma,\Sigma^{-1}])\neq 0$, so $\mbox{Ass}_S(H_i(\cR/\cJ[\Sigma,\Sigma^{-1}]))\neq\emptyset$.  Let $P\in \mbox{Ass}_S(H_i(\cR/\cJ[\Sigma,\Sigma^{-1}]))$ and set $W=\max_{V\in L_{\Sigma,\Sigma^{-1}}}\{I(V)|I(V)\subset P\}$.  If $W=\R^{n+1}$, so that $I(W)=0$, then $H_i(\cR/\cJ[\Sigma,\Sigma^{-1}])_P=0$ for $1\le i \le n$.  So if $P\in \mbox{Ass}_S(H_i(\cR/\cJ[\Sigma,\Sigma^{-1}]))$, it must contain at least one ideal of the form $I(\mbox{aff}(\tau))$, $\tau\in\Sigma^{\ge 0}_n$, and $W$ must be a proper subspace of $\R^{n+1}$.  Then $P\in \mbox{Ass}_S(H_i(\cR/\cJ[\Sigma,\Sigma^{-1}]))$
\[
\begin{array}{rl}
\iff & PS_P\in\mbox{Ass}_{S_P}(H_i(\cR/\cJ[\Sigma,\Sigma^{-1}])_P)\\
\iff & PS_P\in\mbox{Ass}_{S_P}(H_i(\cR/\cJ[\Sigma_{W,\sigma},\Sigma^{-1}_{W,\sigma}])_P)\mbox{ for some } \sigma\in\Sigma_{n+1}\\
\iff & P\in \mbox{Ass}_S(H_i(\cR/\cJ[\Sigma_{W,\sigma},\Sigma^{-1}_{W,\sigma}]))\mbox{ for some } \sigma\in\Sigma_{n+1}.
\end{array}
\]
The first and last equivalences follow from Proposition~\ref{prop:AssFacts} part (1).  The second equivalence follows from Lemma~\ref{lem:Localize} and Lemma~\ref{prop:AssFacts} part (2).

Now pick $\sigma$ so that $P\in \mbox{Ass}_S(H_i(\cR/\cJ[\Sigma_{W,\sigma},\Sigma^{-1}_{W,\sigma}]))$.  It is clear that $W\subset \bigcap\limits_{\tau\in (\Sigma^{\ge 0}_{W,\sigma})_n} \mbox{aff}(\tau)$.  Hence by Lemma~\ref{lem:ProjAssW}, $P\subseteq I(W)$.  But $I(W)\subseteq P$ by construction, so $P=I(W)$.

By Proposition~\ref{prop:GHom}, $H_k(\cR[\Sigma_{W,\sigma},\Sigma^{-1}_{W,\sigma}])=0$ for $k\le \dim(W)$.  By the long exact sequence in homology corresponding to the short exact sequence
\[
0\rightarrow\cJ[\Sigma_{W,\sigma},\Sigma^{-1}_{W,\sigma}]\rightarrow \cR[\Sigma_{W,\sigma},\Sigma^{-1}_{W,\sigma}] \rightarrow \cR/\cJ[\Sigma_{W,\sigma},\Sigma^{-1}_{W,\sigma}]\rightarrow 0,
\]
we obtain that $H_{\dim(W)}(\cR/\cJ[\Sigma_{W,\sigma},\Sigma^{-1}_{W,\sigma}])=0$ as well.  Hence $i\ge \dim(W)+1$ if $I(W)\in\mbox{Ass}(H_i(\cR/\cJ[\Sigma,\Sigma^{-1}])),$ which gives the condition on dimension.  This concludes the proof of $(1)$.

If in addition $H_{\dim(W)+1}(\cR[\Sigma_{W,\sigma}),\Sigma^{-1}_{W,\sigma}])=0$, then the long exact sequence in homology yields
\[
H_{\dim(W)+1}(\cR/\cJ[\Sigma_{W,\sigma},\Sigma^{-1}_{W,\sigma}])\cong H_{\dim(W)}(\cJ[\Sigma_{W,\sigma},\Sigma^{-1}_{W,\sigma}]).
\]
If $\left(\Sigma^{\ge 0}_{W,\sigma}\right)_{\dim W}=\emptyset$ then $\cJ[\Sigma_{W,\sigma},\Sigma^{-1}_{W,\sigma}]_{\dim(W)}=0$ automatically.  If $\left(\Sigma^{\ge 0}_{W,\sigma}\right)_{\dim W}\neq\emptyset$ then we saw in the proof of Proposition~\ref{prop:GHom} that $\Sigma_{W,\sigma}$ must be the star of a face $\gamma$ with $\mbox{aff}(\gamma)=W$.  In this case $H_{\dim(W)}(\cJ[\Sigma_{W,\sigma},\Sigma^{-1}_{W,\sigma}])$ is the cokernel of the map
\[
\bigoplus_{\psi\in\left(\Sigma^{\ge 0}_{W,\sigma}\right)_{\dim(W)+1}} J(\psi) \rightarrow J(\gamma).
\]
Since $\gamma$ is an interior face, the sum on the left hand side runs over all ideals of faces $\psi\in \left(\Sigma_{W,\sigma}\right)_{\dim(W)+1}$ so that $\gamma\in\psi$.  By definition,
\[
\sum\limits_{\substack{\gamma\in\psi\\ \psi\in\left(\Sigma^{\ge 0}_{W,\sigma}\right)_{\dim(W)+1}}} J(\psi)=\sum\limits_{\substack{\gamma\in\tau\\ \tau\in\left(\Sigma^{\ge 0}_{W,\sigma}\right)_n}} J(\tau)=J(\gamma),
\]
so the map above is surjective and $H_{\dim(W)}(\cJ[\Sigma_{W,\sigma},\Sigma^{-1}_{W,\sigma}])=0$, hence 

\noindent $H_{\dim(W)+1}(\cR/\cJ[\Sigma_{W,\sigma},\Sigma^{-1}_{W,\sigma}])=0$ as well.  This completes the proof of $(2)$.

Now suppose $\Sigma$ is simplicial.  Then $\Sigma_{W,\sigma}=\mbox{st}_\Sigma(\gamma)$ for some $\gamma$ with $W\subset\mbox{aff}(\gamma)$, by Proposition~\ref{lem:FanStar}.  Replacing $W$ with $\mbox{aff}(\gamma)$ if necessary, we may assume that $W=\mbox{aff}(\gamma)$, hence $P=I(W)=I(\gamma)$.  We also have that $H_{\dim(\gamma)+1}(\cR[\Sigma_{W,\sigma},\Sigma^{-1}_{W,\sigma}])=0$ by part $(2)$ of Proposition~\ref{prop:GHom}.  As in the proof of $(2)$, this yields
\[
H_{\dim(W)+1}(\cR/\cJ[\Sigma_{W,\sigma},\Sigma^{-1}_{W,\sigma}])=0,
\] 
so the condition on dimension follows.  This concludes the proof of $(3)$.
\end{proof}

\begin{remark}
Originally Billera defined the complex $\cR/\cJ'[\Sigma,\partial\Sigma]$ with uniform smoothness $r$ using the ideals $J'(\gamma)=I(\gamma)^{r+1}$~\cite{Homology}.  The proof of Theorem~\ref{thm:assPrimes1} shows precisely where using these ideals leads to associated primes of higher dimension: namely, the map
\[
\bigoplus\limits_{\substack{\psi\in\Sigma^{\ge 0}_{\dim(W)}\\ \gamma\in\psi}} J'(\psi)\rightarrow J'(\gamma)
\]
is not necessarily surjective, so while (1) would hold for this complex, (2) and (3) would not.  The price for using the ideals $J'(\gamma)$, which are simpler to understand, is more complicated homology modules.
\end{remark}

In Example~\ref{ex:3DimMorganScot2} we will see how to couple Theorem~\ref{thm:assPrimes1} with Lemma~\ref{lem:StarProjection} to obtain some interesting relationships between associated primes and the existence of `unexpected' splines of certain degrees.

We can be even more precise about associated primes $I(W)$ of $H_i(\cR/\cJ[\Sigma])$ with $\dim(W)=i-1$.  This is a slight generalization of~\cite[Theorem~2.6]{Chow}.

\begin{thm}\label{thm:assPrimes2}
Let $\Sigma\subset\R^{n+1}$ be a pure, hereditary, $(n+1)$-dimensional fan with smoothness parameters $\alpha$.  Let $W\in L_{\Sigma,\Sigma^{-1}}$ be a flat of dimension $d-1$, where $2\le d\le n+1$.  Then $I(W)$ is associated to $H_d(\cR/\cJ[\Sigma_{W,\sigma},\Sigma^{-1}_{W,\sigma}])$ iff one of the following equivalent conditions hold.
\begin{enumerate}
\item $H_d(\cR[\Sigma_{W,\sigma},\Sigma^{-1}_{W,\sigma}])\neq 0$
\item $H_d(\cR[\Sigma_{W,\sigma},\Sigma^{-1}_{W,\sigma}])=S$
\item $G_W(\Sigma_{W,\sigma},\Sigma^{-1}_{W,\sigma})$ has no $v_b$ vertex and $\Sigma_{W,\sigma}$ is not the star of a face.
\end{enumerate}

Moreover, we have
\[
H_d(\cR/\cJ[\Sigma_W,\Sigma^{-1}_W])=\bigoplus\limits_{\substack{\sigma\in\Gamma_W\\ H_d(\cR[\Sigma_{W,\sigma},\Sigma_{W,\sigma}^{-1}])\neq 0}} \left(\dfrac{S}{\sum_{\tau\in(\Sigma^{\ge 0}_{W,\sigma})_n} J(\tau)}\right),
\]
where $\Gamma_W$ runs across a set of representatives for the equivalence classes $[\sigma]_W$.
\end{thm}
\begin{proof}
The equivalence of the three conditions is a consequence of Proposition~\ref{prop:GHom}.  Assuming any one of these, the long exact sequence coming from
\[
0\rightarrow\cJ[\Sigma_{W,\sigma},\Sigma^{-1}_{W,\sigma}]\rightarrow \cR[\Sigma_{W,\sigma},\Sigma^{-1}_{W,\sigma}] \rightarrow \cR/\cJ[\Sigma_{W,\sigma},\Sigma^{-1}_{W,\sigma}]\rightarrow 0
\]
yields that
\[
H_d(\cR/\cJ[\Sigma_{W,\sigma},\Sigma^{-1}_{W,\sigma}])=\dfrac{S}{\sum_{\tau\in(\Sigma^{\ge 0}_{W,\sigma})_n} J(\tau)}.
\]
Now the result follows from
\[
H_d(\cR/\cJ[\Sigma_W,\Sigma^{-1}_W])=\bigoplus_{\sigma\in\Gamma_W} H_d(\cR/\cJ[\Sigma_{W,\sigma},\Sigma^{-1}_{W,\sigma}]).
\]
\end{proof}

\begin{cor}\label{cor:dimHi}
Let $\Sigma\subset\R^{n+1}$ be a pure, hereditary, $(n+1)$-dimensional fan with smoothness parameters $\alpha$.  Then for $2\le i \le n$, $\dim H_i(\cR/\cJ[\Sigma,\Sigma^{-1}])\le i-1$ with equality iff $H_i(\cR[\Sigma_{W,\sigma},\Sigma^{-1}_{W,\sigma}])\neq 0$ for some $i-1$ dimensional flat $W\in L_{\Sigma,\Sigma^{-1}}$ and some $\sigma\in\Sigma_{n+1}$.
\end{cor}
\begin{proof}
The statement $\dim H_i(\cR/\cJ[\Sigma,\Sigma^{-1}])\le i-1$ is a consequence of Theorem~\ref{thm:assPrimes1} part (1).  The backward implication for equality is Theorem~\ref{thm:assPrimes1} part (2), while the forward implication is provided by Theorem~\ref{thm:assPrimes2}.
\end{proof}

\section{Examples}

From Corollary~\ref{cor:dimHi} we see that if $\dim H_i(\cR/\cJ[\Sigma,\Sigma^{-1}])= i-1$ then there is some nontrivial topology of a lattice fan $\Sigma_{W,\sigma}$ relative to $\Sigma^{-1}_{W,\sigma}$ for a flat $W\in L_{\Sigma,\Sigma^{-1}}$ with $\dim W=i-1$.  This behavior is far from generic, but it is not so difficult to construct examples manifesting such nontrivial topology.  In the following example we provide two fans which illustrate such nongeneric behavior.

\begin{exm}\label{ex:3DSymmetric}
The two polytopal complexes $\PC_1$, $\PC_2$ in Figure~\ref{fig:Hyper} (shown without boundary faces to clarify the inner structure) are both formed by placing a polytope inside of a scaled version of itself and connecting vertices as shown.  In Figure~\ref{fig:HyperTet}, we start with a tetrahedron which is the convex hull of $(0, 0, 8), (-4, -6, -3),$ $(-4, 6, -3),(6, 0, -3)$, then scale it up by a factor of $4$ and place the smaller one inside.  In Figure~\ref{fig:HyperCube} we do the same procedure starting with the cube with vertices $(\pm 1,\pm 1,\pm 1)$.  
\begin{figure}[htp]
\begin{subfigure}[b]{.35\textwidth}
\includegraphics[width=\textwidth]{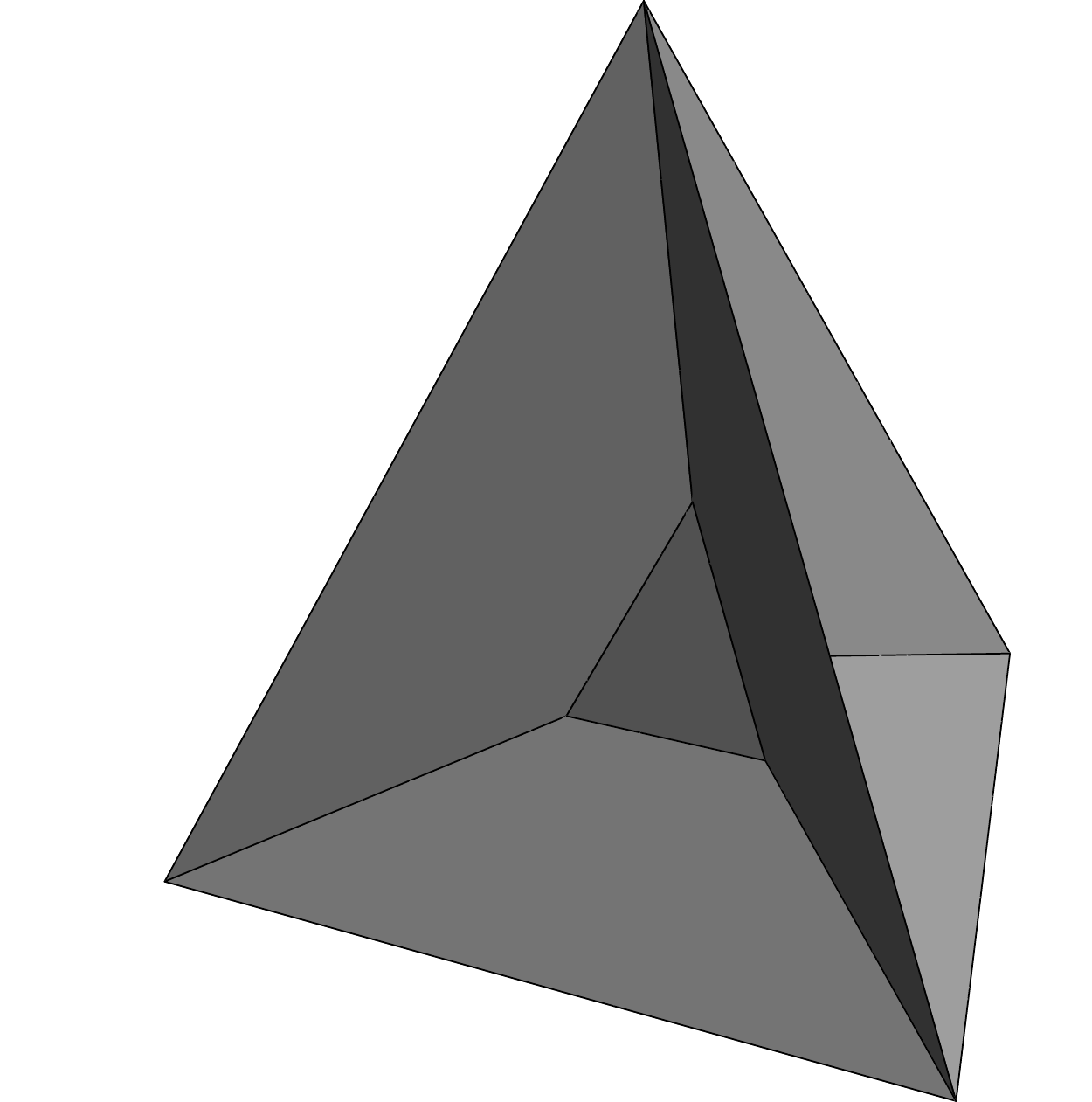}
\caption{$\PC_1$}\label{fig:HyperTet}
\end{subfigure}
\begin{subfigure}[b]{.35\textwidth}
\includegraphics[width=\textwidth]{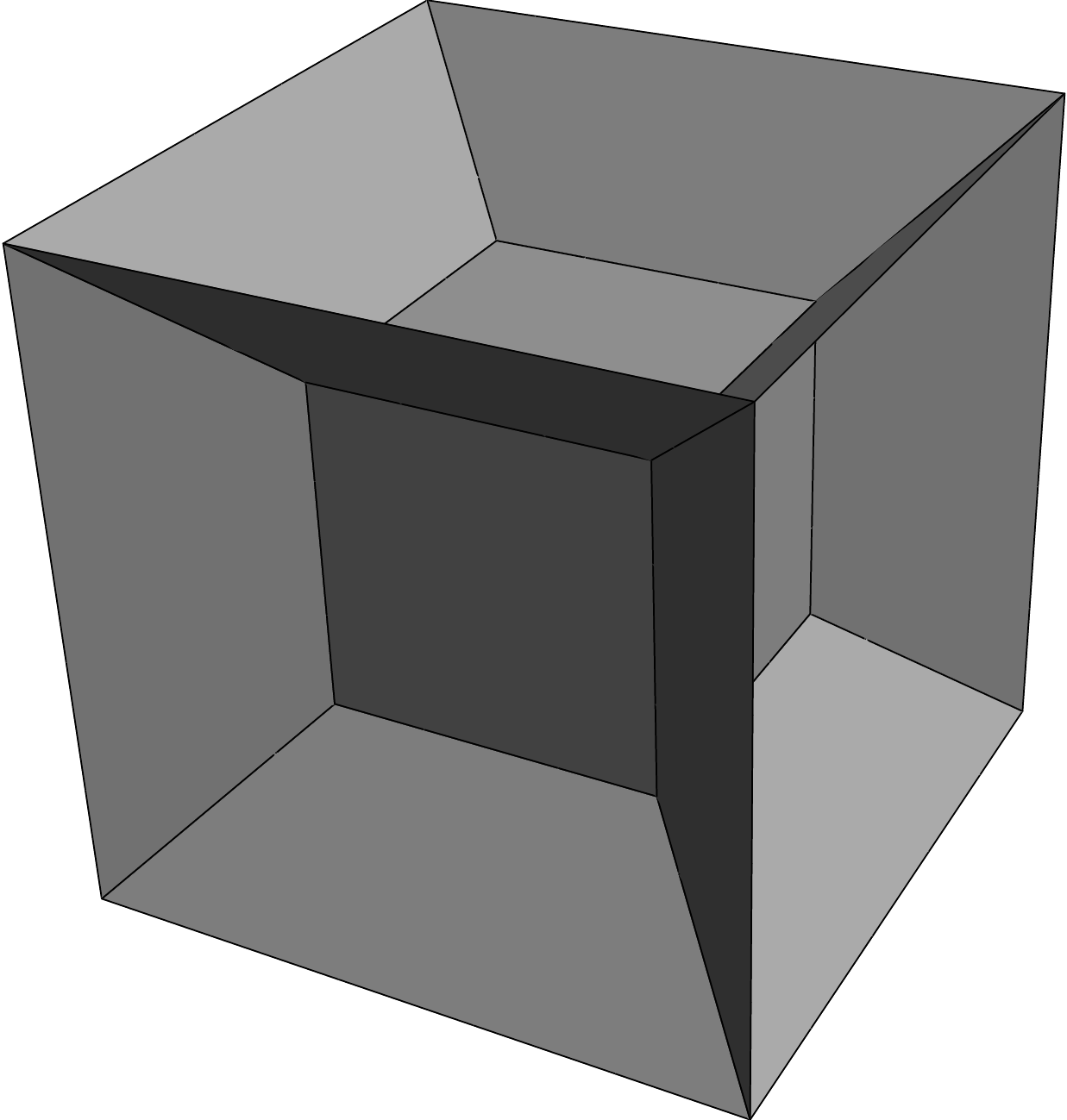}
\caption{$\PC_2$}\label{fig:HyperCube}
\end{subfigure}
\caption{}\label{fig:Hyper}
\end{figure}
Let $\Sigma_1=\wPC_1,\Sigma_2=\wPC_2$.  Take $S=\R[w,x,y,z]$, where $w$ is the cone variable.  If we do not impose any vanishing along the boundaries of $\PC_1,\PC_2$, computations in Macaulay2~\cite{M2} yield the following information about associated primes.    By dimension $-1$ we mean the module vanishes.
\begin{table}[htp]
\centering
\begin{tabular}{c|c|c}
Module & Dimension & Minimal Associated Primes\\
\hline
$H_3(\cR/\cJ[\Sigma_1,\partial\Sigma_1])$ & -1 & None\\
$H_2(\cR/\cJ[\Sigma_1,\partial\Sigma_1])$ & 1 & $(x,y,z)$\\
$H_3(\cR/\cJ[\Sigma_2,\partial\Sigma_2])$ & 2 & $(x,y),(y,z),(x,z)$\\
$H_2(\cR/\cJ[\Sigma_2,\partial\Sigma_2])$ & 1 & $(x,y,z)$
\end{tabular}
\caption{}\label{tbl:AssHyp}
\end{table}

The only information in Table~\ref{tbl:AssHyp} that we cannot deduce from Theorem~\ref{thm:assPrimes2} is the fact that $H_3(\cR/\cJ[\Sigma_1,\partial\Sigma_1])=0$.  If we impose vanishing along all $6$ codimension one boundary faces of $\Sigma_2$, then we obtain three additional codimension two associated primes of $H_3(\cR/\cJ[\Sigma_2])$.

\begin{table}[htp]
\centering
\begin{tabular}{c|c|c}
Module & Dimension & Minimal Associated Primes\\
\hline
$H_3(\cR/\cJ[\Sigma_2])$ & 2 & $(x,y),(y,z),(x,z),(x,w),(y,w),(z,w)$
\end{tabular}
\end{table}

The associated primes $(x,w),(y,w),(z,w)$ correspond to intersections at infinity of the affine spans of the four codimension one faces parallel to the $yz$, $xz$, and $xy$ planes, respectively.  Imposing vanishing on only three of four parallel affine spans will result in losing the corresponding associated prime.  This is easily seen using the graph $G_W((\Sigma_2)_{W,\sigma},(\Sigma^{-1}_2)_{W,\sigma})$, where $W$ is the line at infinity along which these affine spans intersect.
\end{exm}

It is much more difficult to describe associated primes which do not arise from mere topological considerations.  The following example, which we will continue in Section~\ref{sec:Fourth}, is one such.

\begin{exm}\label{ex:3DimMorganScot}
Consider the fan $\Sigma=\wDelta$, where $\Delta$ is the simplicial complex formed by placing an inverted tetrahedron symmetrically within a larger tetrahedron and connecting vertices as in Figure~\ref{fig:3DSimpMorganScot}.  The chosen coordinates for the inner tetrahedron in Figure~\ref{fig:3DSimpMorganScot} are $(0, 0, 8), (-4, -6, -3), (-4, 6, -3), (6, 0, -3)$ for the vertices labelled $0,1,2,3$, respectively.  The vertices of the outer tetrahedron are obtained by multiplying the coordinates of the inner tetrahedron by $-5$.  In this simplicial complex there are $15$ tetrahedra (listed by their vertices): 1234, 1678, 2578, 3568, 4567, 1278, 1368, 1467, 2358, 2457, 3456, 1238, 1346, 1247, 2345.

The important geometric consideration here is that the lines between vertices $0$ and $4$, $1$ and $5$, $2$ and $6$, $3$ and $7$ all intersect at the origin.  This is the three dimensional analogue of an example due to Morgan-Scott~\cite{Morgan} considered by Schenck ~\cite[Example~5.3]{Spect}.
\begin{figure}[htp]
\includegraphics[width=.5\textwidth]{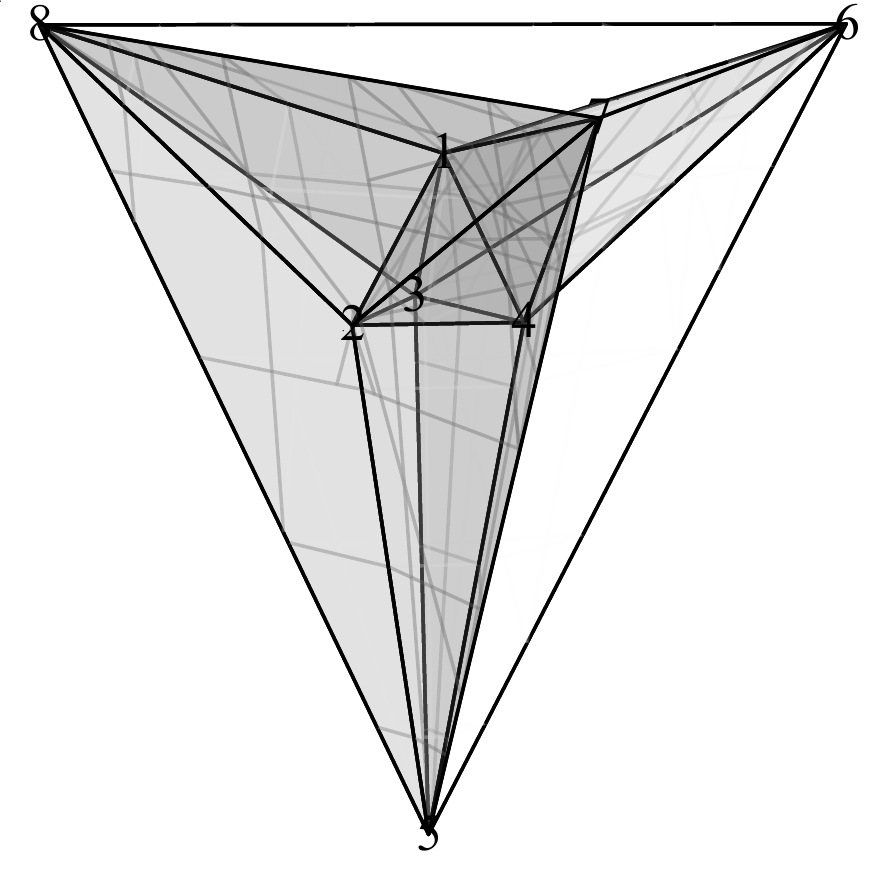}
\caption{Three dimensional Morgan-Scott analogue}\label{fig:3DSimpMorganScot}
\end{figure}

Let us consider the algebra $C^1(\Sigma)$ - recall this means that we assign smoothness parameters $\alpha(\tau)=1$ to every interior codimension one face and impose no vanishing conditions along the boundary, so $\Sigma^{-1}=\partial\Sigma=\widehat{\partial\Delta}$ and $C^1(\Sigma)=H_4(\cR/\cJ[\Sigma,\partial\Sigma])$.  Schenck computes that $H_2(\cR/\cJ[\Sigma,\partial\Sigma])=0$, a computation readily verified in Macaulay2 (in his paper the homological degree is shifted down one from ours).  He also finds that $H_3(\cR/\cJ[\Sigma,\partial\Sigma])$ has associated primes in codimensions three and four.  The associated prime of codimension four is the homogeneous maximal ideal of $S=\R[x_0,x_1,x_2,x_3]$.  By Theorem~\ref{thm:assPrimes1} part $(3)$, the associated primes of codimension three have the form $I(v)$, where $v$ is a ray of $\Sigma$, corresponding to a vertex in $\Delta$.  Indeed, computations in Macaulay2 indicate that there are $8$ associated primes of codimension $3$ and these are precisely the homogeneous ideals of the vertices of $\Delta$.  We will return to this example in Section~\ref{sec:Fourth} to understand how these associated primes contribute to the fourth coefficient of the Hilbert polynomial of $C^1(\Sigma)$.
\end{exm}

\begin{remark}
In general it is quite difficult to analyze $H_3(\cR/\cJ[\Sigma,\Sigma^{-1}])$ for a fan $\Sigma\subset\R^4$.  We will see in Section~\ref{sec:Fourth} that if $\Sigma$ is simplicial then we can deduce the dimension of this module in large degrees if we are able to compute $H_2(\cR/\cJ[\Sigma,\Sigma^{-1}])$ for simplicial $\Sigma\subset\R^3$.  This latter module, while simpler than $H_3(\cR/\cJ[\Sigma,\Sigma^{-1}])$, is still largely not understood.
\end{remark}

The three dimensional analogue of the Morgan-Scott configuration in Example~\ref{ex:3DimMorganScot} gives rise to interesting associated primes in the case of uniform smoothness $r=1$.  One way to mimic a Morgan and Scott example with polytopal complexes is to start with a polytope $\PC$ and fit it symmetrically within the polar polytope $\PC^o$, and connect up vertices belonging to dual faces.

\begin{exm}\label{ex:CubeOct}
Let $\PC$ be the pure $3$-dimensional polytopal complex constructed by starting with an octahedron having vertices $(\pm 1,0,0),(0,\pm 1,0),(0,0,\pm 1)$.  Fit this inside a cube with vertices $(\pm 2,\pm 2,\pm 2)$.  Then let the facets be the inner octahedron, the convex hull of each edge of the octahedron with its dual edge on the cube, and the convex hull of each vertex of the octahedron with its dual cube face, yielding $f_3(\PC)=27$.  Labelling each facet by its vertices, the $20$ tetrahedra are: (1,3,9,13), (1,4,10,14), (2,4,7,11), (2,3,8,12), (1,5,13,14), (4,5,11,14), (2,5,11,12),(3,5,12,13), (1,6,9,10), (4,6,7,10), (2,6,7,8), (3,6,8,9), (1,4,5,14), (2,4,5,11), (2,3,5,12), (1,3,5,13), (1,4,6,10), (2,4,6,7), (2,3,6,8), (1,3,6,9).  There are also $6$ pyramids with square bases: (1,9,10,13,14), (2,7,8,11,12), (3,8,9,12,13), (4,7,10,11,14), (5, 11,12,13,14), (6,7,8,9,10)

\begin{figure}
\centering
\includegraphics[width=.5\textwidth]{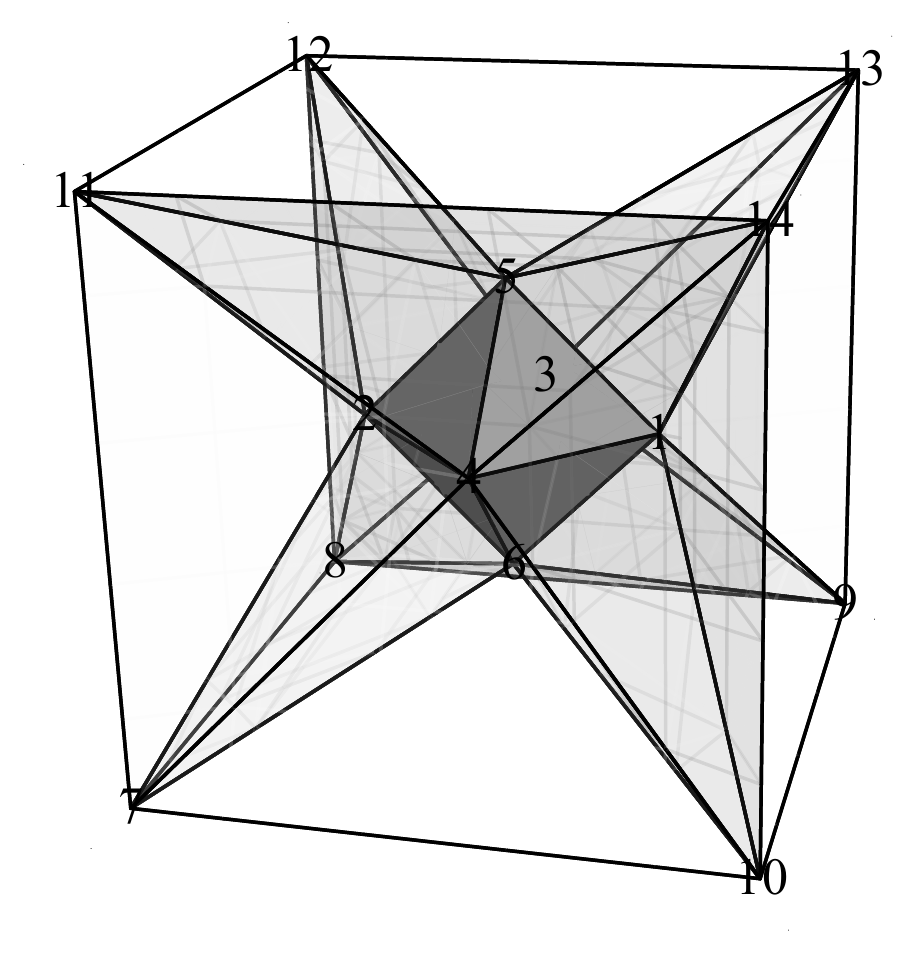}
\caption{$\PC$}
\end{figure}

Let $\Sigma=\wPC$, and consider uniform smoothness with $r=1$.  Set $S=\R[w,x,y,z]$, where $w$ is the cone variable.  Computations in Macaulay2 yield the results of Table~\ref{tbl:AssCubeOct}.

\begin{table}[htp]
\centering
\begin{tabular}{c|c|c}
Module & Dimension & Minimal Associated Primes\\
\hline
$H_3(\cR/\cJ[\Sigma,\partial\Sigma])$ & 1 & Ideals of all vertices, $(x,y,w),(x,z,w),(y,z,w)$\\
\hline
$H_2(\cR/\cJ[\Sigma,\partial\Sigma])$ & -1 & None
\end{tabular}
\caption{}\label{tbl:AssCubeOct}
\end{table}

Like the simplicial three dimensional analogue of the Morgan-Scott configuration, the ideals of the vertices are not something we can see arising in a topological manner.  However, the three additional ideals are of interest and \textit{are} in fact topological.  Since they are all symmetric, consider the ideal $(x,y,w)$.  This is the ideal of the point $w\in L_{\Sigma,\Sigma^{-1}}$ at which the $z$-axis meets the hyperplane at infinity.  The lattice fan $\Sigma_{w,\sigma}$ for any facet $\sigma\in\Sigma_4$ having a codimension one face $\tau$ with $w\in\mbox{supp}(\tau)$, consists of $8$ facets which surround the $z$-axis, namely the cones over the facets with labels (1,9,10,13,14), (1,4,10,14), (4,7,10,11,14), (2,4,7,11), (2,7,8,11,12), (2,3,8,12), (3,8,9,12,13), (1,3,9,13).  Topologically, the pair $(\Sigma_{w,\sigma},\partial\Sigma_{w,\sigma})$ is the cone over a torus $\mathbb{T}^2$ and its boundary.  Via excision this yields $H_3(\cR[\Sigma_{w,\sigma},\partial\Sigma_{w,\sigma}])\cong H_2(\mathbb{T}^2,\partial\mathbb{T}^2;S)=S$, hence via long exact sequences $H_3(\cR/\cJ[\Sigma_{w,\sigma},\partial\Sigma_{w,\sigma}])\cong \dfrac{S}{\sum_{\tau\in\Sigma^0_{w,\sigma}} J(\tau)}$.  This gives us the associated prime $(x,y,w)$.  The others follow analogously.

\begin{remark}\label{rem:Yuz}
The facets of the lattice fan $\Sigma_{w,\sigma}$ form an equivalence class $[\sigma]_w$ which is not present in the equivalence relation defined by Yuzvinsky~\cite[\S~2]{Yuz}.  The reason for this is that the flat $w\in L_{\Sigma,\partial\Sigma}$ cannot be obtained by intersecting affine spans of codimension one faces of any single facet $\sigma\in\Sigma$.
\end{remark}

\end{exm}

\section{Hilbert Polynomials}

In this section we prove Proposition~\ref{prop:HFHomMainTerm}, which is the primary tool for translating our observations on associated primes into computations of Hilbert polynomials.  The following two sections address computations of the third and fourth coefficients of the Hilbert polynomial of $C^\alpha(\Sigma)$.  We begin by summarizing the commutative algebra which we will need.

If $M$ is any $S=\R[x_0,\ldots,x_n]$-module, a \textit{finite free resolution} of $M$ of length $r$ is an exact sequence of free $S$-modules
\[
F_\bullet: \mbox{ }0\rightarrow F_r\xrightarrow{\phi_r} F_{r-1}\xrightarrow{\phi_{r-1}} \cdots \xrightarrow{\phi_1} F_0
\]
such that $\mbox{coker }\phi_1=M$.  The Hilbert syzygy theorem guarantees that $M$ has a finite free resolution.  The \textit{projective dimension} of $M$, denoted $\mbox{pd}(M)$, is the minimum length of a finite free resolution.  If $M$ is a graded $S$-module with $\mbox{pd}(M)=\delta$ then $M$ has a minimal free resolution $F_\bullet\rightarrow M$ of length $\delta$, unique up to graded isomorphism.  This resolution is characterized by the property that the entries of any matrix representing the differentials $\phi.$ in $F_\bullet$ are contained in the homogeneous maximal ideal $(x_0,\ldots,x_n)$.

Recall that if $M$ is a finitely generated nonnegatively graded $S$-module, we may write $M=\bigoplus_{i\ge 0} M_i$, where each $M_i$ is an $\R$-vector space.  The \textit{Hilbert function} of $M$ in degree $d$ is $HF(M,d)=\dim M_d$.  For $d\gg 0$ this agrees with a polynomial called the \textit{Hilbert polynomial} of $M$, denoted $HP(M,d)$.  If $HP(M,d)=0$, then $M_d=0$ for $d\gg 0$.  Such modules are said to have finite length.  If $M$ is a module of finite length, then its \textit{socle degree} is the largest degree $k$ so that $M_k\neq 0$.

The standard use of the complex $\cR/\cJ[\Sigma,\Sigma^{-1}]$ is to compute the dimensions of the vector spaces $C^\alpha(\Sigma)$ via an Euler characteristic computation, which we can state in terms of Hilbert functions as
\[
\sum_{i=0}^{n+1} (-1)^i HF(\cR/\cJ[\Sigma,\Sigma^{-1}]_{n+1-i},d)=\sum_{i=0}^{n+1} (-1)^i HF(H_{n+1-i}(\cR/\cJ[\Sigma,\Sigma^{-1}]),d).
\]
Set
\[
\begin{array}{rl}
\chi(\cR/\cJ[\Sigma,\Sigma^{-1}],d) & =\sum_{i=0}^{n+1} (-1)^i HF(\cR/\cJ[\Sigma,\Sigma^{-1}]_{n+1-i},d)\\
&=\sum\limits_{i=0}^{n+1} (-1)^i\left( \sum\limits_{\gamma\in\Sigma^{\ge 0}_{n+1-i}} HF\left(\dfrac{S}{J(\gamma)},d \right)\right).
\end{array}
\]
Recall from Lemma~\ref{lem:SplinesTop} that $H_{n+1}(\cR/\cJ[\Sigma,\Sigma^{-1}])=C^\alpha(\Sigma)$.  This yields

\begin{multline}\label{eqn:EulerCharacteristic}
HF(C^\alpha(\Sigma),d)=\chi(\cR/\cJ[\Sigma,\Sigma^{-1}],d) \\ - \sum\limits_{i=1}^{n+1}(-1)^i HF(H_{n+1-i}(\cR/\cJ[\Sigma,\Sigma^{-1}]),d).
\end{multline}

Determining $HF(C^\alpha(\Sigma),d)$ from Equation~\ref{eqn:EulerCharacteristic} requires two tasks, both of which are unsolved in general.  The first task is to determine the dimensions of the vector spaces $\left(\dfrac{S}{J(\gamma)}\right)_d$, which are quotients of the polynomial ring by an ideal generated by powers of linear forms.  This is itself a rich field of research with connections to Waring's problem and fat point ideals ~\cite{Ciliberto, Waring}.  In~\cite{Jimmy}, Shan exploits these connections (particularly an algorithm due to Geramita-Harbourne-Migliore~\cite{GHMFatPoint} for computing Hilbert functions of certain fat point ideals) to obtain bounds on $\dim C^2(\Sigma)_d$ for $\Sigma\subset\R^3$.

The second task is to compute $\dim H_i(\cR/\cJ[\Sigma,\Sigma^{-1}])_d$.  There are few tools for dealing with these homology modules.  Mourrain and Villamizar give bounds on the dimension of these homology modules when $\Sigma=\wDelta$, the cone over a simplicial complex $\Delta$, when $\Delta\subset\R^2$ and $\Delta\subset\R^3$ \cite{D2,D3}.  Armed with these bounds and the current knowledge of fat point ideals, they obtain bounds on the dimension of the spline space $C^r(\wDelta)_d$ using Equation~\eqref{eqn:EulerCharacteristic}.

A slightly different approach, taken primarily by Schenck with various co-authors, is to compute the Hilbert polynomial of $C^\alpha(\Sigma)$ using Equation~\eqref{eqn:EulerCharacteristic}~\cite{FatPoints,TSchenck08,Chow}.  This approach, which we also take, ignores information that `eventually vanishes.'  Our first step is to pull out the leading term of the Hilbert polynomial of the homology module $H_i(\cR/\cJ[\Sigma,\Sigma^{-1}])$.  This should be seen as a generalization of  ~\cite[Theorem~3.10]{TSchenck08} and ~\cite[Corollary~2.7]{Chow}.  An important difference is that the aforementioned results only apply when $\dim H_i(\Sigma,\Sigma^{-1})=i-1$, the maximal possible dimension.  In particular, these formulas do not apply to simplicial complexes, where $\dim H_i(\Sigma,\Sigma^{-1})<i-1$ by Theorem~\ref{thm:assPrimes1}.

Recall that $H_i(\cR/\cJ[\Sigma_W,\Sigma^{-1}_W])=\bigoplus_{\sigma\in\Gamma_W} H_i(\cR/\cJ[\Sigma_{W,\sigma},\Sigma^{-1}_{W,\sigma}])$, where $\Gamma_W$ is a set of representatives for the equivalence class of facets modulo the equivalence relation $\sim_W$ of Definition~\ref{def:LatticePair}.

\begin{prop}\label{prop:HFHomMainTerm}
Let $\Sigma\subset\R^{n+1}$ be a pure,hereditary, $(n+1)$-dimensional fan with smoothness parameters $\alpha$ and set $k=\dim H_i(\cR/\cJ[\Sigma,\Sigma^{-1}])$.  Then
\[
HP(H_i(\cR/\cJ[\Sigma,\Sigma^{-1}]),d)= \sum\limits_{\substack{W\in L_{\Sigma,\Sigma^{-1}},\\ \dim(W)=k}} HP(H_i(\cR/\cJ[\Sigma_W,\Sigma^{-1}_W]),d) + O(d^{k-2}).
\]
If $k=1$, $O(d^{k-2})$ is understood to be $0$.
\end{prop}
\begin{proof}
For any $W\in L_{\Sigma,\Sigma^{-1}}$ there is a map of complexes
\[
\cR/\cJ[\Sigma,\Sigma^{-1}]\xrightarrow{q_W} \cR/\cJ[\Sigma_W,\Sigma^{-1}_W],
\]
The right hand side is the quotient of $\cR/\cJ[\Sigma,\Sigma^{-1}]$ by the sub-chain complex $\cR/\cJ[\Sigma^c_W,\Sigma^c_W\cup\Sigma^{-1}]$, where $\Sigma^c_W$ is the subfan of faces whose affine span does not contain $W$.  This descends to a map $\bar{q}_{W,i}$ in homology,
\[
H_i(\cR/\cJ[\Sigma,\Sigma^{-1}])\xrightarrow{\bar{q}_{W,i}} H_i(\cR/\cJ[\Sigma_W,\Sigma^{-1}_W]).
\]
Summing over all $W\in L_{\Sigma,\Sigma^{-1}}$ with $\dim W=k$ and setting $\bar{q}_i=\sum_W \bar{q}_{W,i}$ we obtain:
\begin{equation}\label{eqn:Quotient}
H_i(\cR/\cJ[\Sigma,\Sigma^{-1}])\xrightarrow{\bar{q}_i} \bigoplus\limits_{W} H_i(\cR/\cJ[\Sigma_W,\Sigma^{-1}_W]).
\end{equation}
If $M$ is a graded $S=\R[x_0,\ldots,x_n]$ module with $\dim M=k$, then $HP(M,d)$ has degree $k-1$.  By the additivity of the Hilbert polynomial across exact sequences, we will be done if we can show that the kernel and cokernel of $\bar{q}_i$ both have dimension $\le k-1$.  This in turn will follow if we show
\begin{enumerate}
\item[(A)] The target of $\bar{q}_i$ in ~\eqref{eqn:Quotient} has dimension $k$
\item[(B)] $\bar{q}_i$ becomes an isomorphism under localization at primes of codimension exactly $n+1-k$
\end{enumerate}

We refer to the source of $\bar{q}_i$ in~\eqref{eqn:Quotient} as LHS and the target of $\bar{q}_i$ as RHS.  Suppose that $P$ is an associated prime of RHS.  Then by Proposition~\ref{prop:AssFacts} part (2), $P$ is an associated prime of $H_i(\cR/\cJ[\Sigma_{W,\sigma},\Sigma^{-1}_{W,\sigma}])$ for some $W,\sigma$, with $\dim W=k$.  

Now set $\varGamma=\Sigma_{W,\sigma}$ and $\varGamma^{-1}=\Sigma^{-1}_{W,\sigma}$.  Then $L_{\varGamma,\varGamma^{-1}}$ is the sublattice of $L_{\Sigma,\Sigma^{-1}}$ consisting of the flats 
\[
\{V=\mbox{aff}(\gamma)\in L_{\Sigma,\Sigma^{-1}}| W\in \mbox{supp}(\gamma), \gamma\in\sigma'\mbox{ for some } \sigma'\sim_W\sigma \}.
\]
Furthermore, for any $V\in L_{\varGamma,\varGamma^{-1}}$ and $\sigma\in\varGamma_{n+1}$, $\varGamma_{V,\sigma}=\Sigma_{V,\sigma}$.  By Theorem~\ref{thm:assPrimes1}, $P=I(V)$ for some $V\in L_{\varGamma,\varGamma^{-1}}$.  Lemma~\ref{lem:Localize} yields
\[
\begin{array}{rl}
(H_i(\cR/\cJ[\varGamma,\varGamma^{-1}]))_{I(V)}= & \bigoplus\limits_{\sigma\in \Gamma'_V}(H_i(\cR/\cJ[\varGamma_{V,\sigma},\varGamma^{-1}_{V,\sigma}]))_{I(V)}\\
= & \bigoplus\limits_{\sigma\in \Gamma'_V}(H_i(\cR/\cJ[\Sigma_{V,\sigma},\Sigma^{-1}_{V,\sigma}]))_{I(V)},
\end{array}
\]
where $\Gamma'_V$ runs across representatives of the equivalence classes $[\sigma]_V$ for $\sigma\in\varGamma_{n+1}$.  The final direct sum above appears as a summand of $H_i(\cR/\cJ[\Sigma,\Sigma^{-1}])_{I(V)}$, according to Lemma~\ref{lem:Localize}.  It follows from Proposition~\ref{prop:AssFacts} parts (1) and (2) that $I(V)$ is an associated prime of $H_i(\cR/\cJ[\Sigma,\Sigma^{-1}])$.  Making use of the formula
\[
\dim M=\max\{\dim R/P| P\in \mbox{Ass}(M)\},
\]
we find that $k=\dim H_i(\cR/\cJ[\Sigma,\Sigma^{-1}])\ge \dim V\ge \dim W=k\implies V=W$ and $P=I(W)$.  By dimension considerations, $I(W)$ must be a minimal associated prime of $H_i(\cR/\cJ[\Sigma,\Sigma^{-1}])$.

Thus the associated primes of RHS are precisely the minimal associated primes of LHS, and these are contained in the set of primes $\{I(W)|W\in L_{\Sigma,\Sigma^{-1}}, \dim(W)=k\}$.  It follows immediately that $\dim$ RHS$=k$, proving (A).  To prove (B) we need only show that $\bar{q}_i$ becomes an isomorphism under localization at primes of the form $I(V)$, $\dim V=k$.  By Lemma~\ref{lem:Localize},
\[
H_i(\cR/\cJ[\Sigma,\Sigma^{-1}])_{I(V)}=H_i(\cR/\cJ[\Sigma_V,\Sigma^{-1}_V])_{I(V)}
\]
The summands of RHS in~(2) have the form $H_i(\cR/\cJ[\Sigma_W,\Sigma^{-1}_W])$, where $\dim W=k$.  As we have seen, each of these summands either has dimension less than $k$ or has the unique minimal associated prime $I(W)$.  It follows that a summand of RHS vanishes under localization at $I(V)$ unless it is precisely the summand  $H_i(\cR/\cJ[\Sigma_V,\Sigma^{-1}_V])$.  This completes the proof of (2).
\end{proof}

\section{Third Coefficient of Hilbert Polynomial}\label{sec:Third}
In this section we apply Proposition~\ref{prop:HFHomMainTerm} and Theorem~\ref{thm:assPrimes1} to yield a formula for the third coefficient of the Hilbert polynomial $HP(C^\alpha(\Sigma),d)$ for any assignment of smoothness parameters $\alpha$.  Our approach synthesizes computations from two papers: Geramita and Schenck's computation for planar simplicial complexes with mixed smoothness in~\cite{FatPoints} and McDonald and Schenck's computation of the third coefficient of $HP(C^r(\wPC),d)$ for arbitrary polytopal complexes and uniform smoothness in~\cite{TSchenck08}.  Our main contributions to this story are twofold: we allow arbitrary vanishing conditions to be imposed along codimension one boundary faces, and we connect the third coefficient (in the polytopal case) to the topology of the lattice fans $\Sigma_{W,\sigma}$.

Looking back to Equation~\ref{eqn:EulerCharacteristic}, we see that by dimension considerations there are $4$ terms that will contribute to the first three coefficients of $HP(C^\alpha(\Sigma))$, for $\Sigma\subset\R^{n+1}$ a pure $(n+1)$-dimensional hereditary polyhedral fan.  These are recorded in Table~\ref{tbl:Dims}.

\begin{table}[htp]
\renewcommand{\arraystretch}{2}%
\centering
\begin{tabular}{c|c|c}
Dimension & Module & Hilbert Polynomial \\
\hline
$n+1$ & $S^{f_{n+1}(\Sigma)}$ & $f_{n+1}(\Sigma)\binom{d+n}{n}$ \\
\hline
$n$ & $\bigoplus\limits_{\tau\in\Sigma^{\ge 0}_n}\dfrac{S}{J(\tau)}$ & $\sum\limits_{\tau\in\Sigma^{\ge 0}_n} \binom{d+n}{n}-\binom{d+n-\alpha(\tau)-1}{n}$ \\
\hline
$n-1$ & $\bigoplus\limits_{\gamma\in\Sigma^{\ge 0}_{n-1}}\dfrac{S}{J(\gamma)}$ & ?\\
\hline
$n-1$ & $H_n(\cR/\cJ[\Sigma,\Sigma^{-1}])$ & ?
\end{tabular}
\caption{}\label{tbl:Dims}
\end{table}

The Hilbert polynomials of the first two entries on the table are simple to derive.  The first is well known and the second follows from the fact that $J(\tau)=\langle l_\tau^{\alpha(\tau)+1} \rangle$ and the one-step resolution
\[
S(-\alpha(\tau)-1)\rightarrow S
\]
for $J(\tau)$.  The question marks in the table are resolved by understanding Hilbert functions of ideals of powers of linear forms in two variables, which is the heart of the paper by Geramita and Schenck~\cite{FatPoints}.  We summarize this result, which is obtained using inverse systems to translate the problem into calculating dimensions of ideals of \textit{fat points} in $\mathbb{P}^1$.

\begin{thm}\label{thm:FatPoints}\cite[Theorem~2.7]{FatPoints}
Suppose $\alpha_1,\ldots,\alpha_\mu$ are positive integers, \\
$L_1, \ldots ,L_\mu\in S=\R[x,y]$ are linear forms (not all multiples of the same linear form) and let $J$ be the $(x,y)$-primary ideal minimally generated by $(L_1^{\alpha_1},\ldots,L_\mu^{\alpha_\mu})$.  Let
\[
\Omega=\left\lfloor\dfrac{\sum_{i=1}^\mu \alpha_i -\mu}{\mu-1} \right\rfloor+1.
\]
Then $\Omega-1$ is the socle degree of $S/J$ and the graded minimal free resolution of $J$ has the form
\[
0\rightarrow S(-\Omega-1)^a\oplus S(-\Omega)^{t-1-a} \rightarrow \oplus_{i=1}^\mu S(-\alpha_i) \rightarrow J \rightarrow 0,
\]
where $a=\sum_{i=1}^\mu \alpha_i+(1-\mu)\cdot\Omega$.
\end{thm}

\begin{cor}\label{cor:FatPoints}
Suppose $L_1,\ldots,L_\mu\in S=\R[x_0,\ldots,x_n]$ are linear forms which vanish on a common codimension $2$ linear subspace $W$.  Let $\alpha_1,\ldots,\alpha_\mu$ and $\Omega$ be as defined in Theorem~\ref{thm:FatPoints}, so that $L^{\alpha_1}_1,\ldots,L^{\alpha_\mu}_\mu$ are minimal generators for the ideal the generate.  Then $J$ has minimal free resolution
\[
0\rightarrow S(-\Omega-1)^a\oplus S(-\Omega)^{\mu-1-a} \rightarrow \oplus_{i=1}^t S(-\alpha_i) \rightarrow J \rightarrow 0,
\]
where $a=\sum_{i=1}^\mu \alpha_i+(1-\mu)\cdot\Omega$.
\end{cor}
\begin{proof}
Choose linear forms $l_1,\ldots,l_{n-1}$ which do not vanish on $W$.  These form a regular sequence on $S/J$.  Cutting down by these to the case of Theorem~\ref{thm:FatPoints} yields the result.  The exact statement we need is the following:  let $f\in S_1$ be a linear form which is a nonzerodivisor on $S/J$.  Then $F_\bullet\rightarrow S/J$ is exact if and only if $F_\bullet/fF_\bullet\rightarrow (S/J+(f))$ is exact.  This is an easy consequence of the long exact sequence in homology induced by the short exact sequence $0\rightarrow S(-1)/J\xrightarrow{\cdot f} S/J \rightarrow S/(J+(f))\rightarrow 0$.  Now induct.
\end{proof}

Both question marks in Table~\ref{tbl:Dims} are resolved by applying Corollary~\ref{cor:FatPoints}.  We mainly need some notation to state the results.

Set $\alpha'(\tau)=\alpha(\tau)+1$.  For each $W\in L_{\Sigma,\Sigma^{-1}}$, let $\mu(W,\sigma)$ be the number of minimal generators of the ideal $\langle l^{\alpha'(\tau)}_\tau|\tau\in \left(\Sigma^{\ge 0}_{W,\sigma}\right)_n \rangle$ and $\beta(W,\sigma)=(\alpha'(\tau_1),\ldots,\alpha'(\tau_{\mu(W,\sigma)}))$ the exponent vector for a set of minimal generators.  Set
\[
\Omega(W,\sigma)=\left\lfloor\dfrac{\sum_{i=1}^{\mu(W,\sigma)} \alpha'(\tau_i) -\mu(W,\sigma)}{\mu(W,\sigma)-1} \right\rfloor+1.
\]
Also let $a(W,\sigma)=\sum_{\alpha'(\tau)\in \beta(W,\sigma)} \alpha'(\tau)+(1-\mu(W,\sigma))\cdot\Omega(W,\sigma)$ and $b(W,\sigma)=\mu(W,\sigma)-1-a(W,\sigma)$.  If $\Sigma_{W,\sigma}=\mbox{st}(\gamma)$, then replace $\mu(W,\sigma),\beta(W,\sigma),\Omega(W,\sigma),$ $a(W,\sigma),b(W,\sigma)$ by $\mu(\gamma),\beta(\gamma),\Omega(\gamma),a(\gamma),b(\gamma)$, respectively.  Now we can finish off Table~\ref{tbl:Dims}.

\begin{cor}\label{cor:ThirdCoeff}
Let $\Sigma\subset\R^{n+1}$ be a pure $(n+1)$-dimensional fan with smoothness parameters $\alpha$. Using the notation above, Table~\ref{tbl:Dims} may be completed as in Table~\ref{tbl:Dims2}.
\begin{table}[htp]
\renewcommand{\arraystretch}{2}%
\centering
\begin{tabular}{c|c}
Module & Hilbert Polynomial \\
\hline
$S^{f_{n+1}(\Sigma)}$ & $f_{n+1}(\Sigma)\binom{d+n}{n}$ \\
\hline
$\bigoplus\limits_{\tau\in\Sigma^{\ge 0}_n}\dfrac{S}{J(\tau)}$ & $\sum\limits_{\tau\in\Sigma^{\ge 0}_n} \binom{d+n}{n}-\binom{d+n-\alpha(\tau)-1}{n}$ \\
\hline
$\bigoplus\limits_{\gamma\in\Sigma^{\ge 0}_{n-1}}\dfrac{S}{J(\gamma)}$ & 
$
\begin{array}{c}
\sum\limits_{\gamma\in\Sigma^{\ge 0}_{n-1}} \left( \binom{d+n}{n}-\sum\limits_{\alpha'(\tau)\in\beta(\gamma)} \binom{d+n-\alpha'(\tau)}{n} \right. \\ \left. +a(\gamma)\binom{d+n-\Omega(\gamma)-1}{n}+b(\gamma)\binom{d+n-\Omega(\gamma)}{n} \vphantom{\sum\limits_{\alpha'(\tau)\in\beta(W,\sigma)}} \right)
\end{array}
$
\\[10 pt]
\hline
$H_n(\cR/\cJ[\Sigma,\Sigma^{-1}])$ & 
$
\begin{array}{c}
\sum\limits_{\substack{W\in L_{\Sigma,\Sigma^{-1}}\\ \dim W=n-1}} \sum\limits_{\substack{\sigma\in\Gamma_W\\ H_{n-1}(\cR[\Sigma_{W,\sigma},\Sigma_{W,\sigma}^{-1}])\neq 0}}\left( \binom{d+n}{n}  \vphantom{\sum\limits_{\alpha'(\tau)\in\beta(W,\sigma)}} \right. \\
 -\sum\limits_{\alpha'(\tau)\in\beta(W,\sigma)} \binom{d+n-\alpha'(\tau)}{n}
 +a(W,\sigma)\binom{d+n-\Omega(W,\sigma)-1}{n} \\ 
 \left. +b(W,\sigma)\binom{d+n-\Omega(W,\sigma)}{n} \vphantom{\sum\limits_{\alpha'(\tau)\in\beta(W,\sigma)}} \right) + O(d^{n-3})
\end{array}
$
\end{tabular}
\caption{}\label{tbl:Dims2}
\end{table}

\end{cor}

\begin{proof}
The third entry is a direct application of Corollary~\ref{cor:FatPoints} to the quotients $S/J(\gamma)$.  By Theorem~\ref{thm:assPrimes2} and Proposition~\ref{prop:HFHomMainTerm},
\begin{align*}
HP(H_{n-1}(\cR/\cJ[\Sigma,\Sigma^{-1}]),d)= & \sum\limits_{\substack{W\in L_{\Sigma,\Sigma^{-1}}, \dim W=n-1,\sigma\in\Gamma_W\\ H_{n-1}(\cR[\Sigma_{W,\sigma},\Sigma_{W,\sigma}^{-1}])\neq 0}} HP\left(\dfrac{S}{\sum_{\tau\in(\Sigma^{\ge 0}_{W,\sigma})_n} J(\tau)},d \right)\\
& + O(d^{n-3}).
\end{align*}
Recognizing $\sum_{\tau\in(\Sigma^{\ge 0}_{W,\sigma})_n} J(\tau)$ as the ideal $\langle l^{\alpha'(\tau)}_\tau|\tau\in \left(\Sigma^{\ge 0}_{W,\sigma}\right)_n \rangle$, we are done.
\end{proof}

At this point we can extract the first three coefficients of $HP(C^\alpha(\Sigma),d)$ by taking the appropriate alternating sums of the expressions in Corollary~\ref{cor:ThirdCoeff}.  Instead of doing this in full generality, we restrict to the case where $\Sigma\subset\R^3$, in which case we recover the full Hilbert polynomial.  The second and third entries in the table above give the constant term.

\begin{cor}\label{cor:HilbertPoly3DFan}
Let $\Sigma\subset\R^3$ be a pure, hereditary, $3$-dimensional fan with smoothness parameters $\alpha$.  Then the Hilbert polynomial $HP(C^\alpha(\Sigma),d)$ is given by
\[
f_{n+1}(\Sigma)\dbinom{d+2}{2}-\left(\sum\limits_{\tau\in\Sigma^{\ge 0}_2} \binom{d+2}{2}-\binom{d+2-\alpha(\tau)-1}{2}\right)+ \sum\limits_{\substack{W\in L_{\Sigma,\Sigma^{-1}}\\ \dim W=1}} c_W,
\]
where
\[
c_W=\sum_{\sigma\in\Gamma_W} c_{W,\sigma},
\]
and
\[
c_{W,\sigma}=\left\lbrace
\begin{array}{l}
1-\sum\limits_{\alpha'(\tau)\in\beta(\gamma)} \binom{\alpha'(\tau)-1}{2} + a(\gamma)\binom{\Omega(\gamma)}{2}+b(\gamma)\binom{\Omega(\gamma)-1}{2} \\
\hfill \mbox{ if } \Sigma_{W,\sigma}=\mbox{st}(\gamma),\gamma\in\Sigma^{\ge 0}_1 \\[10 pt]
1-\sum\limits_{\alpha'(\tau)\in\beta(W,\sigma)} \binom{\alpha'(\tau)-1}{2} + a(W,\sigma)\binom{\Omega(W,\sigma)}{2}+b(W,\sigma)\binom{\Omega(W,\sigma)-1}{2} \\
\hfill \mbox{ if } H_1(\cR[\Sigma_{W,\sigma},\Sigma^{-1}_{W,\sigma}])\neq 0\\[10 pt]
0 \hfill \mbox{ otherwise}
\end{array}
\right.
\]
\end{cor}

\begin{exm}
Let $\Sigma$ be the cone over the Schlegel diagram of a cube (Example~\ref{ex:SC}).  Impose vanishing of order $r$ along interior codimension one faces and vanishing of order $s$ along boundary codimension one faces.  If $s=-1$, i.e. no vanishing is imposed along codimension one boundary faces, then the only lattice fan with nontrivial $H_1$ is the lattice fan $\Sigma_{W,\sigma}$ where $W$ is the $z$-axis, shown in Figure~\ref{fig:SCCentralLinePF}.  The ideal generated by the forms $l_\tau$, $\tau\in\Sigma^{\ge 0}_1$, vanishing on $W$ is a complete intersection with two generators in degree $r+1$.  We have $\Omega(W,\sigma)=2r+1$, $a(W,\sigma)=1$, $b(W,\sigma)=0$, hence $c_{W,\sigma}=1-2\binom{r}{2}+\binom{2r+1}{2}$.  The ideal of every interior ray $\gamma$ of $\Sigma$ is generated by three forms of degree $r+1$, so $\Omega(\gamma)=\lfloor 3r/2 \rfloor +1$, $a(\gamma)=3r+3-2\lfloor (3r+2)/2 \rfloor, b(\gamma)=2\lfloor (3r+2)/2 \rfloor -3r-1$.  Using Corollary~\ref{cor:HilbertPoly3DFan} and simplifying, we have
\[
\begin{array}{rl}
HP(C^\alpha(\Sigma),d)=& \frac{5}{2}  d^2+\left(-8 r-\frac{1}{2}\right)d \\
 &-4 \left\lfloor \frac{3
   r}{2}\right\rfloor ^2+12 r \left\lfloor \frac{3 r}{2}\right\rfloor
   -r^2+4 r+2
\end{array}
\]
Now suppose $s\ge 0$, so vanishing of degree $s$ is imposed along the boundary of $\Sigma$.  In addition to the lattice fan around the $z$-axis, there are two others corresponding to the $x$ and $y$-axes (see Figure~\ref{fig:SCYLine}) for which $H_1(\cR[\Sigma,\Sigma^{-1}])$ is nontrivial.  The corresponding ideals are generated by two forms of degree $s+1$ and two forms of degree $r+1$.  For each of these we have $\Omega=\lfloor 2(r+s)/3 \rfloor+1$, $a=2(r+s)+1-3\lfloor 2(r+s)/3 \rfloor$, $b=3\lfloor 2(r+s)/3 \rfloor-2(r+s)+2$.  In addition to the interior rays, we also must incorporate the four boundary rays of $\Sigma_1$.  The corresponding ideals of these rays are generated by two forms of degree $s+1$ and a form of degree $r+1$.  For each of these we have $\Omega=\lfloor (2s+r)/2 \rfloor$, $a=2s+r+1-2\lfloor (2s+r)/2 \rfloor$, $b=1-r-2s+2\lfloor (2s+r)/2 \rfloor$.  Using Corollary~\ref{cor:HilbertPoly3DFan} and simplifying, we have
\[
\begin{array}{rl}
HP(C^\alpha(\Sigma),d)= & \frac{5}{2} d^2+\left(-8 r-4 s-\frac{9}{2}\right) d \\
&-3 \left\lfloor \frac{2(r+s)}{3}\right\rfloor ^2+4 r \left\lfloor \frac{2
   (r+s)}{3}\right\rfloor +4 s \left\lfloor \frac{2 (r+s)}{3}\right\rfloor
   -\left\lfloor \frac{2 (r+s)}{3}\right\rfloor \\
   & - 4 \left\lfloor \frac{r}{2}\right\rfloor ^2-4 \left\lfloor \frac{3 r}{2}\right\rfloor
   ^2+4 r \left\lfloor \frac{r}{2}\right\rfloor +12 r \left\lfloor \frac{3
   r}{2}\right\rfloor\\
   & -5 r^2+4 r s+8 r+4 s+4
\end{array}
\]
This formula is correct when the number of generators of the ideals above are as indicated in the preceding paragraph.  When one of $r,s$ is small compared to the other, then the number of minimal generators of the above ideals may drop, which will change the formula.  For instance, when $r=3$ and $s=0$, the above formula gives a constant term of $81$, while the actual constant is $87$.  This is because the minimal number of generators of several of the ideals drops for these values.
\end{exm}

\begin{remark}
In~\cite{D2}, Mourrain and Villamizar bound the dimension of the homology module $H_2(\cR/\cJ[\Sigma,\partial\Sigma])$ in the case of uniform smoothness, where $\Sigma=\wDelta$ for $\Delta\subset\R^2$ a simplicial complex.  In the simplicial case this module has finite length and so vanishes in high degree.  Their are no known formulas for $\dim H_2(\cR/\cJ[\Sigma,\partial\Sigma])_d$ in small degrees $d$.  We revisit this module in the next section.
\end{remark}

\begin{remark}
In another paper we address the question of how large $d$ must be in order for the formula in Corollary~\ref{cor:HilbertPoly3DFan} to hold.  We obtain a combinatorial bound for such $d$ which is valid for any fan $\Sigma\subset\R^3$.  This is the first such bound obtained for arbitrary polyhedral fans.  In the simplicial case we can tighten this bound and, reducing to the case of uniform smoothness, recover the $3r+2$ bound of Ibrahim and Schumaker~\cite{SuperSpline}.  This is the best known bound for arbitrary planar simplicial complexes, although Alfeld and Schumaker have reduced this bound to $3r+1$ in the case of \textit{generic} simplicial complexes~\cite{AS3r}.
\end{remark}

\section{Simplicial Fourth Coefficient and the Generic Dimension of $C^1$ Tetrahedral Splines}\label{sec:Fourth}
In this section we consider the computation of the Hilbert polynomial of $C^r(\Sigma)$, where $\Sigma=\wDelta$ and $\Delta\subset\R^3$.  We then revisit the computation by Alfeld, Schumaker, and Whiteley of $\dim C^1(\wDelta)_d$ for $d\ge 8$~\cite{ASWTet}.

As in the previous section, we start by describing the modules of relevant dimension for the computation of $HP(C^r(\Sigma),d)$, referring back to Equation~\ref{eqn:EulerCharacteristic}.  We leave out $H_2(\cR/\cJ[\Sigma,\Sigma^{-1}])$ since by Theorem~\ref{thm:assPrimes1} it has finite length and will not contribute to the Hilbert polynomial.

\begin{table}[htp]
\renewcommand{\arraystretch}{2}%
\centering
\begin{tabular}{c|c|c}
Dimension & Module & Hilbert Polynomial \\
\hline
$4$ & $S^{f_4(\Sigma)}$ & $f_4(\Sigma)\binom{d+3}{3}$ \\
\hline
$3$ & $\bigoplus\limits_{\tau\in\Sigma^{\ge 0}_3}\dfrac{S}{J(\tau)}$ & $\sum\limits_{\tau\in\Sigma^{\ge 0}_3} \binom{d+3}{3}-\binom{d+3-\alpha(\tau)-1}{3}$ \\
\hline
$2$ & $\bigoplus\limits_{e\in\Sigma^{\ge 0}_2}\dfrac{S}{J(e)}$ & see Table~\ref{tbl:Dims2}\\
\hline
$1$ & $\bigoplus\limits_{v\in\Sigma^{\ge 0}_1}\dfrac{S}{J(v)}$ & ? \\
\hline
$1$ & $H_3(\cR/\cJ[\Sigma,\Sigma^{-1}])$ & ?
\end{tabular}
\caption{}\label{tbl:Dims3D}
\end{table}

If we approach the first question mark appearing in Table~\ref{tbl:Dims3D} as we did ideals of codimension $2$ faces in the previous section, we would transfer the problem to one of computing Hilbert functions of ideals of fat points in $\mathbb{P}^2$.  We refer the reader to~\cite{D3}, where Mourrain and Villamizar bound the dimension of this piece using the Fr\"oberg sequence in the case of uniform smoothness.  Our main contribution to this story is to elucidate the term $H_3(\cR/\cJ[\Sigma,\Sigma^{-1}])$, the second question mark in Table~\ref{tbl:Dims3D}.  

\begin{prop}\label{prop:SimplicialH3MainTerm}
Let $\Delta\subset\R^3$ be a simplicial complex with smoothness parameters $\alpha$,$\Sigma=\wDelta$, and for $v\in\Delta_0$ let $\Delta_v\subset\R^3$ be the fan with smoothness parameters as in Lemma~\ref{lem:StarProjection}.  Let $S=\R[x_0,x_1,x_2,x_3]$ and $R=\R[x_1,x_2,x_3]$.  Then $HP(H_3(\cR/\cJ[\Sigma,\Sigma^{-1}]),d)$ is the constant given by
\[
HP(H_3(\cR/\cJ[\Sigma,\Sigma^{-1}]),d)=\sum\limits_{\substack{v\in \Delta_0\\ \widehat{v}\in L_{\Sigma,\Sigma^{-1}}}} \sum_{i\ge 0}\dim H_2(\cR/\cJ[\Delta_v,\Delta_v^{-1}])_i.
\]
\end{prop}
Note that the sum $\sum_{i\ge 0}\dim H_2(\cR/\cJ[\Delta_v,\Delta_v^{-1}])_i$ is finite since $H_2(\cR/\cJ[\Delta_v,\Delta_v^{-1}])$ is a module of finite length by Theorem~\ref{thm:assPrimes1}.
\begin{proof}
Applying Proposition~\ref{prop:HFHomMainTerm} yields
\[
HP(H_3(\cR/\cJ[\Sigma,\Sigma^{-1}]),d)=\sum\limits_{\substack{v\in L_{\Sigma,\Sigma^{-1}}\\ \dim(v)=1}} HP(H_3(\cR/\cJ[\Sigma_v,\Sigma_v^{-1}]),d),
\]
where $\Sigma_v=\bigsqcup_{\sigma\in\Gamma_v} \Sigma_{v,\sigma}$.  By Lemma~\ref{lem:FanStar}, $\Sigma_{v,\sigma}=\mbox{st}_\Sigma(\gamma)$ for some face $\gamma\in\Sigma$.  If $\dim(\gamma)\ge 2$, then $\Sigma_{v,\sigma}=\Sigma_{W,\sigma}$, where $W=\mbox{aff}(\gamma)$.  Then $H_3(\cR/\cJ[\Sigma_{v,\sigma},\Sigma^{-1}_{v,\sigma}])=H_3(\cR/\cJ[\Sigma_{W,\sigma},\Sigma^{-1}_{W,\sigma})=0$ by the same as in the proof of Theorem~\ref{thm:assPrimes1}.  So we have
\[
HP(H_3(\cR/\cJ[\Sigma,\Sigma^{-1}]))=\sum\limits_{\substack{v\in \Delta_0\\ \widehat{v}\in L_{\Sigma,\Sigma^{-1}}}} HP(H_3(\cR/\cJ[\mbox{st}_\Sigma(\widehat{v}),\mbox{st}_\Sigma(\widehat{v})^{-1}]),d).
\]
Since $\mbox{st}_\Sigma(\widehat{v})=\widehat{\mbox{st}_\Delta(v)}$, we conclude by Lemma~\ref{lem:StarProjection} that
\[
\cR/\cJ[\mbox{st}_\Sigma(\widehat{v}),\mbox{st}_\Sigma(\widehat{v})^{-1}]=\cR/\cJ[\Delta_v,\Delta^{-1}_v](-1)\otimes_R S.
\]
In particular, 
\[
H_3(\cR/\cJ[\mbox{st}_\Sigma(\widehat{v}),\mbox{st}_\Sigma(\widehat{v})^{-1}])=H_2(\cR/\cJ[\Delta_v,\Delta^{-1}_v])\otimes_R S.
\]
Now the result follows, since, if $M$ is a finitely generated graded $R$-module, there is an isomorphism between the vector spaces $\bigoplus_{i\le d} M_i$ and $(M\otimes S)_d.$
\end{proof}

\begin{cor}\label{cor:3DSimplicialHilbert}
Let $\Delta\subset\R^3$ be a pure, hereditary simplicial complex, with smoothness parameters $\alpha$, and let $\Sigma=\wDelta$.  Set
\[
\begin{array}{rl}
\chi(\cR/\cJ[\Sigma,\Sigma^{-1}],d)= & \sum\limits_{i=0}^4 (-1)^i\left( \sum\limits_{\gamma\in\Sigma^{\ge 0}_{4-i}} HF\left(\dfrac{S}{J(\gamma)},d \right)\right)\\
\chi_H(\cR/\cJ[\Sigma,\Sigma^{-1}],d)= & \sum\limits_{i=0}^3 (-1)^i\left( \sum\limits_{\gamma\in\Sigma^{\ge 0}_{4-i}} HP\left(\dfrac{S}{J(\gamma)},d \right)\right)
\end{array}
\]
Then
\[
HP(C^\alpha(\Sigma),d)=\chi_H(\cR/\cJ[\Sigma,\Sigma^{-1}],d)+C,
\]
where
\[
\begin{array}{rl}
C= & \sum\limits_{\substack{v\in \Delta_0\\ \widehat{v}\in L_{\Sigma,\Sigma^{-1}}}} \sum_{i\ge 0}\dim H_2(\cR/\cJ[\Delta_v,\Delta_v^{-1}])_i\\
= & \sum\limits_{\substack{v\in \Delta_0\\ \widehat{v}\in L_{\Sigma,\Sigma^{-1}}}} \sum_{i\ge 0} (\dim C^\alpha(\Delta_v)_i-\chi(\cR/\cJ[\Delta_v,\Delta^{-1}_v],i))
\end{array}
\]
\end{cor}
\begin{proof}
Write out Equation~\ref{eqn:EulerCharacteristic} for $\Sigma$, applying Proposition~\ref{prop:SimplicialH3MainTerm} to the term $H_2(\cR/\cJ[\Sigma,\Sigma^{-1}])$.  To get the final equality, apply Equation~\ref{eqn:EulerCharacteristic} to $\Delta_v$, yielding 
\[
\dim H_2(\cR/\cJ[\Delta_v,\Delta_v^{-1}])_i=\dim C^\alpha(\Delta_v)_i-\chi(\cR/\cJ[\Delta_v,\Delta^{-1}_v],i).
\]
\end{proof}

\begin{remark}
Corollary~\ref{cor:3DSimplicialHilbert} makes precise the well known fact that, in order to compute $\dim C^\alpha(\wDelta)_d$ for $\Delta\subset\R^3$, even for $d\gg 0$, one must know $\dim C^\alpha(\Sigma)_d$ for $\Sigma\subset\R^3$ arbitrary simplicial fans and all $d$.  See for instance~\cite[Remark~$65$]{ASWTet}.
\end{remark}

\begin{remark}
In the case of uniform parameters and $\Sigma\subset\R^3$ a simplicial fan, $\dim C^r(\Sigma)_d-\chi(\cR[\Sigma,\partial\Sigma],d)=0$ for all $d$ iff $C^r(\Sigma)$ is a free module over the polynomial ring in $3$ variables~\cite[Corollary~4.2]{LCoho}.  Hence we see that $C=0$ in Corollary~\ref{cor:3DSimplicialHilbert} characterizes when the sheaf associated to $C^r(\wDelta)$ is a vector bundle over $\mathbb{P}^3$.  This is the first nontrivial instance of a general characterization due to Schenck and Stiller~\cite[Theorem~3.1]{CohVan}.
\end{remark}

To understand the significance of the constant appearing in Corollary~\ref{prop:SimplicialH3MainTerm}, we return to $C^1(\Sigma)$ for the fan $\Sigma=\wDelta$ from Example~\ref{ex:3DimMorganScot}.

\begin{exm}\label{ex:3DimMorganScot2}
Referring back to Example~\ref{ex:3DimMorganScot}, let $w$ denote the vertex $(0,0,8)$, labelled with a $0$ in Figure~\ref{fig:3DSimpMorganScot}, so $\widehat{w}$ denotes the cone in $\R^4$ over this vertex.  Let $X\subset\R^3$ denote the cone over $\mbox{st}_\Delta(w)$, shown in Figure~\ref{fig:3DSimpMorganScotStar}.  Note that $X^{-1}=\partial X$ since we are considering uniform smoothness.

\begin{figure}[htp]
\includegraphics[width=.5\textwidth]{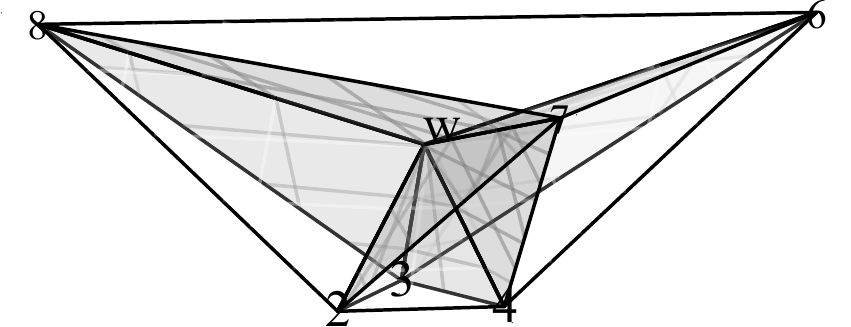}
\caption{Star of the vertex $w$}\label{fig:3DSimpMorganScotStar}
\end{figure}

Localizing $\cR/\cJ[\Sigma,\partial\Sigma]$ at $I(\widehat{w})$ yields
\[
\cR/\cJ[\Sigma,\partial\Sigma]_{I(\widehat{w})}=\cR/\cJ[X,\partial X]_{I(\widehat{w})}
\]
by Lemma~\ref{lem:Localize}.  By Lemma~\ref{lem:StarProjection},
\[
\cR/\cJ[X,\partial X]=\cR/\cJ[\Delta_w,\partial\Delta_w](-1)\otimes_R S,
\]
where $R=\R[x_1,x_2,x_3]$ is the polynomial ring in $3$ variables corresponding to the inclusion of $\R^3$ into $\R^4$ as the hyperplane $x_0=0$, and $\Delta_w$ is obtained by translating $\mbox{st}_\Delta(w)$ to the origin and taking the positive hull of each facet containing $w$.  Since $\partial\Delta_w=\emptyset$ we have
\[
H_3(\cR/\cJ[X,\partial X])=H_2(\cR/\cJ[\Delta_w])\otimes_R S.
\]
By Proposition~\ref{prop:AssFacts}, $I(\widehat{w})$ is associated to $H_3(\cR/\cJ[\Sigma,\partial\Sigma])$ if and only if the homogeneous maximal ideal is associated to $H_2(\cR/\cJ[\Delta_w])$.  This is true iff $H_2(\cR/\cJ[\Delta_w])\neq 0$, since by Theorem~\ref{thm:assPrimes1} the maximal ideal of $R$ is the only ideal that can be associated to $H_2(\cR/\cJ[\Delta_w,\partial\Delta_w])$.

Since $\PC(\Delta_w)$ relative to $\lk(\Delta_w)$ has the topology of a $3$-sphere, $H_2(\cR[\Delta_w])=H_1(\cR[\Delta_w])=0$.  From the long exact sequence in homology corresponding to
\[
0\rightarrow \cJ[\Delta_w] \rightarrow \cR[\Delta_w] \rightarrow \cR/\cJ[\Delta_w] \rightarrow 0
\]
we see that $H_2(\cR/\cJ[\Delta_w])\cong H_1(\cJ[\Delta_w])$.  Since $f_2(\Delta_w)=12$ and $f_1(\Delta_w)=6$, $\cJ[\Delta_w]$ has the form (nonzero in homological degrees $0,1,2$):
\[
0 \rightarrow \bigoplus\limits_{i=1}^{12} J(\tau_i) \rightarrow \bigoplus_{j=1}^6 J(e_j) \rightarrow J(v) \rightarrow 0.
\]
Each of the ideals in this sequence are generated in degree two.  In particular, we have isomorphisms (for all $i,j$)
\[
\begin{array}{rl}
J(\tau_i) & \cong\langle x^2 \rangle\\
J(e_j) & \cong \langle x,y \rangle^2\\
J(v) & \cong \langle x,y,z \rangle^2.
\end{array}
\]
In fact, using the arguments in the proof Corollary~\ref{cor:deg2} we can also show that $H_2(\cJ[\Delta_w])$ is generated in degree two, so it is of particular interest to examine $\cJ[\Delta_w]$ in degree two.

From the isomorphisms above, $\dim J(\tau_i)_2=2,\dim J(e_j)_2=3$, and $\dim J(v)_2=6$.  Hence $\cJ[\Delta_w]_2$ has the form
\[
0\rightarrow \R^{12} \xrightarrow{\delta_2} (\R^3)^6 \xrightarrow{\delta_1} \R^6 \rightarrow 0.
\]
Taking the Euler characteristic yields that the alternating sums of homologies is $0$, so
\[
\dim H_2(\cJ[\Delta_w])_2=\dim H_1(\cJ[\Delta_w])_2.
\]
But the homology $H_2(\cJ[\Delta_w])_2$ can be identified with nontrivial splines on $\Delta_w$ of degree $2$.  It follows that $H_1(\cJ[\Delta_w])\neq 0$ precisely when there is an `unexpected' nontrivial spline on $\Delta_w$ of degree $2$, which is indeed the case.  We check in Macaulay2 that
\[
\dim C^1(\Delta_w)_d -\chi(\cR/\cJ[\Delta_w,\partial\Delta_w],d)=\left\lbrace
\begin{array}{lr}
1 & d=2\\
0 & d\neq 2
\end{array}
\right.
\]
This is the same at each of the eight vertices of $\Delta$.  Hence, by Corollary~\ref{cor:3DSimplicialHilbert} we arrive at the conclusion
\[
\begin{array}{rl}
HP(C^1(\Sigma),d) & =\chi(\cR/\cJ[\Sigma,\partial\Sigma],d)+8\\
&=\dfrac{5}{2}d^3 - 13d^2 + \dfrac{51}{2}d - 3
\end{array}
\]
We emphasize that, unlike the two dimensional case, $HP(C^r(\wDelta),d)$ may not be either an upper or a lower bound for $HF(C^r(\wDelta),d)=\dim C^r_d(\Delta)$ in low degree.  However, for $d\gg 0$, $HF(C^r(\wDelta),d)=HP(C^r(\wDelta),d)$.  For our current example, here is a table of values computed in Macaulay2.

\begin{center}
\begin{tabular}{c|c|c}
$d$ & $HF(C^1(\wDelta),d)$ & $HP(\wDelta,d)=\frac{5}{2}d^3 - 13d^2 + \frac{51}{2}d - 3$\\
\hline
1 & 4 & 12\\
2 & 11 & 16\\
3 & 25 & 24\\
4 & 54 & 51\\
5 & 113 & 112\\
6 & 222 & 222\\
7 & 396 & 396\\
\end{tabular}
\end{center}

\end{exm}

We conclude, as promised, by computing $\dim C^1(\wDelta)_d$ for $\Delta\subset\R^3$ generic and $d\gg 0$.  This formula was shown by Alfeld, Schumaker, and Whiteley to hold for $d\ge 8$~\cite{ASWTet}.  For simplicity, set $f_i(\Delta)=f_i,f^0_i(\Delta)=f^0_i$.

\begin{thm}\label{thm:GenericC1}
Suppose $\Delta\subset\R^3$ is a generic triangulation of a $3$-ball.  Then, for $d\gg 0$,
\[
\dim C^1(\wDelta)_d=\chi(\cR/\cJ[\wDelta,\partial\wDelta],d)=f_3\binom{d+3}{3}-f^0_2(d+1)^2+f^0_1(3d+1)-4f^0_0
\]
\end{thm}

In~\cite[Remark~$5$]{ASWTet}, the authors note that Theorem~\ref{thm:GenericC1} can be derived by a simple heuristic, which is in fact the computation of $\chi_H(\cR/\cJ[\wDelta,\partial\wDelta],d)$.  We make this heuristic argument rigorous by showing that the one relevant homology module vanishes in large degree.  There are two key steps.  First, we use Corollary~\ref{cor:3DSimplicialHilbert} to reduce the argument to three dimensional simplicial fans.  Second, for a generic three dimensional simplicial fan $\Sigma$, we must show that $H_2(\cR/\cJ[\Sigma,\partial\Sigma])=0$, or equivalently show that $\dim C^1(\Sigma)_d=\chi(\cR/\cJ[\Sigma,\partial\Sigma],d)$.  This is accomplished in ~\cite[Corollaries~40,41]{ASWTet} using projections of generalized triangulations.  We accomplish this by using some homological algebra to reduce to the case of noncomplete fans, where the methods of Whiteley~\cite{WhiteleyComb} can be applied directly.  We postpone the proof of Theorem~\ref{thm:GenericC1} until we have accomplished this second step.

If $\Sigma\subset\R^3$ has the form $\Delta_w$ for $w\in\Delta$, where $\Delta$ triangulates a three-ball, then $\lk(\Sigma)$ is contractible.  In this case, 
\[
H_2(\cR/\cJ[\Sigma,\partial\Sigma])\cong H_1(\cJ[\Sigma,\partial\Sigma]),
\]
so we will work with $H_1(\cJ[\Sigma,\partial\Sigma])$.

If the union of the cones of $\Sigma$ is $\R^3$, $\Sigma$ is called \textit{complete}.  We use the following description for $H_1(\cJ[\Sigma,\partial\Sigma])$, which is the analog of~\cite[Lemma~3.8]{LCoho} for complete fans.

\begin{lem}\label{lem:H1Jpres}
Let $\Sigma\subset\R^3$ be a hereditary, complete fan.  Define $K^r\subset \bigoplus\limits_{\tau\in\Sigma_2} S(-r-1)$ by
\[
K^r=\{\sum_{v\in\tau} a_\tau e_\tau|v\in\Sigma_1, \sum a_\tau l^{r+1}_\tau=0\}.
\]
Also define $V^r\subset \bigoplus\limits_{\tau\in\Sigma_2} S(-r-1)$ by
\[
V^r=\{\sum_{\tau\in\Sigma_2} a_\tau e_\tau| \sum a_\tau l^{r+1}_\tau=0\}.
\]
Then $K^r\subset V^r$ and $H_1(\cJ[\Sigma])\cong V^r/K^r$ as $S$-modules.
\end{lem}
\begin{proof}
The proof is similar to the proof of ~\cite[Lemma~3.8]{LCoho}.  Let $K_v^r\subset \bigoplus_{v\in\tau} S(-r-1)e_\tau$ be the module of relations around the ray $v\in\Sigma_1$, namely
\[
K^r_v=\{\sum_{v\in\tau} a_\tau e_\tau|\sum a_\tau l^{r+1}_\tau=0\}.
\]
Furthermore, let $J(\nu)$ be the ideal of the central vertex of $\Sigma$.  Set up the following diagram with exact rows, whose first row is the complex $J[\Sigma]$.

\begin{tikzcd}
0 & 0 & 0\\
\bigoplus\limits_{\tau\in\Sigma_2} J(\tau) \ar{r}\ar{u} & \bigoplus\limits_{v\in\Sigma_1} J(v) \ar{r}\ar{u} & J(\nu) \ar{u}\\
\bigoplus\limits_{\tau\in\Sigma_2} S(-r-1) \ar{r}\ar{u} & \bigoplus\limits_{\substack{v\in\Sigma_1,\tau\in\Sigma_2\\ v\in\tau}} S(-r-1) \ar{r}\ar{u} & \bigoplus\limits_{\tau\in\Sigma_2} S(-r-1)\ar{u} \\
0\ar{u}\ar{r} & \bigoplus\limits_{v\in\Sigma_1} K^r_v \ar{u}\ar{r}{\iota} & V^r \ar{u}\\
& 0\ar{u} & 0\ar{u} \\
\end{tikzcd}

The middle row is in fact exact because the inclusion on the left hand side has the effect of gluing together copies of $S(-r-1)$ that correspond to different rays in $\Sigma_1$, leaving a copy of $S(-r-1)$ for every codimension one face $\tau\in\Sigma_2$ in the cokernel.  Now the long exact sequence in homology yields the isomorphisms $H_2(\cJ[\Sigma])\cong \mbox{ker}(\iota)$ and $H_1(\cJ[\Sigma])\cong \mbox{coker}(\iota)$.  The image of $\bigoplus\limits_{v\in\Sigma_1} K^r_v$ under $\iota$ is precisely $K^r$, so we are done.
\end{proof}

\begin{cor}\label{cor:deg2}
Let $\Sigma\subset\R^3$ be a complete, generic, simplicial fan.  For uniform smoothness $r=1$, $H_1(\cJ[\Sigma])$ is generated in degree two.
\end{cor}

\begin{proof}
Since $\Sigma$ is generic, we may assume there are at least $6$ codimension one faces, so $J(\nu)\cong \langle x,y,z\rangle^2$.  There is an eight dimensional space of linear syzygies on $J(\nu)$ and no syzygies of higher degree (the free resolution is in fact \textit{linear}).  It follows that $V^1$ consists of an eight dimensional space of linear syzygies (of degree three), and perhaps many more of degree two.  We need only check that these eight linear syzygies live in the submodule $K^1$, i.e. that these eight linear syzygies may be obtained as syzygies around rays.  Since $\Sigma$ is generic, we may assume that $\Sigma_1$ consists of at least $4$ rays $v_1,v_2,v_3,v_4$ whose linear spans are linearly independent.  Hence $J(v_i)\cong \langle x,y\rangle^2$.  An easy check yields that there are two linear syzygies on this ideal.  We claim that the two linear syzygies contributed from each of these four rays form a vector space of linear syzygies of dimension eight, which spans the entire space of linear syzygies on $J(\nu)$.  We can do this explicitly by making a projective change of coordinates so that $v_1$ points along the positive $x$-axis, $v_2$ along the positive $y$-axis, $v_3$ along the positive $z$-axis, and $v_4$ points in the direction of the vector $\langle -1,-1,-1 \rangle$.  Then we have
\[
\begin{array}{rl}
J(v_1)= \langle y^2,yz,z^2 \rangle &
J(v_2)= \langle x^2,xz,z^2\rangle\\
J(v_3)= \langle x^2,xy,y^2\rangle &
J(v_4)= \langle (x-z)^2,(x-z)(y-z),(y-z)^2\rangle.
\end{array}
\]
The claim is that the linear syzygies on $J(v_1),\ldots,J(v_4)$ generate the linear syzygies on $J(v_1)+\cdots +J(v_4)=J(\nu)$.  This is readily checked by hand or in Macaulay2.
\end{proof}

\begin{thm}\label{thm:2Dfan}
Let $\Sigma\subset\R^3$ be a generic hereditary simplicial fan with $\lk(\Sigma)$ simply connected.  Then
\[
\dim C^1(\Sigma)_d=\chi(\cR/\cJ[\Sigma,\partial\Sigma],d).
\]
Equivalently,
\[
H_2(\cR/\cJ[\Sigma,\partial\Sigma])=0.
\]
\end{thm}

\begin{proof}
It is equivalent to prove that $H_1(\cJ[\Sigma,\partial\Sigma])=0$.  We argue by induction on the number of interior vertices.  First assume $\Sigma$ is complete and let $\Sigma'$ be the simplicial fan obtained by removing any three dimensional cone.  Let $\tau_1,\tau_2,\tau_3$ be the codimension one faces and $v_1,v_2,v_3$ the rays of the cone removed.  We have the following commutative diagram with exact columns.

\begin{tikzcd}
 & 0\ar{d} & 0\ar{d} & 0\ar{d} \\
 & \bigoplus\limits_{i=1}^3 J(\tau_i) \ar{d}\ar{r} & \bigoplus\limits_{i=1}^3 J(v_i) \ar{d}\ar{r} & J(\nu) \ar{d}\\
 J[\Sigma] & \bigoplus\limits_{\tau\in\Sigma_1} J(\tau) \ar{r}\ar{d} & \bigoplus\limits_{v\in \Sigma_1} J(v) \ar{r}\ar{d} & J(\nu)\ar{d}\\
 J[\Sigma',\partial\Sigma'] & \bigoplus\limits_{\tau\in\left(\Sigma'\right)^0_1} J(\tau) \ar{r}\ar{d} & \bigoplus\limits_{v\in \left(\Sigma'\right)^0_1} J(v) \ar{r}\ar{d} & 0\\
 & 0 & 0 &
\end{tikzcd}

By induction, $H_1(\cJ[\Sigma',\partial\Sigma'])=0$.  By Corollary~\ref{cor:deg2}, to prove that $H_1(\cJ[\Sigma])=0$ it suffices to prove that $H_1(\cJ[\Sigma])_2=0$.  The ideals across the top row have the form
\[
\begin{array}{cc}
J(\tau_i)\cong & \langle x^2 \rangle\\
J(v_i)\cong & \langle x,y \rangle^2\\
J(\nu)\cong & \langle x,y,z\rangle^2
\end{array}
\]
Hence in degree two we have
\[
\begin{array}{cc}
\bigoplus\limits_{i=1}^3 J(\tau_i)_2 \cong & \R^3\\
\bigoplus\limits_{i=1}^3 J(v_i)_2 \cong & \R^9\\
J(\nu)_2\cong & \R^3.
\end{array}
\]
The leftmost map of the top row is injective in degree two and the rightmost is surjective.  Since the Euler characteristic is zero, we have that the top row is exact in degree two.  The long exact sequence in homology then yields $H_1(\cJ[\Sigma])_2=0$.

For the remainder of the induction we assume $\Sigma$ is not complete and follow the argument of Whiteley~\cite[\S~4]{WhiteleyComb}.  To prove that $H_1(\cJ[\Sigma,\partial\Sigma])=0$, it suffices to show that, in degree $2$,
\[
\dim H_2(\cJ[\Sigma,\partial\Sigma])_2=\chi(\cJ[\Sigma,\partial\Sigma],2).
\]
$\cJ[\Sigma,\partial\Sigma]$ is concentrated in homological degrees one and two, and has the form
\[
\bigoplus_{\tau\in\Sigma^0_2} J(\tau) \xrightarrow{\delta_2} \bigoplus_{v\in\Sigma^0_1} J(v).
\]
The codomain of $\delta_2$ is contained in the free module $\bigoplus_{v\in\Delta^0_1} S$, where $S=\R[x,y,z]$ is the polynomial ring in three variables.  Since $J(\tau)\cong\langle l_\tau^2\rangle\cong S(-2)$, we can identify $\delta_2$ as a graded map between free $S$-modules of the form
\[
\bigoplus_{\tau\in\Sigma^0_2} S(-2) \xrightarrow{\delta_2} \bigoplus_{v\in\Sigma^0_1} S.
\]
Let $e_\tau,\tau\in \Sigma^0_2$ be the generators of $\bigoplus_{\tau\in\Sigma^0_2} S(-2)$, and $e_v,v\in\Sigma^0_1$ be generators of $\bigoplus_{v\in\Sigma^0_1} S$.  Then $\delta_2(e_\tau)=l^2_\tau e_{v_1} \pm l^2_\tau e_{v_2}$, where $e_{v_1},e_{v_2}$ are the interior rays on the boundary of $\tau$ and the signs come from orientations of the codimension one and two faces of $\Sigma$.  Hence $\delta_2$ is given in the chosen free basis by the matrix $N(\Sigma)$ whose columns are labelled by interior codimension one faces $\tau$, whose rows are labelled by interior rays, and whose entries are given by
\[
N_{v,\tau}=\left\lbrace
\begin{array}{ll}
\pm l^2_\tau & \mbox{if } v\in\tau\\
0 & \mbox{otherwise}.
\end{array}
\right.
\]
Let $N_2(\Sigma)$ be the restriction of $N(\Sigma)$ to degree $2$, i.e. $N_2(\Sigma)$ represents the map $\delta_2$ in degree $2$.  For generic $\Sigma$, $J(\tau)\cong \langle x^2 \rangle$ and $J(v)\cong \langle x^2,xy,y^2 \rangle$.  So in degree $2$, 
\[
\begin{array}{rl}
\dim \bigoplus_{\tau\in\Sigma^0_2} J(\tau)_2 & =f^0_2\\
\dim \bigoplus_{v\in\Sigma^0_1} J(v)_2 & =3f^0_1.
\end{array}
\]
Hence, viewed as a map between $\R$-vector spaces, $N_2(\Sigma)$ has $3f^0_1$ rows (three rows corresponding to each interior ray) and $f^0_2$ columns.  We obtain an explicit form for $N_2(\Sigma)$ by choosing a basis $A_v,B_v,C_v$ for forms of degree $2$ vanishing on any $v\in\Sigma^0_1$.  Then replace the entry $l_\tau^2$ in $N(\Sigma)$ by the $3\times 1$ column vector of coefficients expressing $l_\tau^2$ in terms of this basis.  An important observation is that $N_2(\Sigma)$ has entries which depend continuously on the (homogeneous) coordinates of the rays $v$.  Explicitly, if $\tau$ is a codimension one face joining two rays $v_1=\R_+(x_1,y_1,z_1),v_2=\R_+(x_2,y_2,z_2)$, then
\[
l_\tau=\det
\begin{bmatrix}
x_1 & x_2 & x \\
y_1 & y_2 & y\\
z_1 & z_2 & z
\end{bmatrix}.
\]
Let $f^b_i$ denote the number of boundary $i$-faces of $\Sigma$.  We have the relations
\[
\begin{array}{rl}
3f_3= & 2f^0_2+f^b_2\\
f^0_1-f^0_2+f_3 = & 1
\end{array}
\]
Together these yield $3f^0_1-f^0_2=3-f^b_2$.  Since $f^b_2\ge 3$, we have $3f^0_1\le f^0_2$, so to show that $H_1(\cJ[\Sigma,\partial\Sigma])=0$ generically, it suffices to show that $N_2(\Sigma)$ has full rank.  We show there are no dependencies among the rows of $N_2$ for generic $\Sigma$.

First we reduce to the case where $\lk(\Sigma)$ has triangular boundary.  Let $\Sigma'$ be a subfan of any non-complete fan $\Sigma\subset\R^3$, and order the interior rays and codimension one faces of $\Sigma$ so that those which are also interior faces of $\Sigma'$ appear first.  Then $N_2(\Sigma)$ has block form
\[
\begin{bmatrix}
N_2(\Sigma') & 0\\
A & B
\end{bmatrix},
\]
where the upper right block is a block of zeros.  This block of zeros is present because any codimension one face $\tau\in\Sigma^0_2$ which contains a ray $v\in\left(\Sigma'\right)^0_1$ is also an interior codimension one face of $\Sigma'$.  It follows that any relation among the rows of $N_2(\Sigma')$ immediately gives a relation among the rows of $N_2(\Sigma)$.  For an arbitrary non-complete fan $\Sigma'$, it is simple to build a fan $\Sigma$ having $\Sigma'$ as a subfan so that $\lk(\Sigma)$ has triangular boundary.  Since $\Sigma'$ is not complete, take a simplicial cone $\sigma$ so that $\sigma\cap\Sigma'=0$.  $\Sigma'$ is contained in a component $C$ of $\R^3\setminus\partial\sigma$.  Fill in the region $C\setminus\Sigma'$ with simplicial cones.  Together with $\Sigma'$, this creates a fan $\Sigma$ whose boundary is $\partial\sigma$.  Hence it suffices to prove that there are no relations among the rows of $N_2(\Sigma)$ when $\lk(\Sigma)$ has triangular boundary.

If $\lk(\Sigma)$ has triangular boundary, then $3f^0_1=f^0_2$ and $N_2(\Sigma)$ has the same number of rows as columns.  We now argue that the columns of $N_2(\Sigma)$ are independent.  This is the main induction, and it is performed on the number of interior rays.  As the base case, consider the fan $\Sigma$ with a single interior ray, three codimension one interior faces, and three codimension one boundary faces.  By changing coordinates we may assume the interior ray is the $z$-axis and the three interior codimension one faces are given by $x=0,y=0$, and $x-y=0$.  Then we have
\[
N(\Sigma)=
\begin{bmatrix}
x^2 & y^2 & (x-y)^2
\end{bmatrix}.
\]
These form a basis for forms of degree two in $x$ and $y$, hence the columns of $N_2(\Sigma)$ are independent.  

We will be terse in the remainder of the proof, since the argument for~\cite[Theorem~6]{WhiteleyComb} carries over almost verbatim (Whiteley uses the transpose of the matrix $N_2(\wDelta)$ we use here).  For the inductive step, we apply \textit{vertex splitting} and the inverse process of \textit{edge contraction} (edge \textit{shrinking} in~\cite{WhiteleyComb}).  Applied to the fan $\Sigma$ this process is one of splitting the rays $v\in\Sigma_1$ and contracting codimension one faces $\tau\in\Sigma_2$.  The split of an interior ray adds one interior ray, three interior codimension one faces, and two simplicial facets.  This process does not affect $\partial\Sigma$, hence $\lk(\Sigma)$ remains triangular.  Likewise, the reverse process of contracting an interior codimension one face joining two interior rays does not affect $\partial\Sigma$.  Any noncomplete fan $\Sigma\subset\R^3$ with at least two interior rays has a contractible codimension one face joining two interior rays.  This follows by taking a stereographic projection of $\lk(\Sigma)$ with center outside of $\Sigma$, and then applying~\cite[Lemma~5]{WhiteleyComb}, which says that any triangulated disk with at least two interior vertices has a contractible edge joining two interior vertices.  Such an edge has the property that it is not an edge of any non-facial three-cycle.

Now choose a contractible codimension one face $\tau$ of $\Sigma$ joining two interior rays $v_1,v_2$.  In the process of contracting $\tau$, two codimension one faces, call them $\tau_1,\tau_2$, collapse to two corresponding codimension one faces $e_1,e_2$ when $\tau$ is fully contracted.  Let $\Sigma'$ denote the fan obtained by contracting $\tau$, and let $v$ be the vertex which splits to create $v_1,v_2$.  $M_2(\Sigma)$ is obtained from $M_2(\Sigma')$ by replacing the row corresponding to $v$ by two rows corresponding to $v_1,v_2$ adding the column corresponding to $\tau$, and replacing the columns corresponding to $e_1,e_2$ each by two columns corresponding to $e_1,\tau_1$, $e_2,\tau_2$.  Since $M_2(\Sigma)$ has entries which depend continuously on the homogeneous coordinates of the ray $v_1$, we consider $\lim_{v_1\rightarrow v_0} M_2(\Sigma)$.  Whiteley shows that a nontrivial relation among the columns of $\lim_{v_1\rightarrow v_0} M_2(\Sigma)$  implies a relation among the columns of $M_2(\Sigma')$.  By induction, we assume $M_2(\Sigma')$ has independent columns, so no such relation exists.  Again, arguing by the continuous dependence of $\det M_2(\Sigma)$ on the homogeneous coordinates of $v_1$, most choices of coordinate $v_1$ with $v_1$ close to $v_0$ and $v_1\neq v_0$ yield $\det M_2(\Sigma)\neq 0$.  Hence the columns of $M_2(\Sigma)$ are linearly independent for generic choices of $\Sigma$.
\end{proof}

\begin{proof}[Proof of Theorem~\ref{thm:GenericC1}]
Set $\Sigma=\wDelta$, $S=\R[w,x,y,z]$.  Since we are considering uniform smoothness, $\Sigma^{-1}=\partial\Sigma$. $\cR/\cJ[\Sigma,\partial\Sigma]$ is concentrated in homological degrees $4,3,2,$ and $1$ and has the form
\begin{equation}\label{eqn:C1complex}
S^{f_4(\Sigma)}\rightarrow \bigoplus\limits_{\tau\in\Sigma^0_3} \dfrac{S}{J(\tau)} \rightarrow \bigoplus\limits_{\gamma\in\Sigma^0_2} \dfrac{S}{J(\gamma)} \rightarrow \bigoplus\limits_{v\in \Sigma^0_1} \dfrac{S}{J(v)}.
\end{equation}
We show first that $\chi_H(\cR/\cJ[\Sigma,\partial\Sigma],d)=f_3\binom{d+3}{3}-f^0_2(d+1)^2+f^0_1(3d+1)-4f^0_0$.  Since we assume $\Delta$ is generic, there are at least three interior codimension one faces meet along every edge $\gamma\in\Delta^0_1$ and at least $6$ codimension one faces meet at every vertex $v\in\Delta^0_0$.  Under these assumptions, for $\tau\in\Delta^0_2,\gamma\in\Delta^0_1,v\in\Delta^0_0$ we have
\[
\begin{array}{rl}
J(\tau)\cong & \langle x \rangle^2\\
J(\gamma)\cong & \langle x,y \rangle^2\\
J(v)\cong & \langle x,y,z \rangle^2.
\end{array}
\]
Given these equalities, it follows easily that $HP(S/J(\tau),d)=(d+1)^2$, $HP(S/J(\tau),d)=3d+1$, and $HP(S/J(v),d)=4$.  Also, $f_i(\Sigma)=f_{i-1}(\Delta)$ and $f^0_i(\Sigma)=f^0_{i-1}(\Delta)$.  Taking an alternating sum and using the well-known fact that $HP(S,d)=\binom{d+3}{3}$ gives $\chi_H(\cR/\cJ[\Sigma,\partial\Sigma],d)=f_3\binom{d+3}{3}-f^0_2(d+1)^2+f^0_1(3d+1)-4f^0_0$.

To complete the proof, it suffices by Corollary~\ref{cor:3DSimplicialHilbert} to show that
\[
H_2(\cR/\cJ[\Delta_v,\partial\Delta_v])=0
\]
for all $v\in\Delta_0$.  Equivalently we need to show that
\[
\dim C^1(\Delta_v)_d=\chi(\cR/\cJ[\Delta_v,\partial\Delta_v],d),
\]
for all $d\ge 0$ and all $v\in\Delta^0_0$.  This follows from Theorem~\ref{thm:2Dfan}.
\end{proof}

\section{Acknowledgements}

I thank Hal Schenck for his guidance and patient listening, and Jimmy Shan for helpful conversations.  Macaulay2~\cite{M2} was indispensable for performing the computations in this paper.  Three dimensional images were generated using Mathematica.

\end{document}